\documentclass[11pt]{article}
\usepackage[utf8]{inputenc}
\usepackage{tikz}
\usetikzlibrary{shapes,arrows,positioning}
\usetikzlibrary{automata,arrows,positioning,calc}
\usepackage{amsmath}
\usepackage{amssymb}
\usepackage{amsthm}
\usepackage{hyperref}
\usepackage{bbm}
\usepackage{blkarray}
\usepackage{algorithmic}
\usepackage{caption}
\usepackage{subcaption}
\usepackage{float}
\usepackage{stmaryrd}

\usepackage{fancyhdr}
\usepackage{amsopn}

\usepackage{bbm}

\usepackage{booktabs}
\renewcommand{\arraystretch}{1.2} 
\newcommand{\ra}[1]{\renewcommand{\arraystretch}{#1}}

\hypersetup{colorlinks,
citecolor=orange,
citebordercolor=magenta,
linkcolor=magenta
}

\providecommand{\keywords}[1]
{
  \small	
  \textbf{{Keywords.}} #1
}

\providecommand{\AMS}[1]
{
  \small	
  \textbf{{AMS subject classifications.}} #1
}

\usepackage[linesnumbered,ruled,vlined]{algorithm2e}

\numberwithin{equation}{section}

\usepackage{todonotes}

\newtheorem{theorem}{Theorem}[section]
\newtheorem{remark}[theorem]{Remark}
\newtheorem{definition}[theorem]{Definition}
\newtheorem{lemma}[theorem]{Lemma}
\newtheorem{proposition}[theorem]{Proposition}

\newtheorem{hypothesis}[theorem]{Hypothesis}

\DeclareMathOperator*{\argmin}{arg\,min}

\title{Optimal incentives to mitigate epidemics: \\ A Stackelberg mean field game approach}

\author{Alexander Aurell\footnote{Department of Operations Research and Financial Engineering,
  Princeton University, 
  Princeton, NJ 08544 
  (\href{mailto:aaurell@princeton.edu}{aaurell@princeton.edu},
  \href{mailtorcarmona@princeton.edu}{rcarmona@princeton.edu},
  \href{mailto:gokced@princeton.edu}{gokced@princeton.edu},
  \href{mailto:lauriere@princeton.edu}{lauriere@princeton.edu}).}
\and Ren\'e Carmona
    \footnotemark[1]
\and G\"ok\c ce Dayan{\i}kl{\i}
    \footnotemark[1]
\and Mathieu Lauri\`ere
    \footnotemark[1]
}
\date{}

\begin{document}

\maketitle
\begin{abstract}
    Motivated by the models of epidemic control in large populations, we consider a Stackelberg mean field game model between a principal and a mean field of agents whose states evolve in a finite state space. The agents play a non-cooperative game in which they control their rates of transition between states to minimize an individual cost. The principal influences the nature of the resulting Nash equilibrium through incentives so as to optimize its own objective. We analyze this game using a probabilistic approach. We then propose an application to an epidemic model of SIR type in which the agents control the intensities of their interactions, and the principal is a regulator acting with non pharmaceutical interventions. To compute the solutions, we propose an innovative numerical approach based on Monte Carlo simulations and machine learning tools for stochastic optimization. We conclude with numerical experiments illustrating the impact of the agents' and the regulator's optimal decisions in two specific models: a basic SIR model with semi-explicit solutions and a more complex model with a larger state space. 
\end{abstract}

\vskip3mm

\vskip3mm
\AMS{  92D30,  
  49N90,   
  91A13, 
91A15, 
62M45.  	
}

\vskip 12pt\noindent
\emph{\textbf{Acknowledgments.}}
{This work was done with the support of NSF
DMS-1716673, ARO W911NF-17-1-0578, and AFOSR \# FA9550-19-1-0291.}

\section{Introduction}

Non pharmaceutical interventions such as the reduction of social interactions are powerful measures to limit the spread of an ongoing epidemic. Containment and suppression of disease spread are crucial factors in order to avoid overwhelming the health care system. However, even in the midst of  pandemics, some individuals still refuse to comply with guidelines such as social distancing or mask wearing. From a global perspective, this could push the equilibrium behavior of the population to exceed the limits of the health care system. For this reason, responsible authorities have a keen interest in the design of incentive systems that are acceptable to individuals and sufficiently strong to induce them to successfully combat the epidemic.

In a mathematical model, we would like the decision maker to take both the global state of the society and the individuals' behaviors into account before deciding on an incentive policy. The intractability of large interacting dynamical systems usually prevents that kind of analysis. In this work, we analyze a Stackelberg game between a principal agent representing the regulatory authority, and a field of individuals providing the societal response to the principal's policy. we use the probabilistic approach to mean field games because it provides both a macroscopic description of the state of the population and a microscopic analysis of the behavior of a representative single individual.

Mean Field Game (MFG) models study the equilibrium between a representative player and the distribution of the other players' states and actions. Mean field equilibria are simpler to identify and compute than equilibria of large populations. Moreover, they provide approximate Nash equilibria for certain games with a large but finite number of players. The framework has found numerous applications, from the analysis of growth models in macro-economics, to crowd motion and energy production. Here, we investigate an application to epidemic control.

Compartmental models in epidemic research are in many cases large population limits of interacting Markov chains. Each Markov chain represents the individual's state of health, the transitions between states occurring with rates depending on the global state of the population, the proportions of individuals in different states to be more specific. This form of interaction is clearly in the purview of mean field models. Incorporating in the model the opportunity for individuals to choose their behaviors and control their contributions to the spread of the disease, the interacting system can be analyzed as a MFG. However, a natural choice of controls yields what is often called an extended MFG, where players interact not only through the distribution of their states, but instead through the joint distribution of their states and actions. Next, we give a hands-on example of the extended aspect of our model.

\subsection{The SIR extended MFG with contact factor control}
\label{sec:intro-SIRMFG}

 In order to provide a motivating example, we consider the simplest compartmental model in epidemics, the classical SIR model. First, we outline how to construct it as a large population limit. Secondly, we comment on why an extended MFG formulation is relevant for the large population limit problem if players use what we will call a contact factor control to reduce the risk of disease spread. 

Before moving to the example, we need to introduce some notation. Consider $N$ individuals, each of whom transitions between the states Susceptible ($S$), Infected ($I$), and Removed ($R$). An individual in state $R$ has either gained permanent immunity, or is deceased. Denote the state of individual $j\in\{1,\dots, N\}$ at time $t$ by $X^j_t$, and let $p^N_t = (p^N_t(S),p^N_t(I), p^N_t(R)) := (\frac{1}{N}\sum_{j=1}^N \mathbbm{1}_{i}(X^j_t))_{i\in\{S,I,R\}}$ be the vector of proportions of individuals in each state, in other words, the empirical distribution  of the state at time $t$. 
 
A susceptible individual might meet infected individuals, possibly resulting in disease transmission. Encounters occur pairwise and randomly throughout the population. Their intensity is denoted by $\beta>0$. The number of encounters with infected individuals during a small time interval $[t-\Delta t,t)$ is proportional to the the proportion of the population in state $I$ at $t$. Hence the transition from state $S$ to $I$ happens with intensity $\beta p^N_t(I)$.
    Upon infection an individual starts the path to recovery. The transition from state $I$ to $R$ happens after an exponentially distribution time with rate $\gamma$.
    The state $R$ is absorbing. To summarize, the transition rate matrix, which is common to all agents of the population, is at time $t$,
\begin{equation}
    Q(p^N_t) = 
    \begin{bmatrix}
        -\beta p^N_t(I) & \beta p^N_t(I) & 0
        \\
        0 & -\gamma & \gamma
        \\
        0 & 0 & 0
    \end{bmatrix}.
\end{equation}
As $N\rightarrow \infty$, $(p^N_t)_t$ converges in probability to the unique solution to
\begin{equation}
\label{eq:SIR1}
    \dot{p}_t = p_tQ(p_t),\quad p_0 = p^0,
\end{equation}
if the initial configuration is sampled from a symmetric probability measure with marginals equal to $p^0$, a classical result found for example in \cite{kurtz1981approximation}. 
Scaling $p_t$ by a population size $N$, \textit{i.e.}, letting $Np_t =: (S(t), I(t), R(t))$.
we retrieve the standard formulation of the SIR model,
\begin{equation}
\left\{
    \begin{aligned}
    \dot{S}(t) &= -\frac{\beta}{N}I(t)S(t),& S(0) &= Np^0(S),
    \\
    \dot{I}(t) &= \frac{\beta}{N}I(t)S(t) - \gamma I(t),& I(0) &= Np^0(I),
    \\
    \dot{R}(t) &= \rho\gamma I(t),& R(0) &= Np^0(R).
    \end{aligned}
\right.
\end{equation}

Now, let us assume that each individual has the option to control the intensity they seek or try to avoid interacts with others. Instances occur when individuals try to lower the risk of disease transmission by , \textit{e.g.}, avoiding to ride public transportation at congested hours, shopping online, or wearing protective equipment. 

The probability of the spread of the disease is likely to be a non-linear function of the joint effort of the individuals that interact. Say, for example, that two individuals meet and both have the option to wear a protective face mask. The absolute decrease in risk of transmission is not necessarily equal for each additional mask that is worn. Motivated by this observation, we assume that the individuals' efforts to reduce spread affect the probability of transmission in a multiplicative way: in each encounter the probability of disease spread is scaled by each of the agents effort. We view an individual's effort to meet someone as their control. We often call it their \emph{contact factor} because this effort enters as a factor in the contact rate between individuals of specific states.

Assuming that the meeting frequency is $\beta$, that the pairing is random, that the disease spreads from infected agents to susceptible, 
and that the spread probability is scaled by the effort intensity of the search for meetings, the transition rate for individual $j$, currently susceptible, to the state of infected is
\begin{equation}
    \beta \alpha^j_t \frac{1}{N}\sum_{k=1}^N \alpha^k_t 1_{I}(X^k_{t-}),
\end{equation}
where $\alpha^k_t$ denotes the (contact factor) action of individual $k\in\{1,\dots, N\}$ at time $t$, selected from set $A$ of admissible actions. Along the lines of the heuristics of MFG theory, we anticipate that in an appropriate approximation of our interacting system in the limit $N\rightarrow \infty$, the representative agent transitions from susceptible to infected with rate
\begin{equation}
    \beta \alpha_{t} \int_A a \rho_t(da,I),
\end{equation}
where $\rho_t$ is the joint distribution of action and state of the representative agent in a suitable probability space (rigorously defined in the next section). The joint action-state distribution is often referred to as the extended mean field in the MFG literature.
To summarize, the representative agent transitions between states $S$, $I$, and $R$ according to the rate matrix $Q(t,\alpha_t,\rho_t)$,
\begin{equation}
\label{eq:SIR-Qmatrix-intro}
    Q(t,\alpha,\rho)
    =
    \begin{bmatrix}
    \cdots & \beta\alpha_t\int_A a\rho_t(da,I) & 0
    \\
    0 & \cdots & \gamma
    \\
    \eta  & 0 & \cdots
    \end{bmatrix},
\end{equation}
where $\beta,\gamma,\eta  \in \mathbb{R}_+$ are non-controlled constants, and as usual, the diagonal terms $\cdots$ should be replaced by the negative of the sum of the entries in the same row. See Fig.~\ref{fig:SIR-diagram-intro} for a diagram of the dynamics. 

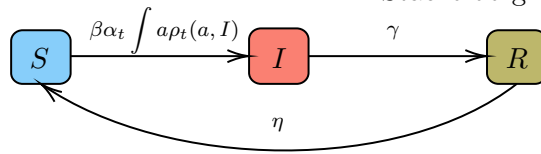
\begin{figure}
\vskip-5mm
\begin{center}

\tikzset{every picture/.style={line width=0.75pt}} 

\begin{tikzpicture}[x=0.75pt,y=0.7pt,yscale=-1,xscale=1]

\draw  [fill={rgb, 255:red, 135; green, 206; blue, 250 }  ,fill opacity=1 ] (100,106.8) .. controls (100,103.6) and (102.6,101) .. (105.8,101) -- (124.2,101) .. controls (127.4,101) and (130,103.6) .. (130,106.8) -- (130,124.2) .. controls (130,127.4) and (127.4,130) .. (124.2,130) -- (105.8,130) .. controls (102.6,130) and (100,127.4) .. (100,124.2) -- cycle ;
\draw  [fill={rgb, 255:red, 250; green, 128; blue, 114 }  ,fill opacity=1 ] (220,106) .. controls (220,102.69) and (222.69,100) .. (226,100) -- (244,100) .. controls (247.31,100) and (250,102.69) .. (250,106) -- (250,124) .. controls (250,127.31) and (247.31,130) .. (244,130) -- (226,130) .. controls (222.69,130) and (220,127.31) .. (220,124) -- cycle ;
\draw  [fill={rgb, 255:red, 189; green, 183; blue, 107 }  ,fill opacity=1 ] (340,106.8) .. controls (340,103.6) and (342.6,101) .. (345.8,101) -- (364.2,101) .. controls (367.4,101) and (370,103.6) .. (370,106.8) -- (370,124.2) .. controls (370,127.4) and (367.4,130) .. (364.2,130) -- (345.8,130) .. controls (342.6,130) and (340,127.4) .. (340,124.2) -- cycle ;
\draw    (130,115) -- (218,115) ;
\draw [shift={(220,115)}, rotate = 180] [color={rgb, 255:red, 0; green, 0; blue, 0 }  ][line width=0.75]    (10.93,-3.29) .. controls (6.95,-1.4) and (3.31,-0.3) .. (0,0) .. controls (3.31,0.3) and (6.95,1.4) .. (10.93,3.29)   ;
\draw    (250,115) -- (320.5,115) -- (338,115) ;
\draw [shift={(340,115)}, rotate = 180] [color={rgb, 255:red, 0; green, 0; blue, 0 }  ][line width=0.75]    (10.93,-3.29) .. controls (6.95,-1.4) and (3.31,-0.3) .. (0,0) .. controls (3.31,0.3) and (6.95,1.4) .. (10.93,3.29)   ;
\draw    (355,130) .. controls (296.79,178.76) and (178.69,177.02) .. (115.94,130.7) ;
\draw [shift={(115,130)}, rotate = 396.94] [color={rgb, 255:red, 0; green, 0; blue, 0 }  ][line width=0.75]    (10.93,-3.29) .. controls (6.95,-1.4) and (3.31,-0.3) .. (0,0) .. controls (3.31,0.3) and (6.95,1.4) .. (10.93,3.29)   ;

\draw (108,109) node [anchor=north west][inner sep=0.75pt]   [align=left] {$\displaystyle S$};
\draw (230,109) node [anchor=north west][inner sep=0.75pt]   [align=left] {$\displaystyle I$};
\draw (348,109) node [anchor=north west][inner sep=0.75pt]   [align=left] {$\displaystyle R$};
\draw (138,87) node [anchor=north west][inner sep=0.75pt]  [font=\scriptsize] [align=left] {$\displaystyle \beta \alpha _{t}\int a\rho _{t}( a,I)$};
\draw (288,97) node [anchor=north west][inner sep=0.75pt]  [font=\scriptsize] [align=left] {$\displaystyle \gamma $};
\draw (230,147) node [anchor=north west][inner sep=0.75pt]  [font=\scriptsize] [align=left] {$\displaystyle \eta$};

\end{tikzpicture}

\end{center}
\caption{SIR model with extended mean-field interactions corresponding to the $Q$-matrix~\eqref{eq:SIR-Qmatrix-intro}.}
\label{fig:SIR-diagram-intro}
\vskip-7mm
\end{figure}

The representative agent is incentivized by a regulator to choose their contact factor close to a level determined by the regulator. The level and incentive can vary between the susceptible, infected, and recovered parts of the population, as the state of an agent naturally influences their contribution to the overall societal risks in an epidemic (which the regulator aims to mitigate). Moreover, the representative agent faces a cost of inconvenience for being sick. 
In particular, let us consider a model where a representative agent pays per unit of time a running cost given as
\begin{equation}
\label{eq:cost-agents-SIR1-intro}
\begin{aligned}
    \frac{c_\lambda}{2}\left(\lambda^{(S)}_t - \alpha_t\right)^2\mathbbm{1}_{S}(x) + \left(\frac{1}{2}\left(\lambda^{(I)}_t - \alpha_t\right)^2 + c_I\right)\mathbbm{1}_{I}(x) 
    + \frac{1}{2}\left(\lambda^{(R)}_t-\alpha_t\right)^2\mathbbm{1}_{R}(x),
\end{aligned}
\end{equation}
where $c_\lambda,c_I \in\mathbb{R}_+$ are constants, $\boldsymbol\lambda^{(\cdot)}$ are the socialization levels recommended by the regulator, and $\alpha_t$ is the contact factor of the representative player at time $t$.\footnote{The running cost \eqref{eq:cost-agents-SIR1-intro} is far from the only possible model choice within the framework presented in this paper, but one that facilitates evaluation of the performance of the proposed numerical method. In a more general setting, the running cost could be a function depending also on the extended mean field interactions.}
In this model, which we will revisit in the section on numerical experiments, a $\boldsymbol\lambda^{(S)}$ valued close to $0$ can be interpreted as a recommendation for low levels of social interactions (or more generally a high level of cautiousness through non-pharmaceutical interventions, such as hand cleaning, lockdown, mandatory mask wearing etc.) for all susceptible individuals. Conversely, $\boldsymbol\lambda^{(S)}$ approximately equal to $1$ amounts to recommending the regular level of social interactions (i.e. no restriction). 
Finally, we assume that the agent also receives a terminal utility $U(\xi)$ depending on the agent's behavior during the time interval $[0,T]$, payed by the regulator as an incentive to follow the recommended socialization protocols. 

 The problem is then two-fold. First, the regulator announces a policy $(\boldsymbol\lambda,\xi)$ so as to minimize an objective function which involves the state of the population (\textit{e.g.}, the proportion of infected people). The minimization is constrained, not all policies will be accepted by the population and the regulator is optimizing only over acceptable ones. Acceptable here does not mean a complete commitment to following the recommendations, but rather accepting the penalty structure for deviations. The way the population reacts to a given policy $(\boldsymbol\lambda,\xi)$ is through a Nash equilibrium in which each agent tries to optimize their own individual cost induced by~\eqref{eq:cost-agents-SIR1-intro} and their terminal payment utility $U(\xi)$. If the representative agent's expected cost at the Nash equilibrium is above some threshold, the policy is deemed unacceptable and the population rejects the policy altogether (the policy is not feasible). The regulator, to find an optimal feasible policy, needs to understand how the population reacts to each policy. The two nested problems comprise a so-called Stackelberg game: the regulator's problem is a constrained optimization problem driven by the Nash equilibrium of the player population. 

The two components $\boldsymbol\lambda$ and $\xi$ of the regulator's policy play different roles. The process $\boldsymbol\lambda$ is used to incentivize the agents to adopt a certain cautiousness level over time. Practically, deviations from $\boldsymbol\lambda$ are penalized. The terminal payment $\xi$ is used to reward participation in the incentive structure. Later, we will see how the regulator decides on a terminal reward or payment such that the representative agent does not reject the proposed incentive scheme altogether. That is, the terminal payment's role is to make the policy feasible in the sense that it is a reasonable compensation to the agent for complying with the cautiousness level recommendations.

\subsection{Related literature}

\subsubsection{Discrete state space MFG} 
The behavior of the population of agents, amongst whom disease spread takes place, is in this paper modeled by a discrete space MFG. MFGs were first developed for continuous state space \cite{lasry2006jeux,lasry2006jeux2,huang2006large}. Soon after works on discrete state spaces  followed \cite{gomes2010discrete, kolokoltsov2012nonlinear, Gomes2013}. Amongst the many contributions to the field of discrete state MFGs we note the minor-major player model \cite{Carmona2016}, the probabilistic approach \cite{Cecchin2018}, the master equation approach~\cite{bayraktar2018analysis}, and the extended game \cite{Carmona2018Extended}. Mean-field optimal control, risk-sensitive control, and zero-sum games are treated in \cite{choutri2018stochastic,choutri2019mean,choutri2019optimal} which cover cases of unbounded jump intensities.

\subsubsection{Compartmental models and MFG in epidemics}
Games and optimal control in compartmental models have been studied intensively for a long time. This literature review focuses on other work within the mean-field approach, which has gained attention is the last decade. Efforts to model the control of disease spread range from strategies for social contacts to vaccination. 

Our work falls within a category of models where agents attempt to suppress the risk of disease spread. In \cite{Elie2020} a deterministic mean-field game is studied where the agents control the contact rate, which is proportional to the risk of disease spread. The agents are penalized if they get infected prior to some terminal time horizon, which introduces a stopping time component to the game similar to that in evacuation problems. A contact rate common to all agents and all states is found such that the agents are in an MFG equilibrium with the crowd.
In \cite{Hubert2020} 
a Stackelberg game where the epidemic evolves in the population of agents, modeled as a MFG, according to a compartmental model is considered. The agents collaborate to find the best contact rate to suppress the epidemic. The compartmental models considered in the paper are stochastic and the uncertainty in the model is controlled by the principal through testing policies. Their goal is to mitigate the saturation of intensive care units. This problem was studied from the point of view of optimal control in~\cite{charpentier2020covid}, where numerical results show that it is optimal to isolate infected individuals so as to maintain a basic reproduction rate close to $1$. In \cite{Cho2020}, SIR and SEIR\footnote{The SEIR model includes the additional ``Exposed" state, modeling the incubation period before the agent transitions to the ``Infected" state.} models where agents control the contact rate are studied and the author compares the MFG equilibrium, the socially optimal strategy, and unconstrained disease spread.

Vaccination is a powerful tool when available. However, it is not considered in this paper. With a MFG formulation of the SIR model, In 
vaccination strategies in a society of non-cooperative individuals are studied. The authors extend the model to include limited vaccination capacity \cite{laguzet2016equilibrium}, limited persistence \cite{salvarani2016individual}, and vital dynamics \cite{hubert2018nash}.\footnote{Persistence here refers to immunity to reinfection and when this is limited the agents will eventually become vulnerable again. The SIR model with vital dynamics includes births and deaths.} 
Vaccination has also been studied with MFG-based SIR models in \cite{doncel2017mean,gaujal20vaccination}, their focus being the loss of efficiency in the mean-field game compared to optimal vaccination policies.

Spatial distance naturally mitigates the risks of the pandemic and \cite{Tembine2020} uses a mean-field type game to take the spatial features of disease spread into account (and many more features, \textit{e.g.}, physical and social status of the agent). In \cite{lee2020}, the authors consider three crowds, each corresponding to a state in the SIR model, which evolve spatially. The pandemic risks are mitigated by a central planner who controls of spatial velocity of the agents. The multi-population mean-field optimal control problem is studied.

\subsubsection{Contract theory and Stackelberg MFG} 

Contract theory studies the interaction of a principal and an agent, where the former proposes a contract to the latter, who decides whether or not they should work for the principal and receive a reward. The principal tries to anticipate the decision of the agent and to design an attractive contract while still trying to maximize their profit. Solutions to this type of problems are typically studied using the concept of Stackelberg equilibrium. In~\cite{holmstrom1987aggregation}, continuous time method is used to study this type of problems, and in~\cite{Sannikov2008,Sannikov2013} dynamic programming and martingale optimality principles are used to characterize the solution in the framework of optimal control theory. These ideas are generalized in~\cite{Cvitanic2018}. In~\cite{djehiche2014principal} the solution in a general class of principal-agent problem is characterized by the stochastic maximum principle.

In the context of MFGs, problems with a principal and a mean-field of agents have been studied in~\cite{elie2019tale} in the continuous state space setting and in~\cite{Carmona2018Contract} for finite state spaces. The theory has been extended in several directions, including problems with delayed information~\cite{MR3376121}. This type Stackelberg mean field models have found applications for instance to advertising~\cite{salhab2018dynamic}, where the principal plays the role of the advertiser and the population of agents decides whether they want to buy a product. Stackelberg equilibria with a mean-field of agents have also been applied in the context of epidemic containment: the aforementioned~\cite{Carmona2018Contract} proposes an application with two cities where the agents can move between cities and the principal can influence the quality of healthcare, while in~\cite{Hubert2020} the authors consider a model where the principal can choose a tax policy and a testing policy. 

\subsection{Contributions and paper structure}

The scientific contribution of this work is two-fold. Firstly, we move beyond current theory and consider a Stackelberg game between a principal and an extended MFG. A common assumption in extended MFGs is that any dependency on the joint distribution of action and state only involves dependencies on the marginal distributions. We avoid this assumption in order to capture the epidemiological aspects outlined in Section~\ref{sec:intro-SIRMFG}. The trade-off is that, at some points in the paper, we need to make assumptions about the existence and uniqueness of mean field Nash equilibria. We work in the weak probabilistic formulation of the problem and we allow the player's action to depend on their state. Our numerical experiments show that contact factors do differ between the compartments of the population. To the best of our knowledge, compartmental models with applications towards epidemics which incorporate at the same time a non-cooperative population and a regulator have not yet been suggested or studied in the literature.

Secondly, we propose an innovative numerical scheme based on neural networks, and validate its performance on simple examples for which we can derive semi-explicit solutions. To obtain a problem amenable to numerical treatment by optimization procedures, we first rewrite the principal's problem under the constraint of the mean field Nash equilibrium as an optimal control problem with two forward stochastic equations. Then, the numerical scheme relies on the approximation of the population by an interacting particle system and the approximation of the controls by neural networks, including the principal's policy. The optimization of the principal's cost is then performed using a variant of stochastic gradient descent to update the neural networks' parameters. 

We carry out multiple numerical experiments studying model characteristics and policy impact. 
As a first step towards understanding how a regulator should design containment policies, we test the population's reaction to various policies. We show that when the agents minimize their own cost and behave as in a Nash equilibrium, they adopt some level of cautiousness, which reduces the severity of the epidemic compared to an unconstrained  \emph{free spread} scenario. 
Moreover, we show that taking early action (\textit{e.g.}, deciding on an early lockdown) has a bigger impact on disease spread than a strategic action taken later. 
In a second numerical test, we solve the full Stackelberg game problem (where the regulator optimizes over the policies to minimize its own cost) for both the SIR-based example of Section~\ref{sec:intro-SIRMFG} and an extended model 
with two more additional states for the agents ($E$: Exposed and $D$: Deceased). For the latter, we show that if agents are not feeling safe enough, they are able to rationally choose lower contact levels than the recommended levels by the regulator. 

The rest of the paper is structured as follows. In Section~\ref{sec:main} the Stackelberg game between a principal and a non-cooperative population is introduced and analyzed. In Section~\ref{sec:numerics} the details of the numerical approach are presented. Finally, Section~\ref{subsec:experiments} contains the evaluation of the numerical method and further simulations. All proofs have been postponed to appendices.

\section{The model}

\label{sec:main}
\subsection{Preliminaries}
We adopt the following notation throughout the paper: $m$ is a finite integer corresponding to the number of states, $E := \{e_1,\dots, e_m\}$ is a state space where $e_i\in \mathbb{R}^m$ is the basis vector in direction $i$, $A := [0,1]$ is an action space, and $\mathcal{R} := \mathcal{P}(A\times E)$ is the set of Borel probability measures on $A\times E$. We endow $A$ with the Euclidean metric $|\cdot |$, $E$ with a bounded discrete metric, and $A\times E$ with the $1$-product metric. We will identify the set $\mathcal{P}(E)$ with the $m$-dimensional simplex and use the Euclidean metric $\|\cdot \|$ to measure distances on $\mathcal{P}(E)$ (the choice of metric on $\mathcal{P}(E)$ is irrelevant since all metrics derived from norms on $\mathcal{P}(E)$ are equivalent).
We endow $\mathcal{R}$ with the $1$-Wasserstein metric $W_{\mathcal{R}}$ which is well-defined on $\mathcal{R}$ since $A\times E$ is compact. 

Let $T>0$ be a constant corresponding to a finite time horizon. Let $\Lambda$ be the set of measurable $\mathbb{R}^m_+$-valued functions with domain $[0,T]$ and let $M(\mathcal{R})$ and $M(\mathcal{P}(E))$ be the set of measurable mappings from $[0,T]$ to $\mathcal{R}$ and to $\mathcal{P}(E)$, respectively. Let $Q : [0,T] \times A \times \mathcal{R} \mapsto \mathbb{R}^{m \times m}$ be a bounded measurable function such that $Q(t,a,\rho)$ is a transition rate matrix, also called $Q$-matrix,\footnote{That is, $q(t,i,j,a,\rho)\geq 0$ for all $1\leq i,j,\leq m$ and $\sum_{j\neq i}q(t,i,j,a,\rho) = -q(t,i,i,a,\rho)$, where $q(t,i,j,a,\rho)$ is the element element at row $i$ and column $j$ of $Q(t,a,\rho)$.} for all $(t,a,\rho)\in [0,T]\times A\times \mathcal{R}$.

A process $(X_t)_{t\in[0,T]}$ will in short-hand be denoted $\boldsymbol X$. Let $\Omega$ be the space of c\`adl\`ag functions $\omega: [0,T]\rightarrow E$ and from now on let $\boldsymbol X$ be the canonical process, $X_t(\omega) = \omega(t)$. Denote by $\mathbb{F} := (\mathcal{F}_t)_{t\in[0,T]}$ the natural filtration generated by $\boldsymbol X$, with $\mathcal{F}_t := \sigma( \{X_s, s\leq t\})$ and $\mathcal{F} := \mathcal{F}_T$,  and by $\mathbb{A}$ the collection of $\mathbb{F}$-predictable processes $\boldsymbol \alpha$ with values in $A$. For any probability measure $\mathbb{Q}$ on $(\Omega, \mathcal{F})$ we denote by $\mathbb{E}^{\mathbb{Q}}$ expectation under $\mathbb{Q}$.

On $(\Omega, \mathbb{F}, \mathcal{F})$ we consider the probability measure $\mathbb{P}$ under which the law of $X_0$ is $p^0\in \mathcal{P}(E)$ and $\boldsymbol X$ is a continuous time Markov chain with transition rate from $e_i$ to $e_j$ equal to $1$ if $(i,j)\in G \subset \{1,\dots,m\}^2$, otherwise zero. Here $G$ represents a graph of states on which a typical agent evolves. Denote the corresponding $Q$-matrix by $Q^0$. We let, for $i=1,\dots,m$,
\begin{equation}
\psi(e_i) := \text{diag}(Q^0 e_i) - Q^0\text{diag}(e_i) - \text{diag}(e_i)Q^0,\quad t\in[0,T],
\end{equation}
and let $\psi_t = \psi(X_{t-})$. Denote by $\psi_t^+$ the Moore-Penrose generalized inverse of the matrix $\psi_t$.
Expectation under $\mathbb{P}$ is abbreviated to $\mathbb{E}$.

We denote by $\mathcal{H}^2$ the set of $\mathbb{F}$-adapted and real-valued c\`adl\`ag processes $\boldsymbol Y$ such that $\mathbb{E}[\int_0^T Y_t^2 dt] < + \infty$ and by $\mathcal{H}^2_X$ the set of $\mathbb{F}$-adapted and $\mathbb{R}^m$-valued left-continuous processes $\boldsymbol Z$ such that $\mathbb{E}[\int_0^T \|Z_t\|^2_{X_{t-}} dt] < +\infty$. The seminorms $\| \cdot \|_{e_i}$, $i=1,\dots,m$,  and the stochastic seminorm $\| \cdot \|_{X_{t-}}$ are defined by 
\begin{equation}
    \|z\|^2_{e_i} := z^* \psi(e_i) z,
    \qquad
    \|z\|^2_{X_{t-}} := z^* \psi_t z,\quad z\in\mathbb{R}^m.
\end{equation}
Hereinafter we use a superscript $*$ to denote the transpose of a vector or a matrix.

\subsection{The Stackelberg extended MFG in a general setting}
\label{sec:general-analysis}

We consider a society made up of a population of non-cooperative players and one principal agent. We begin by focusing on the game between the members of the population. As is common in the MFG paradigm, a representative player takes the role of any individual in the population. Given knowledge of how the population and the principal agent act over time, the representative player optimizes their cost functional. 

To use strategy $\boldsymbol\alpha = (\alpha_t)_{t\in[0,T]} \in \mathbb{A}$ the representative player pays the expected total cost
\begin{equation}
\label{eq:cost_of_minors}
    J^{\boldsymbol\lambda,\xi}(\boldsymbol\alpha, \boldsymbol\rho)
    := 
    \mathbb{E}^{\mathbb{Q}^{\boldsymbol\alpha,\boldsymbol\rho}}\left[
    \int_0^T f(t, X_{t}, \alpha_t, \rho_t; \lambda_t)dt - U(\xi)
    \right],
\end{equation}
where $(\boldsymbol{\lambda},\xi)$ is the principal's policy choice, $f : [0,T]\times E\times A\times \mathcal{R} \rightarrow \mathbb{R}$ is a running cost which depends on the policy $\boldsymbol \lambda$,
$\boldsymbol{\rho} = (\rho_t)_{t \in [0,T]}\in M(\mathcal{R})$ is a flow of measures in $\mathcal{R}$ representing the joint state-control distribution in the population, and $\mathbb{Q}^{\boldsymbol \alpha,\rho }$ is a probability measure over $(\Omega,\mathcal{F})$. The notation will be our convention throughout the paper whenever there is no possibility for confusion. The canonical process $\boldsymbol X$ appearing in \eqref{eq:cost_of_minors} models the representative player's dynamics under the probability measure $\mathbb{Q}^{\boldsymbol{\alpha,\rho}}$. Under $\mathbb{Q}^{\boldsymbol{\alpha,\rho}}$, $\boldsymbol X$ is a pure-jump process with transition rate matrix $Q(t,\alpha_t, \rho_t)$ at time $t$.\footnote{Existence of the measure $\mathbb{Q}^{\boldsymbol\alpha,\boldsymbol\rho}$ is granted by Girsanov Theorem under some conditions, see for example \cite{Carmona2018Extended} and the references therein. The hypothesis on $Q$ stated in Section~\ref{sec:general-analysis} is strong enough for existence to hold.}

The agent is truly representative if their joint distribution of action and state agrees with the population. The consistency condition in MFG assures just this, see $(ii)$ in definition \ref{def:weak-extended-mfg-equilibrium} where the equilibrium notion in the population's problem is formalized.
\begin{definition}
\label{def:weak-extended-mfg-equilibrium}
If the pair $(\boldsymbol{\hat{\alpha}},\boldsymbol{\hat\rho}) \in \mathbb{A}\times M(\mathcal{R})$ satisfies
\begin{itemize}
    \item[(i)] $\boldsymbol{\hat\alpha} = \arg\inf_{\boldsymbol{\alpha}\in\mathbb{A}} J^{\boldsymbol\lambda, \xi}(\boldsymbol\alpha, \boldsymbol{\hat{\rho}})$;
    \item[(ii)] $\forall t\in[0,T]\ :\ \hat{\rho}_t = \mathbb{Q}^{\boldsymbol{\hat{\alpha}},\boldsymbol{\hat{\rho}}}\circ (\hat{\alpha}_t, X_t)^{-1}$,
\end{itemize}
we say that $(\boldsymbol{\hat \alpha, \hat \rho})$ is a mean-field Nash equilibrium given the contract $(\boldsymbol\lambda, \xi)$. We denote by $ \mathcal{N}(\boldsymbol\lambda, \xi)$ the set of such mean field Nash equilibria.
\end{definition}

We state in Proposition~\ref{prop:connection-Nash-BSDE} below that (under suitable assumptions) $(\boldsymbol{\hat{\alpha}},\boldsymbol{\hat{\rho}})\in \mathcal{N}(\boldsymbol\lambda,\xi)$ if $(\boldsymbol{Y},\boldsymbol Z,\boldsymbol{\hat{\alpha}}, \boldsymbol{\hat{\rho}}, \mathbb{Q})$ is a solution to the following equation\footnote{We define a solution to \eqref{eq:MKV-basde} in line with \cite[Def. 2]{Carmona2018Extended}: the tuple $(\boldsymbol{Y,Z,\alpha,\rho,\mathbb{Q}})$ is a solution to the McKean-Vlasov BSDE \eqref{eq:MKV-basde} if $\boldsymbol Y \in \mathcal{H}^2$, $\boldsymbol Z \in \mathcal{H}^2_X$, $\boldsymbol \alpha \in \mathbb{A}$, $\boldsymbol\rho\in M(\mathcal{R})$, $\mathbb{Q}$ is a probability measure on $(\Omega,\mathcal{F})$, and \eqref{eq:MKV-basde} is satisfied $\mathbb{P}$-a.s. for all $t\in[0,T]$. } under $\mathbb{P}$
\begin{equation}
\label{eq:MKV-basde}
\left\{
\begin{aligned}
    Y_t 
    &= 
    U(\xi) + \int_t^T \hat{H}(s,X_{s-}, Z_s, \hat{\rho}_s)ds - \int_t^T Z^*_sd\mathcal{M}_s,
    \\
    \mathcal{E}_t 
    &= 
    1 + \int_0^t \mathcal{E}_{s-}X^*_{s-}\left(Q(s,\hat{\alpha}_t, \hat{\rho}_s) - Q^0\right)\psi^+_sd\mathcal{M}_s,
    \\
    \hat{\rho}_t 
    &= 
    \mathbb{Q}\circ\left( \hat{\alpha}_t, X_t \right)^{-1},\ \
    \frac{d\mathbb{Q}}{d\mathbb{P}} = \mathcal{E}_T,
    \ \ \hat{\alpha}_t 
    = 
    \hat{a}(t,X_{t-},Z_t, \hat{\rho}_t),
    \end{aligned}
    \right.
\end{equation}
where $\hat{H}$ is the minimized Hamiltonian of the representative player and $\hat{a}$ is the minimizer, defined in \eqref{eq:def-of-minimized-Hamiltonian} below. The solution has a $(\boldsymbol\lambda,\xi)$-dependence (entering the problem through $U(\xi)$ and the Hamiltonian) which we suppress to alleviate the notation.

The principal's problem is to find the policies that yield the most favorable configuration of minor players in terms of their cost. By using policies $(\boldsymbol \lambda, \xi)$ as incentives, the principal can modify the set of mean-field Nash equilibria $\mathcal{N}(\boldsymbol \lambda, \xi)$ and hence exert influence over the population's behavior. In the sequel, unless otherwise mentioned, we consider the following class of policies for the principal. 
\begin{definition}
A policy $(\boldsymbol\lambda,\xi)$ is admissible if the deterministic mapping $\lambda\in \Lambda$, the real-valued random variable $\xi$ is $\mathcal{F}$-measurable, and that $\mathcal{N}(\boldsymbol{\lambda},\xi)$ is a singleton. We denote the set of admissible policies by $\mathcal{C}$.
\end{definition}

To use an admissible policy $(\boldsymbol \lambda, \xi)\in \mathcal{C}$ the principal pays the cost

where $\hat{p}_t^{\boldsymbol\lambda,\xi}(e_i) = \hat{\rho}^{\boldsymbol\lambda,\xi}_t(A,e_i)$, $i=1,\dots,m$, and
$(\boldsymbol{\hat{\alpha}}^{\boldsymbol\lambda,\xi}, \boldsymbol{\hat{\rho}}^{\boldsymbol\lambda,\xi}) = \mathcal{N}(\boldsymbol \lambda,\xi)$.

The last aspect of the problem is a walk-away option of the minor players: all Nash equilibria are disregarded in which the representative agent's expected total cost is higher than the reservation threshold $\kappa$. The principal's optimization problem is
\begin{equation}
    V(\kappa) := \inf_{(\boldsymbol\lambda,\xi)\in \mathcal{C}}\inf_{J^{\boldsymbol \lambda, \xi}(\mathcal{N}(\boldsymbol\lambda,\xi)) \leq \kappa}
    J(\boldsymbol\lambda,\xi).
\end{equation}

\subsection{Analysis of the Stackelberg extended MFG}
\label{sec:analysis}
In this section we will state results under the hypotheses presented below. While setting the hypotheses, we also make the notation used in the previous section precise.

\begin{hypothesis}[Structure and regularity of the $Q$-matrix]
\label{hyp:q}
\vspace{-4pt}
\begin{itemize}
    \item [ ]
    \item[(i)] There exists constants $C_1,C_2>0$ such that for all $(t,i,j,\alpha,\rho)\in [0,T]\times G \times A \times \mathbb{R}$ we have $0<C_1<q(t,i,j,\alpha,\rho)< C_2$. For all $(i,j)\in\{1,\dots,m\}^2\backslash G$, $q(t,i,j,\alpha,p) = 0$ for $t\in[0,T]$, $\alpha\in A$, $\rho\in \mathcal{R}$.
    \item[(ii)] There exists a constant $C>0$ such that for all $t\in [0,T]$, $ (i,j)\in G$, $\alpha,\alpha'\in A$, and $\rho,\rho'\in \mathcal{R}$, we have
    \begin{equation}
        |q(t,i,j,\alpha,\rho) - q(t,i,j,\alpha',\rho')| \leq C\left(|\alpha-\alpha'| + W_{\mathcal{R}}(\rho,\rho')\right).
    \end{equation}
\end{itemize}
\end{hypothesis}
\begin{hypothesis}[Regularity of the running cost]
\label{hyp:f}
\vspace{-4pt}
\begin{itemize}
    \item[ ]
    \item[ ] There exists a constant $C>0$ such that for all $(t,i,\ell) \in [0,T]\times\{1,\dots, m\}\times \mathbb{R}^m_+$, $\alpha, \alpha'\in A$, $p,p'\in\mathcal{P}(E)$, $\rho,\rho'\in\mathcal{R}$, we have
    \begin{equation}
        |f(t,e_i,\alpha,\rho; \ell) - f(t,e_i,\alpha',\rho'; \ell)| \leq C\left(|\alpha - \alpha'| + W_{\mathcal{R}}(\rho,\rho')\right).
    \end{equation}
\end{itemize}
\end{hypothesis}
 Given a policy $(\boldsymbol\lambda,\xi)\in\mathcal{C}$, the Hamiltonian for the representative player's optimization problem is the function $H : [0,T]\times E\times\mathbb{R}^m \times A \times \mathcal{R} \rightarrow \mathbb{R}$
\begin{equation}
H: (t,x,z,\alpha,\rho) \mapsto x^*\left(Q(t,\alpha,\rho)-Q^0\right)z + f(t,x,\alpha,\rho; \lambda_t).
\end{equation}
The representative player's reduced Hamiltonian in state $e_i$ is $H_i : (t,z,\alpha,\rho) \mapsto H(t,e_i,z,\alpha,\rho)$,  $i=1,\dots, m$.
\begin{hypothesis}[Minimizer of the Hamiltonian]
\label{hyp:minimization-of-H-extended}
\vspace{-4pt}
\begin{itemize}
    \item[ ]
    \item[(i)] For any $t\in [0,T]$, $i\in\{1,\dots, m\}$, $z\in \mathbb{R}^m$ and $\rho\in\mathcal{R}$, the mapping $\alpha\mapsto H_i(t,z,\alpha,\rho)$ admits a unique minimizer which we denote by $\hat{a}_i(t,z,\rho)$.
    \item[(ii)] $\hat{a}_i$ is measurable on $[0,T]\times \mathbb{R}^m \times\mathcal{R}$ for every $i \in \{1,\dots,m\}$.
\end{itemize}
\end{hypothesis}
With the minimizers at hand we define the representative player's optimized Hamiltonian $\hat{H}$ and the optimizer $\hat{\alpha}$ as
\begin{equation}
\label{eq:def-of-minimized-Hamiltonian}
        \hat{H}(t,x,z,\rho) 
        := 
        \sum_{i=1}^m 1_{e_i}(x)\hat{H}_i(t,z,\rho),\quad
        \hat{a}(t,x,z,\rho)
        :=
        \sum_{i=1}^m 1_{e_i}\hat{a}_{i}(t,z,\rho),
\end{equation}
where $\hat{H}_i(t,z,\rho) = H_i(t,z,\hat{a}_i(t,z,\rho),\rho)$.

In the notation, we intentionally differentiate between a strategy and the function minimizing the Hamiltonian by denoting the former with the greek letter $\alpha$ and the latter with the hatted latin letter $\hat a$. By evaluating the function $\hat a$ as in \eqref{eq:MKV-basde} we get an admissible strategy of feedback form (feedback on state, aggregate, and joint distribution). Later, in Proposition~\ref{prop:connection-Nash-BSDE}, we study how $\hat a$ can be used to construct a mean-field Nash equilibrium.

\begin{hypothesis}[Regularity of the Hamiltonian minimizer]
\label{hyp:alpha-lip-in-z}
\vspace{-4pt}
\begin{itemize}
    \item[ ]
    \item[ ] There exists a constant $C>0$, independent of the principal's policy, such that for all $(t,i,\rho) \in [0,T]\times \{1,\dots, m\}\times \mathcal{R}$ and $z,z'\in\mathbb{R}^m$: 
    \begin{equation}
        |\hat{a}_i(t,z,\rho) - \hat{a}_i(t,z',\rho)| \leq C\|z-z'\|_{e_i}.
    \end{equation}
\end{itemize}
\end{hypothesis}

The following result provides necessary and sufficient conditions for a mean-field Nash equilibrium. The proof follows the lines of \cite[Thm. 1]{Carmona2018Contract}.
\begin{proposition}
\label{prop:connection-Nash-BSDE}
Assume that Hypothesis~\ref{hyp:q}--\ref{hyp:alpha-lip-in-z} hold true. If \eqref{eq:MKV-basde} admits a solution $(\boldsymbol Y, \boldsymbol Z, \boldsymbol \alpha, \boldsymbol \rho, \mathbb{Q})$ then $(\boldsymbol\alpha,\boldsymbol\rho)$ is a mean-field Nash equilibrium (according to Definition~\ref{def:weak-extended-mfg-equilibrium}). Conversely, if $(\boldsymbol{\hat{\alpha}},\boldsymbol{\hat{\rho}})$ is a mean-field Nash equilibrium then \eqref{eq:MKV-basde} admits a solution $(\boldsymbol Y, \boldsymbol Z, \boldsymbol \alpha, \boldsymbol \rho, \mathbb{Q})$ such that $\boldsymbol\alpha = \boldsymbol{\hat{\alpha}}$, $d\mathbb{P}\otimes dt$-a.s., and $\rho_t = \hat{\rho}_t$, $dt$-a.e.
\end{proposition}

Given $\boldsymbol Z \in \mathcal{H}^2_X$, $\boldsymbol\lambda \in \Lambda$, and real-valued $\mathcal{F}_0$-measurable $Y_0$, consider under $\mathbb{P}$:
\begin{equation}
\label{eq:Y-ZlambdaY0}
\left\{
\begin{aligned}
    Y^{\boldsymbol Z,\boldsymbol \lambda, Y_0}_t &= Y_0 - \int_0^t \hat{H}(s, X_{s-}, Z_s, \hat \rho^{\boldsymbol Z,\boldsymbol \lambda, Y_0}_s)ds + \int_0^t Z^*_sd\mathcal{M}_s,
    \\
    \mathcal{E}_t &= 1 + \int_0^t \mathcal{E}_{s-}X^*_{s-}\left(Q(s,\hat \alpha^{\boldsymbol Z,\boldsymbol \lambda, Y_0}_s,\hat \rho^{\boldsymbol Z,\boldsymbol \lambda, Y_0}_s) - Q^0\right)\psi^+_sd\mathcal{M}_s,
    \\
    \hat \rho^{\boldsymbol Z,\boldsymbol \lambda, Y_0}_t &= \mathbb{Q}^{\boldsymbol Z, \boldsymbol \lambda, Y_0}\circ\left( \hat{\alpha}^{\boldsymbol Z,\boldsymbol \lambda, Y_0}_t, X_t \right)^{-1}, \ \ 
    \frac{d\mathbb{Q}^{\boldsymbol Z, \boldsymbol \lambda, Y_0}}{d\mathbb{P}} = \mathcal{E}_T,
    \\
    \hat \alpha^{\boldsymbol Z,\boldsymbol \lambda, Y_0}_t &= \hat a(t,X_{t-}, Z_t, \hat \rho^{\boldsymbol Z,\boldsymbol \lambda, Y_0}_t), \ \
    \hat{p}^{\boldsymbol Z,\boldsymbol \lambda, Y_0}_t(\cdot) = \hat{\rho}^{\boldsymbol Z,\boldsymbol \lambda, Y_0}_t(A,\cdot).
\end{aligned}
\right.
\end{equation}
These are the same equations as \eqref{eq:MKV-basde}, except that the dynamic of $\boldsymbol Y$ is written in the forward direction of time. Here $Y_0$ is fixed instead of $Y_T$. 

\begin{hypothesis}[Regularity of the principal's cost]
\label{hyp:cost-in-tilde-V}
\vspace{-4pt}
\begin{itemize}
    \item []
    \item[(i)] The function $U:\mathbb{R}\rightarrow \mathbb{R}$ is invertible.
    \item[(ii)] $c_0, f_0$ are measurable on $[0,T]\times \mathbb{R}^3$.
\end{itemize}
\end{hypothesis}
Consider the following optimal control problem
\begin{equation}
\label{eq:tilde-V}
\begin{aligned}
    \widetilde V(\kappa) &:= \inf_{Y_0:\mathbb{E}[Y_0]\leq \kappa}\inf_{\substack{\boldsymbol Z \in \mathcal{H}^2_X \\ \boldsymbol\lambda \in \Lambda}}\mathbb{E}^{\mathbb{Q}^{\boldsymbol Z, \boldsymbol \lambda, Y_0}}\Bigg[
    \int_0^T \left( c_0\left(t, \hat{p}_t^{\boldsymbol Z, \boldsymbol \lambda, Y_0}\right) + f_0(t,\lambda_t) \right)dt
    \\
    &\hspace{4cm}
    + C_0\left(\hat{p}_T^{\boldsymbol Z, \boldsymbol \lambda, Y_0}\right) + U^{-1}\left(-Y_T^{\boldsymbol Z, \boldsymbol \lambda, Y_0}\right)
    \Bigg],
\end{aligned}
\end{equation}
under the dynamic constraint~\eqref{eq:Y-ZlambdaY0} under $\mathbb{P}$ (the dynamic under $\mathbb{Q}^{\boldsymbol Z, \boldsymbol \lambda, Y_0}$ is given below in~\eqref{eq:Y-Mbis}).
The optimization is now performed not only over the principal's control policy, $\boldsymbol\lambda$, but also over the initial condition $Y_0$ and the $\boldsymbol{Z}$ component. Since we want to find a solution to~\eqref{eq:MKV-basde}, the terminal $Y_T$ must, by definition of the representative agent's problem, equal the utility $U(\xi)$ of the terminal payment $\xi$. This remark allows us to remove $\xi$ from the principal's problem, replacing it by $U^{-1}(Y_T)$.

\begin{proposition}
\label{prop:rewriting}
If Hypothesis~\ref{hyp:q}--\ref{hyp:alpha-lip-in-z}, \ref{hyp:cost-in-tilde-V} hold true, then $\widetilde V(\kappa) = V(\kappa)$.
\end{proposition}
The proof follows the lines of \cite[Thm. 2]{Carmona2018Contract}.

Our final result says that certain extended MFGs of the type \eqref{eq:cost_of_minors} are equivalent to regular MFGs (where there is no dependence on the distribution of player actions) in the sense that the two problem's Nash equilibria are the same. The key property of games with this feature is that their Hamiltonian and transition rate matrix evaluated at the mean-field Nash equilibrium are functions of the equilibrium state distribution, not the full joint distribution of equilibrium state and control. The property is formalized in the following hypothesis:
\begin{hypothesis}[Properties for Nash equilibria simplification]
\label{hyp:reduction}
\vspace{-4pt}
\begin{itemize}
\item[ ]
\item[(i)] There exists a unique solution $(\boldsymbol{ \hat{Y}, \hat{Z}, \hat{\alpha},\hat{\rho},\hat{\mathbb{Q}}})$ to \eqref{eq:MKV-basde}.
\item[(ii)] There exists measurable functions $\bar{a}_i: [0,T]\times \mathbb{R}^m\times\mathcal{P}(E)\rightarrow A$,  $\bar{f} : [0,T]\times E\times A \times  \mathbb{R}^m\times A\times \mathcal{P}(E) \rightarrow \mathbb{R}$ and $\bar{Q} : [0,T]\times A\times \mathcal{P}(E)$ such that for all $(t,i,z,\rho,\ell)\in[0,T]\times\{1,\dots, m\}\times\mathbb{R}^m\times\mathcal{R}\times\mathbb{R}^m_+$:
\begin{equation}
    \begin{aligned}
            \hat{a}_i(t,z,\hat{\rho}_t) &= \bar{a}_i(t,z,\hat{p}_t),
            \\
            f(t,e_i,\hat{a}_i(t,z,\hat{\rho}_t),\hat{\rho}_t; \ell) &= \bar{f}(t,e_i,\bar{a}_i(t,z,\hat{p}_t),\hat{p}_t; \ell),
            \\
          Q(t,\hat{a}_i(t,z,\hat{\rho}_t),\hat{\rho}_t) 
            &= \bar{Q}(t,\bar{a}_i(t,z,\hat{p}_t),\hat{p}_t),
    \end{aligned}
\end{equation}
where $\hat{p}_t(e_i) := \hat \rho_t(A,e_i)$, $i = 1,\dots, m$.
\item[(iii)] There exists constants $C_1$ and $C_2$, independent of the principal's policy, such that for all $(t,i)\in[0,T]\times \{1,\dots,m\}$, $z,z'\in\mathbb{R}^m$ and $p,p'\in\mathcal{P}(E)$:
\begin{equation}
    |\bar{a}_i(t,z,p) - \bar{a}_i(t,z',p')| \leq C_1\|z-z'\|_{e_i} + (C_1 + C_2\|z\|_{e_i})\|p-p'\|.
\end{equation}
\end{itemize}
\end{hypothesis}
Assuming that hypothesis~\ref{hyp:reduction} is true we define the non-extended mean field Nash equilibrium of the game as follows:
\begin{definition}
\label{def:weak-mfg-equilibrium}
Let $(\boldsymbol{\alpha},\boldsymbol{p}) \in \mathbb{A}\times M(\mathcal{P}(E))$ and denote by $\mathbb{Q}^{\boldsymbol{\alpha,p}}\in\mathcal{P}(\Omega)$ the measure such that the coordinate process $X_t$ has transition rate matrix $\bar{Q}(t,\alpha_t,p_t)$ under $\mathbb{Q}^{\boldsymbol{\alpha,p}}$.
Assume that $(\boldsymbol{\bar{\alpha}},\boldsymbol{\bar p})\in \mathbb{A}\times M(\mathcal{P}(E))$ satisfies
\begin{itemize}
    \item[(i)] $\boldsymbol{\bar\alpha} = \arg\inf_{\alpha\in\mathbb{A}}
    \mathbb{E}^{\mathbb{Q}^{\boldsymbol{\alpha, \bar{p}}}}\left[\int_0^T \bar{f}(t,X_t,\alpha_t, \bar{p}_t)dt - U(\xi)\right]$,
    \item[(ii)] $\forall t\in[0,T], i\in\{1,\dots, m\} :\ \bar{p}_t(i) = \mathbb{Q}^{\boldsymbol{\bar{\alpha}},\boldsymbol{\bar{\rho}}}\left(X_t=e_i\right)$.
\end{itemize}
Then $(\boldsymbol{\bar{\alpha}},\boldsymbol{\bar p})$ is called a non-extended mean field Nash equilibrium.
\end{definition}
The following result allows a simplification of the Nash equilibrium through a non-extended problem. The proof is found in Appendix~\ref{sec:proof-prop-reduction}.
\begin{proposition}
\label{prop:reduction}
Assume Hypothesis~\ref{hyp:q}--\ref{hyp:minimization-of-H-extended}, \ref{hyp:reduction} to be true. Denote the tuple of Hypothesis~\ref{hyp:reduction}(i) by $(\boldsymbol{\hat{Y}}, \boldsymbol{\hat{Z}}, \boldsymbol{\hat{\alpha}}, \boldsymbol{\hat{\rho}},\boldsymbol{\mathbb{Q}})$. The pair $(\boldsymbol{\hat{\alpha}}, \boldsymbol{\hat{\rho}})$ is a mean-field Nash equilibrium.
Let $\hat{p}_t$ be the the $E$-marginal of $\hat{\rho}_t$ and let $(\boldsymbol{\bar{\alpha},\bar{p}})$ be a non-extended mean field Nash equilibrium, satisfying Definition~\ref{def:weak-mfg-equilibrium}. Then $\hat{p}_t = \bar{p}_t$ for $dt$-a.e. $t\in[0,T]$ and $\hat{\alpha}_t = \bar{\alpha}_t\ d\mathbb{P}\otimes dt\text{-a.e.}$.
\end{proposition}

\section{Numerical approach}
\label{sec:numerics}

In this section, we propose a numerical method to solve the Stackelberg equilibrium. This requires finding the optimal policy of the principal and the associated mean-field Nash equilibrium for the population of agents. The principal's policy influences the Nash equilibrium in a rather intricate way. For this reason, we depart from existing numerical methods for finite state mean field games such as~\cite{MR3873029}, and we propose a probabilistic method in which Monte Carlo samples are generated to train neural networks approximating the optimal controls, including the principal's policy.

\subsection{Monte Carlo simulation}
\label{subsec:MonteCarlo}

From Proposition~\ref{prop:rewriting}, we know that solving the original Stackelberg MFG problem amounts to solving an optimal control problem in which the state can be viewed as $(\boldsymbol{X},\boldsymbol{Y})$ and has a forward dynamics: under $\mathbb{P}$, $\boldsymbol{X}$ is a continuous time Markov chain with $Q$-matrix $Q^0$ and $\boldsymbol{Y}$ satisfies~\eqref{eq:Y-ZlambdaY0}. We recall that the controls are $\boldsymbol Z \in \mathcal{H}^2_X$, $\boldsymbol\lambda \in \Lambda$, and a real-valued $\mathcal{F}_0$-measurable random variable $Y_0$. This problem involves the state and action distribution. We replace this distribution by an empirical distribution obtained with an interacting system of particles and we discretize the time integral. We first note that, given a triple of controls $(\boldsymbol Z,\boldsymbol \lambda, Y_0)$, the dynamic~\eqref{eq:Y-ZlambdaY0} can be written as:
\begin{equation}
\label{eq:Y-Mbis}
\left\{
\begin{aligned}
    Y^{\boldsymbol Z,\boldsymbol \lambda, Y_0}_t &= Y_0 - \int_0^t f(s, X_{s-}, \hat{\alpha}^{\boldsymbol Z,\boldsymbol \lambda, Y_0}_t, \hat \rho^{\boldsymbol Z,\boldsymbol \lambda, Y_0}_s; \lambda_s)ds + \int_0^t Z^*_sd\mathcal{M}^{\boldsymbol Z,\boldsymbol \lambda, Y_0}_s,
    \\
    \mathcal{E}_t &= 1 + \int_0^t \mathcal{E}_{s-}X^*_{s-}\left(Q(s,\hat \alpha^{\boldsymbol Z,\boldsymbol \lambda, Y_0}_s,\hat \rho^{\boldsymbol Z,\boldsymbol \lambda, Y_0}_s) - Q^0\right)\psi^+_sd\mathcal{M}_s,
    \\
    \hat \rho^{\boldsymbol Z,\boldsymbol \lambda, Y_0}_t &= \mathbb{Q}^{\boldsymbol Z, \boldsymbol \lambda, Y_0}\circ\left( \hat{\alpha}^{\boldsymbol Z,\boldsymbol \lambda, Y_0}_t, X_t \right)^{-1}, \ \ 
    \frac{d\mathbb{Q}^{\boldsymbol Z, \boldsymbol \lambda, Y_0}}{d\mathbb{P}} = \mathcal{E}_T,
    \\
    \hat \alpha^{\boldsymbol Z,\boldsymbol \lambda, Y_0}_t &= \hat \alpha(t,X_{t-}, Z_t, \hat \rho^{\boldsymbol Z,\boldsymbol \lambda, Y_0}_t; \lambda_s), \ \
    \hat{p}^{\boldsymbol Z,\boldsymbol \lambda, Y_0}_t(\cdot) = \hat{\rho}^{\boldsymbol Z,\boldsymbol \lambda, Y_0}_t(A,\cdot),
\end{aligned}
\right.
\end{equation}
where the process $\boldsymbol{\mathcal{M}}^{\boldsymbol Z,\boldsymbol \lambda, Y_0}$ is defined by:
$$
    \mathcal{M}_t^{\boldsymbol Z,\boldsymbol \lambda, Y_0}
    = \mathcal{M}_t - \int_0^t X_{s-}^* \left(Q(s,\hat \alpha^{\boldsymbol Z,\boldsymbol \lambda, Y_0}_s,\hat \rho^{\boldsymbol Z,\boldsymbol \lambda, Y_0}_s) - Q^0\right) ds, 
$$
is a $\mathbb{Q}^{\boldsymbol Z,\boldsymbol \lambda, Y_0}$-martingale. Furthermore, from this definition we see that the canonical process $\boldsymbol{X}$ satisfies, under $\mathbb{Q}^{\boldsymbol Z,\boldsymbol \lambda, Y_0}$,
\begin{equation}
\label{eq:X-Mbis}
    X_t = X_0 + \int_0^t X_{s-}^* Q(s,\hat \alpha^{\boldsymbol Z,\boldsymbol \lambda, Y_0}_s,\hat \rho^{\boldsymbol Z,\boldsymbol \lambda, Y_0}_s) ds + \mathcal{M}_t^{\boldsymbol Z,\boldsymbol \lambda, Y_0}.
\end{equation}
In other words, under the probability measure $\mathbb{Q}^{\boldsymbol Z,\boldsymbol \lambda, Y_0}$, the intensity rate of $\boldsymbol{X}$ is given by $Q(s,\hat \alpha^{\boldsymbol Z,\boldsymbol \lambda, Y_0}_s,$ $\hat \rho^{\boldsymbol Z,\boldsymbol \lambda, Y_0}_s)$.

To simulate Monte Carlo trajectories, we will use the expressions~\eqref{eq:Y-Mbis}--\eqref{eq:X-Mbis}. For simplicity of the implementation, we assume that given any admissible policy $(\boldsymbol{\lambda}, \xi)$, we can express the induced equilibrium control $\boldsymbol{\hat{\alpha}}$ as a function of $\boldsymbol{\hat{p}}$, the flow of second marginals of the equilibrium distribution flow $\boldsymbol{\hat{\rho}}$. To wit, for every $(\boldsymbol{\lambda}, \xi)$, denoting $(\boldsymbol{Y},\boldsymbol Z,\boldsymbol{\hat{\alpha}}, \boldsymbol{\hat{\rho}}, \mathbb{Q})$ the solution to~\eqref{eq:MKV-basde}, we assume that there exists $\check{a}:[0,T] \times E \times \mathbb{R}^m \times \mathcal{P}(E) \to \mathbb{R}$ such that 
$$
    \hat{\alpha}_t = \check{a}\left(t,X_t,Z_t,\hat{p}_t\right),
$$
where $\hat{p}_t = \hat{\rho}_t(A,\cdot)$ is the state marginal of $\hat{\rho}$. This is automatically true if the minimizer $\hat{a}$ of the Hamiltonian is independent of the first marginal of $\rho$, i.e., 
 $$
    \hat{a}_i(t,z,\rho) = \check{a}_i(t, z,\rho(A,\cdot)),
 $$
 for a function $\check{a}_i : [0,T] \times \mathbb{R}^m \times \mathcal{P}(E) \to \mathbb{R}$, 
which is often assumed to be true in extended MFGs, see e.g.~\cite[Assumption 3.4]{Carmona2018Extended} in the finite space setting and~\cite{laurieretangpi2019backward,laurieretangpi2020convergence} in the continuous space setting.  However, it also holds in more general situations, for instance when the equilibrium control can be expressed in terms of the solution to a forward-backward PDE system~\cite{gomes2014existence,chan2015bertrand,kobeissi2019classical} in the continuous space setting. In such cases, the equilibrium control is expressed as a feedback function of $(t,x)$ related to the backward PDE, and the distribution of actions can be recovered from this feedback control and the state distribution related to the forward PDE.

We now present the scheme with a finite number of particles and discrete time steps. We consider $N>0$ particles and denote $\llbracket N \rrbracket = \{1,\dots,N\}$ the set of indexes. Given measurable control functions $z: [0,T] \times E \to \mathbb{R}^m, \lambda: [0,T] \to \mathbb{R}^m_+, y_0: E \to \mathbb{R}$, we construct trajectories $(X^i_{t_n},Y^i_{t_n})_{n, i=1,\dots,N}$. After initialization, we proceed iteratively for every $n \ge 1$ while $t_n \le T$: $X^i_{t_n}$ is sampled with $Q$-matrix given by 
$$
    Q^{i}_{t_n} := Q(t_n, \alpha^{i}_{t_n}, \overline{\rho}^{N}_{t_n}),
$$ 
where 
$$
    \overline{\rho}^{N}_{t_n}
    =
    \frac{1}{N} \sum_{i=1}^N \delta_{\left(X^i_{t_n},  \alpha^{i}_{t_n}\right)}
$$
is the empirical action-statae  distribution, and
$$
     \alpha^{i}_{t_n}
    =
    \check{a}(t,X^i_{t_n}, z(t_n, X^i_{t_n}), \overline{p}^{N}_{t_n}), \qquad \hbox { where } \overline{p}^{N}_{t_n}
    =
    \frac{1}{N} \sum_{i=1}^N \delta_{X^i_{t_n}}.
$$
Based on~\eqref{eq:X-Mbis}, we define:
$$
    \Delta \mathcal{M}_{t_n}^{i}
    =
    - (X^i_{t_{n+1}} - X^i_{t_{n}}) - (X^i_{t_n})^* Q^{i}_{t_n} (t_{n+1} - t_n),
$$
and then, based on~\eqref{eq:Y-Mbis}, we let: $Y^i_{0} = y(X^i_0)$ and for $n\ge 0$,
\begin{align*}
    Y^i_{t_{n+1}} 
    &= Y^i_{t_{n}} - f(t_n, X^i_{t_n}, {\alpha}^{i}_{t_n}, \overline{\rho}^{N}_{t_n}; \lambda(t_n)) (t_{n+1} - t_{n}) 
    + z(t_n, X^i_{t_n})^* \Delta \mathcal{M}_{t_n}^{i}.
\end{align*}
Here $t_n$ corresponds to the time of the $n$-th jump in the particle system $(\boldsymbol X^i)_{i=1,\dots,N}$. In the implementation, these trajectories are constructed using a time marching procedure from time $0$ until time $T$, with steps corresponding to jumps. For more details, see Algorithm~\ref{algo:simu-forward-XY}.

\subsection{Approximation based on neural networks}
\label{subsec:NN_approx}

In order to have a problem amenable to numerical treatment, we replace the controls $(\boldsymbol Z,\boldsymbol \lambda, Y_0)$ by parameterized functions $z_{\theta_1}: [0,T] \times E \to \mathbb{R}^m$, $\lambda_{\theta_2}: [0,T] \to \mathbb{R}^m_+$, and $y_{0,\theta_3}: E \to \mathbb{R}$ with respective parameters $\theta_1, \theta_2, \theta_3$.  Since the number of states $m$ is potentially large, we choose to use neural networks and, to be specific, in the implementation we take feedforward fully connected neural networks. Then, taking into account the time discretization and the approximation using a finite number of particles as described above, instead of~\eqref{eq:tilde-V}, and considering a finite number $M$ of Monte Carlo samples, the goal is now to minimize over $\theta = (\theta_1,\theta_2,\theta_3)$ the objective function: 
\begin{equation}
\label{eq:J-NN}
\begin{aligned}
    \mathbb{J}^N(\theta)
    =
    & \frac{1}{M} \sum_{j=1}^M \Bigg[
    \sum_{n=0}^{n_{tot}-1} \left( c_0\left(t_n, \bar{p}_{t_n}^{j, N, \theta}\right) 
    + f_0(t_n,\lambda_{\theta_2}(t_n)) \right) (t_{n+1} - t_{n})
    \\
    &\hspace{2.5cm}
    + C_0\left(\bar{p}_T^{j, N, \theta}\right) + \frac{1}{N} \sum_{i=1}^N U^{-1}\left(-Y_T^{j,i,\theta}\right)
    \Bigg]
\end{aligned}
\end{equation}
where, for $j=1,\dots,M$, $(\boldsymbol{Y}^{j, i, \theta})_{i \in \llbracket N \rrbracket}$ and $\boldsymbol{\bar{p}}^{j, N, \theta}$ are constructed by Algorithm~\ref{algo:simu-forward-XY} using $(z,\lambda,y_0) = (z_{\theta_1}, \lambda_{\theta_2}, y_{0,\theta_3})$. Intuitively, in the limit when $M$, $N$ and the number of parameters $\theta$ go to infinity, we would expect $\inf_\theta \mathbb{J}^N(\theta)$ to converge to the principal's optimal cost.

To optimize over $\theta = (\theta_1,\theta_2,\theta_3)$, we rely on a variant of stochastic gradient descent (namely the Adaptive Moment Estimation algorithm). This kind of methods is particularly well suited to the minimization of $\mathbb{J}^N$ in~\eqref{eq:J-NN} since on the one hand the number of parameters in deep neural networks is potentially large and on the other hand this cost is written as an expectation and can thus be computed using Monte Carlo samples. Our method can be viewed as an adaptation of the second algorithm in~\cite{carmona2019convergenceFinite} to the finite state case and its generalization to the Stackelberg setting with a principal. More precisely, we introduce for a sample $S = (X^i_{t_n},Y^i_{t_n},Z^i_{t_n})_{n =0,\dots,n_{tot}, i \in \llbracket N \rrbracket}$ of a population of $N$ particles:
\begin{equation}
\label{eq:J-NN-Sample}
\begin{aligned}
    \mathbb{J}^N_{S}(\theta)
    =
    & 
    \sum_{n=0}^{n_{tot}-1} \left( c_0\left(t_n, \bar{p}_{t_n}^{N, \theta}\right) 
    + f_0(t_n,\lambda_{\theta_2}(t_n)) \right) (t_{n+1} - t_{n})
    \\
    &\hspace{2.5cm}
    + C_0\left(\bar{p}_T^{N, \theta}\right) + \frac{1}{N} \sum_{i=1}^N U^{-1}\left(-Y_T^{i,\theta}\right)
\end{aligned}
\end{equation}
where $\overline{p}^{N}_{t_n} = \frac{1}{N} \sum_{i=1}^N \delta_{X^i_{t_n}}$. Note that~\eqref{eq:J-NN} is simply the average of $\mathbb{J}^N_{S}$ over $M$ samples (here one sample corresponds to the trajectories of one population). See Algorithm~\ref{algo:SGD-SMFG} for more details on the stochastic gradient descent (SGD) method applied in the context of Stackelberg mean field game.

\begin{remark}
If the regulator is not optimizing its own outcomes, in other words if the policies set by the regulator, $(\boldsymbol{\lambda}, \xi)$, are exogenous, the neural network approach can be still used. Since $\boldsymbol{\lambda}$ is known in Section~\ref{subsec:MonteCarlo}, we can write system \eqref{eq:Y-Mbis} controlled only by $\boldsymbol{Z}$ and $Y_0$. Then in Section~\ref{subsec:NN_approx}, we replace controls $(\boldsymbol{Z}, Y_0)$ by parameterized functions $z_{\theta_1}: [0,T] \times E \to \mathbb{R}^m$ and $y_{0,\theta_2}: E \to \mathbb{R}$ with respective parameters $\theta_1, \theta_2$. Then considering a finite number M of Monte Carlo samples with N particles in each, the goal is to minimize over $\theta=(\theta_1, \theta_2)$ the loss function:
\begin{equation}
    \mathbb{L}^N(\theta)= \frac{1}{M}\sum_{j=1}^M\Big[\big(U(\xi) - \frac{1}{N} \sum_{i=1}^{N} Y^{j,i, \theta}_T\big)^2\Big]
\end{equation}
\end{remark}

\begin{algorithm}
\caption{Monte Carlo simulation an interacting batch \label{algo:simu-forward-XY}}
\textbf{Input:} Transition rate matrix $Q$; number of particles $N$; time horizon $T$; initial distribution $p_{0}$; control functions $\lambda, y_0, z$

\textbf{Output:} Approximate sampled trajectories of $(\boldsymbol X,\boldsymbol Y,\boldsymbol Z)$ solving~\eqref{eq:Y-Mbis}--\eqref{eq:X-Mbis}
\begin{algorithmic}[1]
\STATE{Let $n=0$, $t_0 = 0$; pick $X^i_0 \sim p^{0}$ i.i.d and set $Y^i_0 = y_0(X^i_0), i \in \llbracket N \rrbracket$} 
\WHILE{$t_n \le T$}
    \STATE{Set $ Z^i_{t_n} = z({t_n}, X^i_{t_n}), \alpha^i_{t_n} = \check{a}({t_n},X^i_{t_n}, Z^i_{t_n}, p_{t_n}), i \in \llbracket N \rrbracket$}
    \STATE{Let $\bar \rho^N_{t_n} = \frac{1}{N} \sum_{i=1}^N \delta_{(X^i_{t_n}, \alpha^i_{t_n})}$ and $\bar p^N_{t_n} = \frac{1}{N} \sum_{i=1}^N \delta_{X^i_{t_n}}$ }
    \STATE{Pick $(T^{i,e})_{e \in E, i \in \llbracket N \rrbracket}$ i.i.d. with exponential distribution of parameter $1$ }
    \STATE{Set the holding times: $\tau^{i,e} = T^{i,e} / Q_{X^i_{t_n},e}(t_n, \alpha^i_{t_n}, \bar\rho^N_{t_n})$, $i \in \llbracket N \rrbracket, e \in E$ }
    \STATE{Let $e^{i}_{\star} = \argmin_{e \in E} \tau^{i,e}$ and $\tau^{i}_{\star} = \tau^{i,e^{i}_{\star}} = \min_{e \in E} \tau^{i,e}, i \in \llbracket N \rrbracket$ } 
    \STATE{Let $i_{\star} = \argmin_{i \in \llbracket N \rrbracket} \tau^{i}_{\star}$ be the first particle to jump }
    \STATE{Let $\Delta t = \tau^{i_{\star}}_{\star}$ be the time increment }
    \STATE{Set $X^{i_{\star}}_{t_n+\Delta t} = e^{i_{\star}}_{\star}$, and for every $i \neq i_{\star}$, set $X^{i}_{t_n+\Delta t} = X^{i}_{t_n}$ }
    \STATE{Let $\Delta M^{i}_{t_n} = X^{i}_{t_n+\Delta t} - X^{i}_{t_n} -  (X^{i}_{t_n})^* Q(t_n, \alpha^i_{t_n}, \bar\rho^N_{t_n})  \Delta t$, $i \in \llbracket N \rrbracket$}
    \STATE{Let 
        $
            Y^{i}_{t_n+\Delta t} = Y^{i}_{t_n} - f(t, X^{i}_{t_n}, \alpha^{i}_{t_n}, \bar\rho^N_{t_n}; \lambda(t_n))\Delta t + (Z^{i}_{t_n})^* \Delta M^{i}_{t_n}
        $, $ i \in \llbracket N \rrbracket$
    }
    \STATE{Set $n=n+1$ and $t_n = t_{n-1}+\Delta t$}
\ENDWHILE
\STATE{Set $n_{tot} = n, t_{n_{tot}} = T,  (X^i_{t_{n_{tot}}},Y^i_{t_{n_{tot}}},Z^i_{t_{n_{tot}}}) = (X^i_{t_{n_{tot}-1}},Y^i_{t_{n_{tot}-1}},Z^i_{t_{n_{tot}-1}})$} 
\RETURN $(X^i_{t_n},Y^i_{t_n},Z^i_{t_n})_{n =0,\dots,n_{tot}, i \in \llbracket N \rrbracket}$ and $(t_n)_{n =0,\dots,n_{tot}}$
\end{algorithmic}
\end{algorithm}

\begin{algorithm}
\caption{SGD for Stackelberg Mean Field Game\label{algo:SGD-SMFG}}

\textbf{Input:} Initial parameter $\theta_0$; number of iterations $K$; sequence $(\beta_k)_{k = 0, \dots, K-1}$ of learning rates; transition rate matrix $Q$; number of particles $N$; time horizon $T$; initial distribution $p_{0}$

\textbf{Output:} Approximation of $\theta^*$ minimizing $\mathbb{J}^N$ defined by ~\eqref{eq:J-NN}
\begin{algorithmic}[1]

\FOR{$k = 0, 1, 2, \dots, K-1$}
    \STATE{Sample $S = (X^i_{t_n},Y^i_{t_n},Z^i_{t_n})_{n =0,\dots,n_{tot}, i \in \llbracket N \rrbracket}$ and $(t_n)_{n =0,\dots,n_{tot}}$ using Algorithm~\ref{algo:simu-forward-XY} with control functions $(z, \lambda, y_0) = (z_{\theta_{k,0}}, \lambda_{\theta_{k,1}}, y_{0,\theta_{k,2}})$  and parameters:  transition rate matrix $Q$; number of particles $N$; time horizon $T$; initial distribution $p_{0}$ }
    \STATE{Compute the gradient $\nabla \mathbb{J}^N_{S}(\theta_{k})$ of $\mathbb{J}^N_{S}(\theta_{k})$ defined by~\eqref{eq:J-NN-Sample} }
    \STATE{Set $\theta_{k+1} =  \theta_{k} -\beta_k \nabla \mathbb{J}^N_{S}(\theta_{k})$ }
\ENDFOR

\RETURN $\theta_{K}$
\end{algorithmic}
\end{algorithm}

\section{Experimental results}
\label{subsec:experiments}

\subsection{SIR equilibrium with a fixed policy}

Our first numerical experiment is an extended MFG with SIR dynamics. The principal is not active in this experiment; we think of the principal as a static regulator who has beforehand declared a fixed policy. The purpose is two-fold. Firstly, since a semi-explicit solution is attainable the numerical method's accuracy can be evaluated (see Fig.~\ref{fig:LL-ode} and \ref{fig:LL-NN}), and secondly, we illuminate the game-aspect of the model and how the agents are prone to ``cheat" a little and use a less conservative socialization protocol than the incentivized one.

We begin by recalling the problem introduced in Section~\ref{sec:intro-SIRMFG} and we focus here on the limiting problem with a continuum of agents. The representative agent evolves according to the rate matrix $Q(t,\alpha_t,\rho_t)$ given in \eqref{eq:SIR-Qmatrix-intro} and the running cost in \eqref{eq:cost-agents-SIR1-intro}. As stated in Section~\ref{sec:intro-SIRMFG}, the running cost does not depend on the extended mean field, however the dynamic does, making the problem an extended MFG. The representative agent's terminal utility is $U(\xi) = \xi$.

\begin{table}[h!]
\caption{Parameter values, policies, and contact factor in the four numerical experiments on the SIR extended MFG. The regulator declares a fixed policy $(\boldsymbol\lambda, \xi)$.}
\centering
\ra{1.4}
\begin{tabular}{@{}llccccccc@{}}
\toprule
\textbf{Test case} & \textbf{Contact factor} &
$\xi$ & $\lambda^{(S)}_t$ & $\lambda^{(I)}_t$ & $\lambda^{(R)}_t$
\\
\midrule
Free spread & Constant & $0$ & $1$ & $1$ & $1$
\\ 
No lockdown & MF Nash eq. & $0$ & $1$ & $1$ & $1$
\\ 
Late lockdown & MF Nash eq.  & $0$ & $1 - 0.3\mathbbm{1}_{t> 40}$ & $0.9 - 0.3\mathbbm{1}_{t>40}$ & $1$
\\ 
Early lockdown & MF Nash eq. & 0 & $1 - 0.3\mathbbm{1}_{t\leq 10}$ & $0.9 - 0.3\mathbbm{1}_{t\leq 10}$ & $1$
\\ 
\bottomrule
\end{tabular}
\begin{tabular}{@{}lcccccccc@{}}
\textbf{Parameter} & $T$ & $p^0$ &
$c_\lambda$ & $c_I$ & $\beta$ & $\gamma$ & $\eta$
\\
\midrule
\textbf{Value in tests} & $50$ & $(0.9,0.1,0)$ & $10$ & $1$ & $0.25$ & $0.1$ & $0$
\\ 
\bottomrule
\end{tabular}
\label{table:test1}
\end{table}

Given $(\boldsymbol\lambda,\xi)$, the mean field Nash equilibrium can be reduced to the solution of a system of forward-backward ODEs as we explain in more details below. We solve the ODEs by iteratively solving the forward and the backward equations, until the distance between two iterates is small enough.

Since the neural network based method we propose is new, we consider a testbed on which we can assess its correctness using a well-studied approach, which serves as a benchmark. 
For any given $(\boldsymbol\lambda,\xi)$, the mean field Nash equilibrium can be reduced to the solution of a system of forward-backward ODEs, in the spirit of, for example, \cite[Section 7.2.2]{Carmona2018book}. The forward equation describes the evolution of the population distribution while the backward characterizes the value function of an infinitesimal agent:
\begin{equation}
    \begin{aligned}
    \label{eq:ODE-SIR-MFG}
    \dot p_t(S) &= - \beta\Big(\lambda_t^{(S)} + \frac{\beta}{c_{\lambda}} \lambda_t^{(I)} p_t(I) (u_t(S)- u_t(I))\Big) \lambda^{(I)} p_t(S) p_t(I) + \eta p_t(R),
    \\
    \dot p_t(I) &= \beta\Big(\lambda_t^{(S)} + \frac{\beta}{c_{\lambda}} \lambda_t^{(I)} p_t(I) (u_t(S)- u_t(I))\Big) \lambda^{(I)} p_t(S) p_t(I) - \gamma p_t(I),
    \\
    \dot p_t(R) &=\gamma p_t(I)  - \eta  p_t(R)
    \\
    \dot u_t(S) &= \beta \lambda^{(S)} \lambda^{(I)} p_t(I) (u_t(S)-u_t(I)) + \frac{1}{2 c_{\lambda}} \big(\beta \lambda^{(I)}p_t(S)(u_t(S)-u_t(I))\big)^2
    \\
    \dot u_t(I) &= \gamma(u_t(I)-u_t(R)) - c_I
    \\
    \dot u_t(R) &= \kappa(u_t(R)-u_t(S)) 
    \\
    u_T(e)&= 0,\quad p_0(e) = p_0^{e},\quad e\in \{S, I, R\},
    \end{aligned}
\end{equation}

 They are, in the finite state MFG, the counterparts to Kolmogorov-Fokker-Planck and Hamilton-Jacobi-Bellman partial differential equations arising in continuous space MFGs.  For more details on the derivation of these ODEs for a (slightly different) class of finite-state MFGs, we refer the reader to~\cite[Section 7.2.2]{Carmona2018book}. 

The six equations are coupled, reflecting the fact that an agent cannot compute their value function (and their optimal control) without knowing the population evolution when in Nash equilibrium, and the population distribution cannot be computed without knowing the controls chosen by the agents. For this reason, we cannot solve one equation before the other. We propose to solve this system by solving iteratively the forward and the backward ODEs in turn, plugging the solution of the previous iteration in the equation at the current iteration. To implement this strategy, we discretize time and replace the distribution and the value function by vectors (see Algorithm~\ref{algo:ODE-MFG}). In our implementation, we used an explicit Euler scheme.

\begin{algorithm}
\caption{ODE Approach for the finite-state Mean Field Game\label{algo:ODE-MFG}}
\textbf{Input:} Time horizon $T$; Time Increments $\Delta t$; Initial discretized flow of state distribution and value function $\boldsymbol{p}^{(0)} = \{p_0,p_{\Delta t},p_{2\Delta t}, \dots, p_{T}\}$ and $\boldsymbol{u}^{(0)}   = \{u_0, u_{\Delta t}, u_{2\Delta t}, \dots, u_{T}\}$; initial state distribution $p_{0}$; terminal condition of value function $u_T$; Tolerance $\tau$

\textbf{Output:} Equilibrium discretized flow of state distribution and corresponding value function: $(\hat p_0, \hat p_{\Delta t}, \dots,\hat p_{T})$ and $(\hat u_0, \hat u_{\Delta t}, \dots, \hat u_{T})$ 
\begin{algorithmic}[1]
\STATE{$k \leftarrow 0$}
\WHILE{$||\boldsymbol{p}^{(k)}-\boldsymbol{p}^{(k-1)}||>\tau$ or $||\boldsymbol{u}^{(k)}-\boldsymbol{u}^{(k-1)}||>\tau$}
    \STATE{Compute $\boldsymbol{p}^{(k+1)}$ solving the forward equation in~\eqref{eq:ODE-SIR-MFG} with $\boldsymbol{u}$ replaced by $\boldsymbol{u}^{(k)}$
    }
    \STATE{Compute $\boldsymbol{u}^{(k+1)}$ solving the backward equation in~\eqref{eq:ODE-SIR-MFG} with $\boldsymbol{p}$ replaced by  $\boldsymbol{p}^{(k+1)}$}
    \STATE{
    Update $\boldsymbol{p}^{(k-1)} \leftarrow \boldsymbol{p}^{(k)}$, $\boldsymbol{p}^{(k)}\leftarrow\boldsymbol{p}^{(k+1)}$, $\boldsymbol{u}^{(k-1)}\leftarrow\boldsymbol{u}^{(k)}$, $\boldsymbol{u}^{(k)}\leftarrow\boldsymbol{u}^{(k+1)}$}
    \STATE{$k \leftarrow k+1$}
\ENDWHILE
\RETURN $\boldsymbol{p}^{(k)}$ and $\boldsymbol{u}^{(k)}$
\end{algorithmic}
\end{algorithm}

When we analyze Figure~\ref{fig:LL-ode} and Figure~\ref{fig:LL-NN} (also Figure~\ref{fig:EL-ode} and~\ref{fig:EL-NN}), we see that neural network based method is accurately computes the results obtained by the ODE method.

\begin{figure}[H]
\centering
\begin{subfigure}{.32\textwidth}
    \includegraphics[width=5.1cm,height=3.3cm]{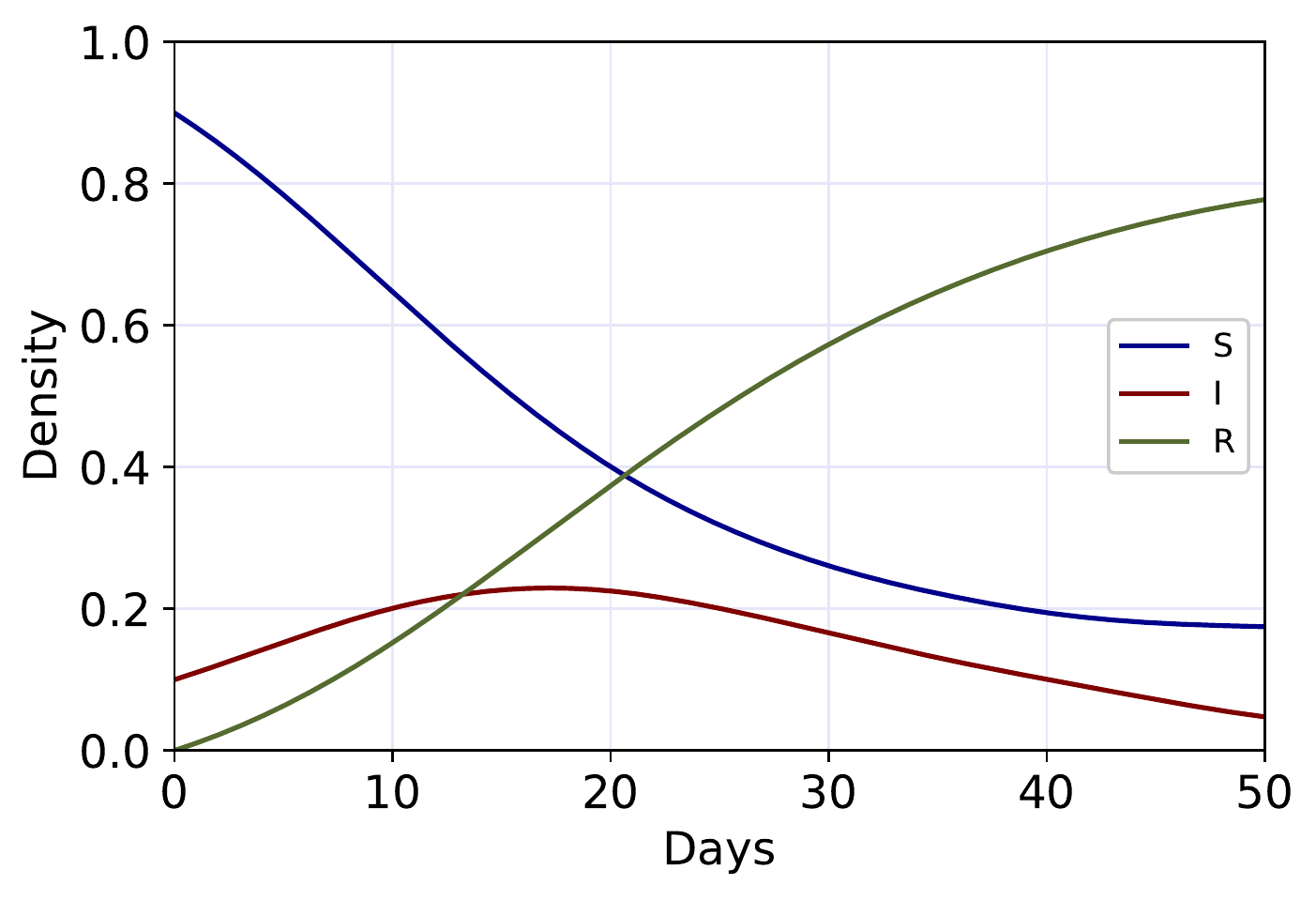}
\end{subfigure}
\hfill
\begin{subfigure}{.32\textwidth}
    \includegraphics[width=5.1cm,height=3.3cm]{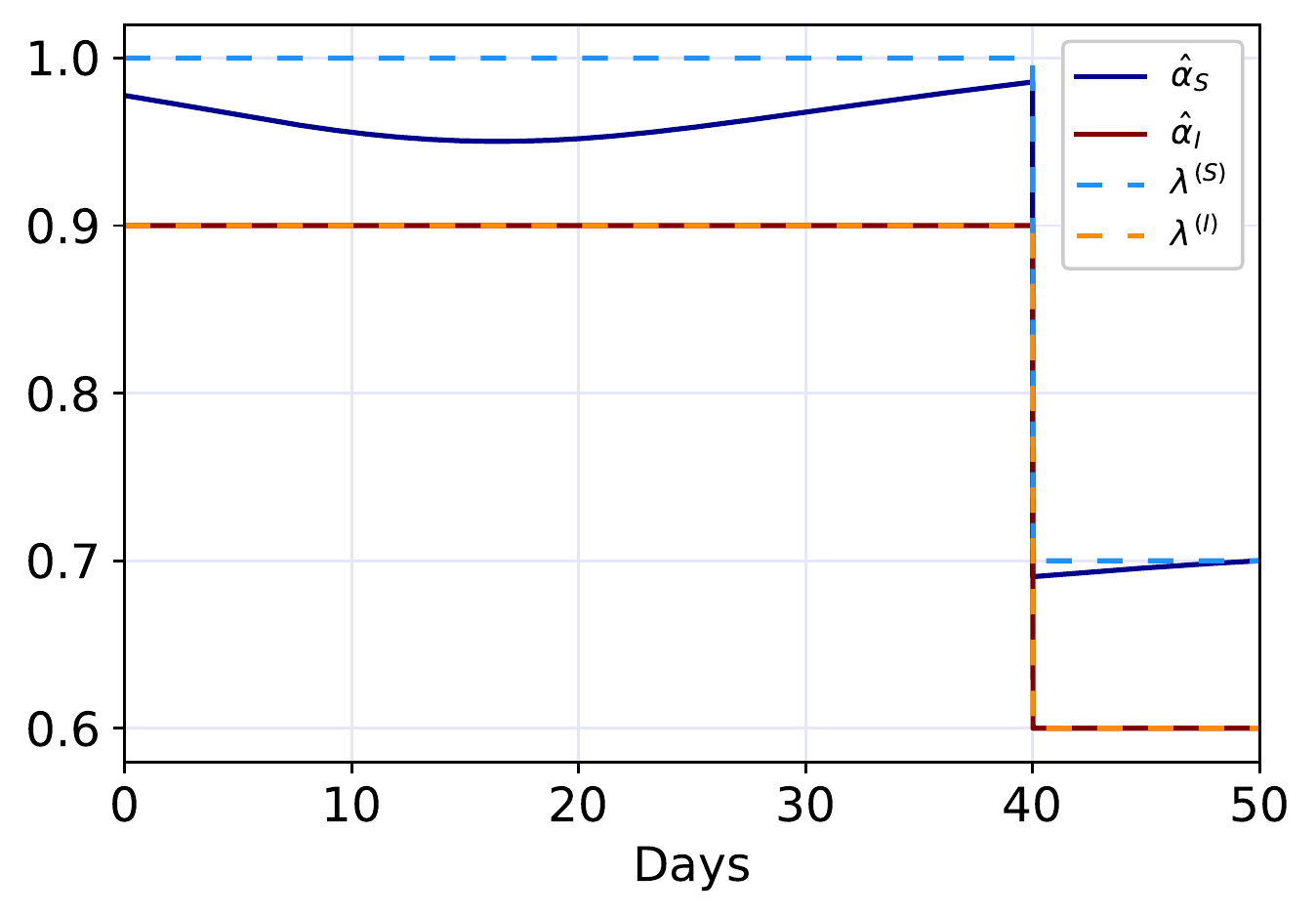}
\end{subfigure}
\hfill
\begin{subfigure}{.32\textwidth}
    \includegraphics[width=5.1cm,height=3.3cm]{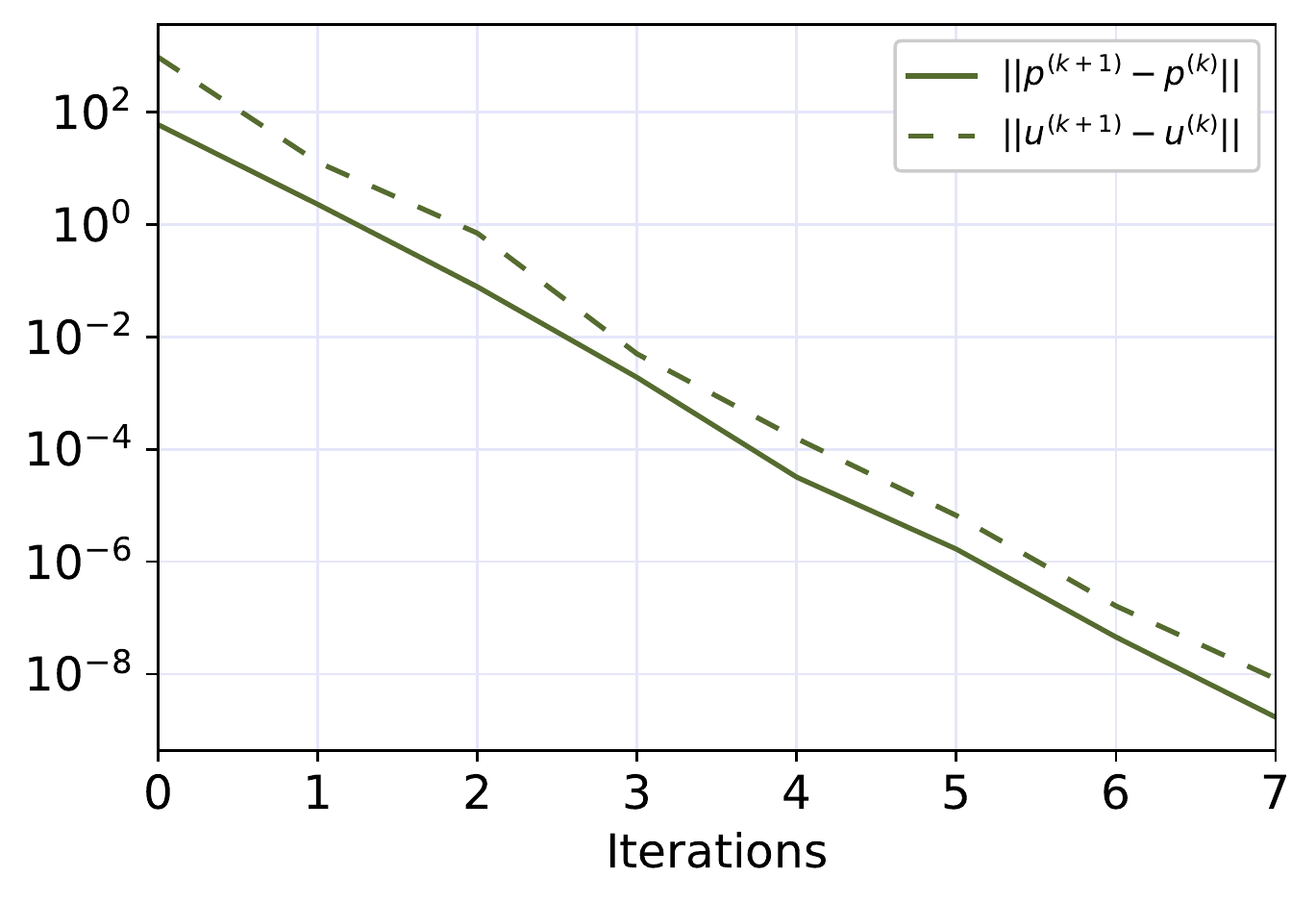}
\end{subfigure}
\caption{Late lockdown with the ODE solver. Evolution of the population state distribution (left), evolution of the controls (middle), convergence of the solver (right).}
\label{fig:LL-ode}
\end{figure}

\begin{figure}[H]
\centering
\begin{subfigure}{.32\textwidth}
    \includegraphics[width=5.5cm,height=3.5cm]{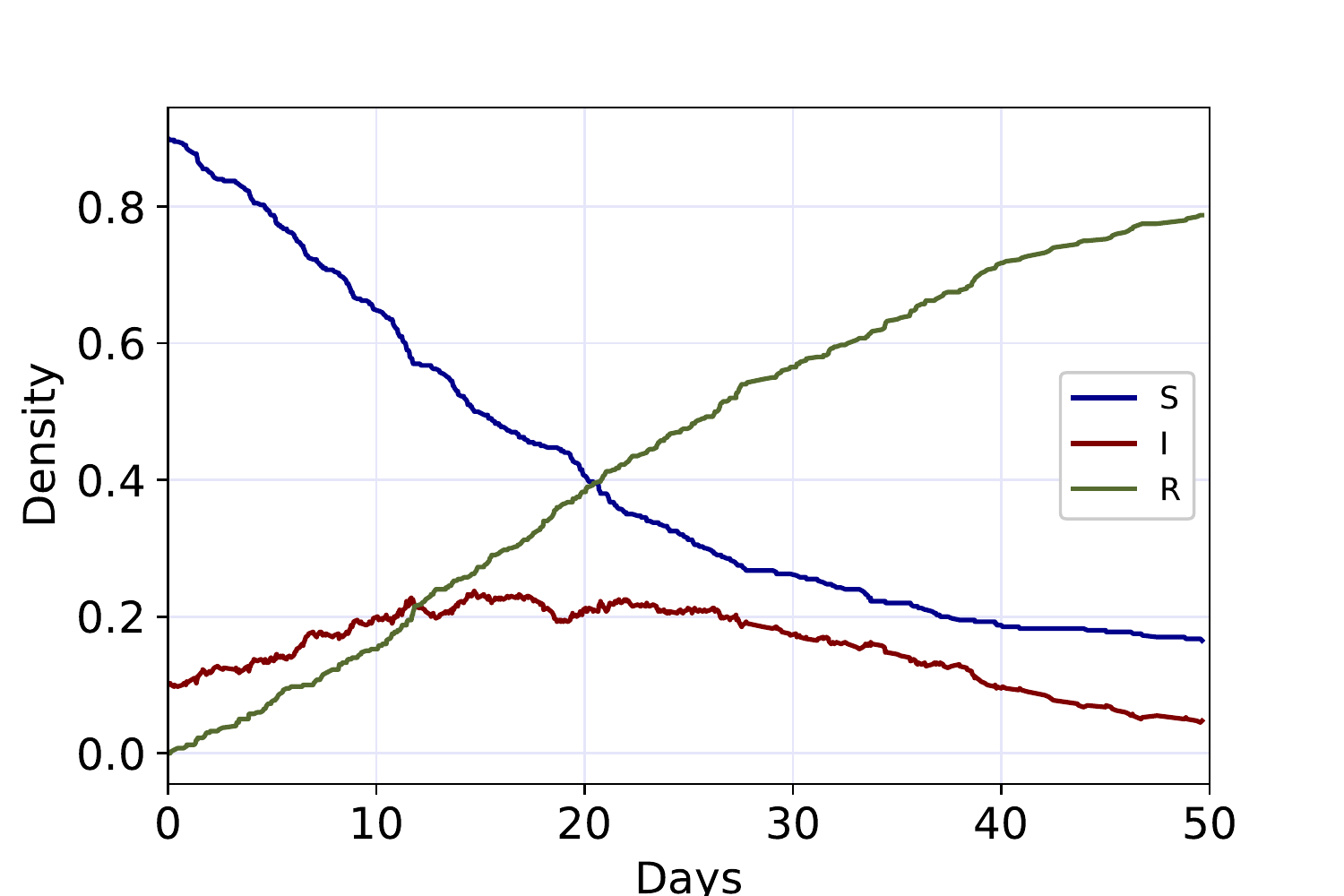}
\end{subfigure}
\hfill
\begin{subfigure}{.32\textwidth}
    \includegraphics[width=5.5cm,height=3.5cm]{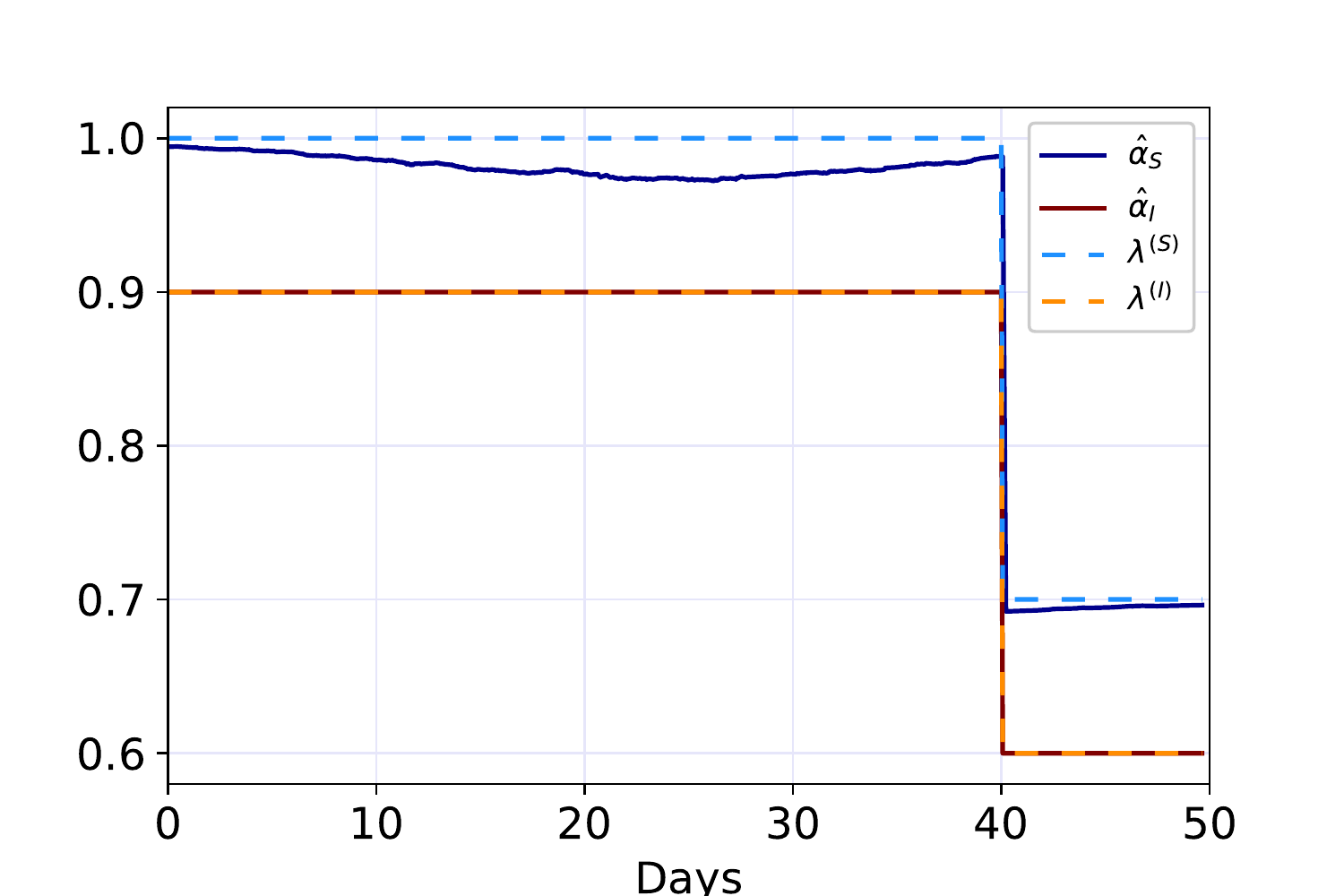}
\end{subfigure}
\hfill
\begin{subfigure}{.32\textwidth}
    \includegraphics[width=5.5cm,height=3.5cm]{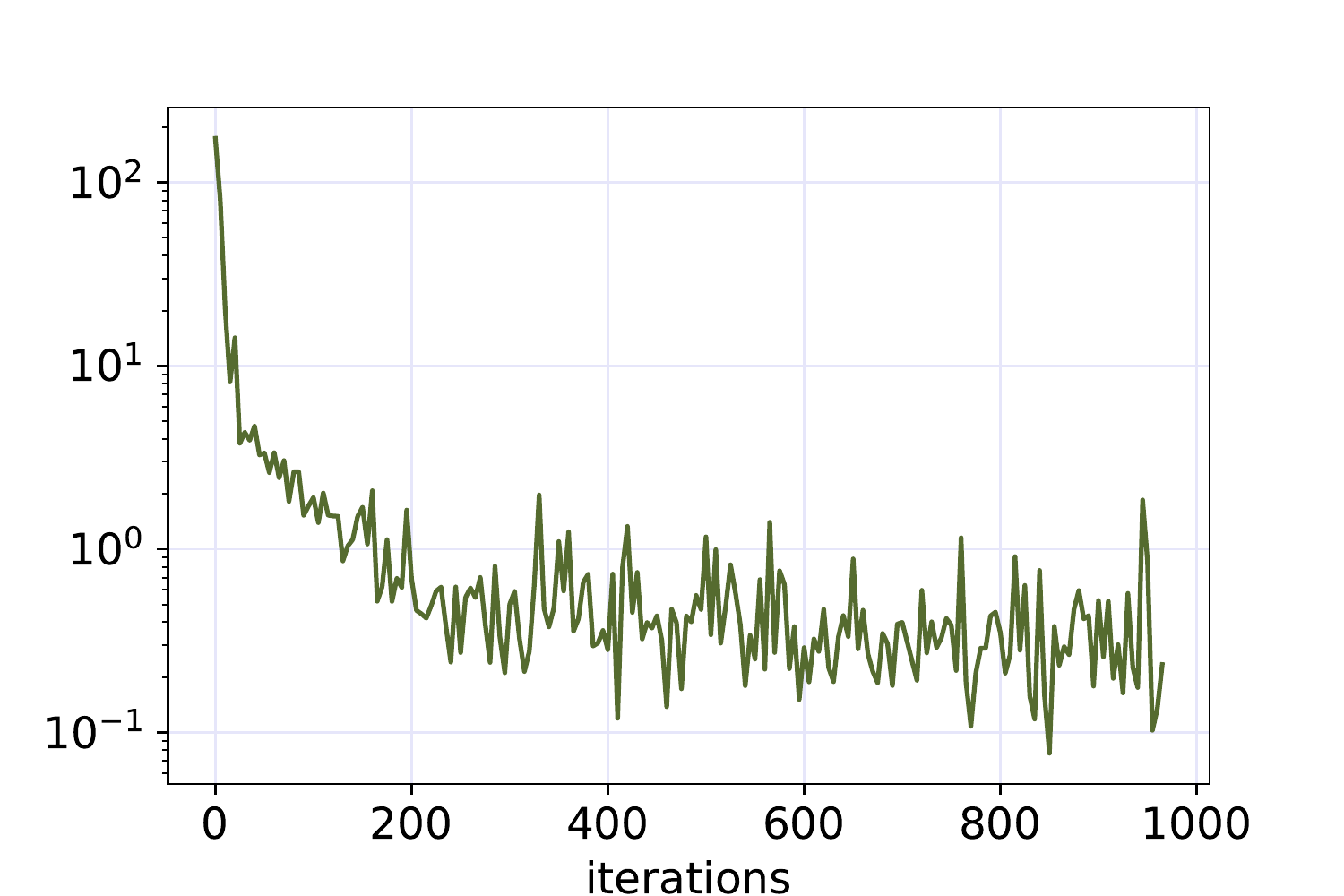}
\end{subfigure}
\caption{Late lockdown with Algorithm~\ref{algo:SGD-SMFG}. Evolution of the population state distribution (left), evolution of the controls (middle), convergence of the loss value (right).}
\label{fig:LL-NN}
\end{figure}

\begin{figure}[H]
\centering
\begin{subfigure}{.32\textwidth}
    \includegraphics[width=5.1cm,height=3.3cm]{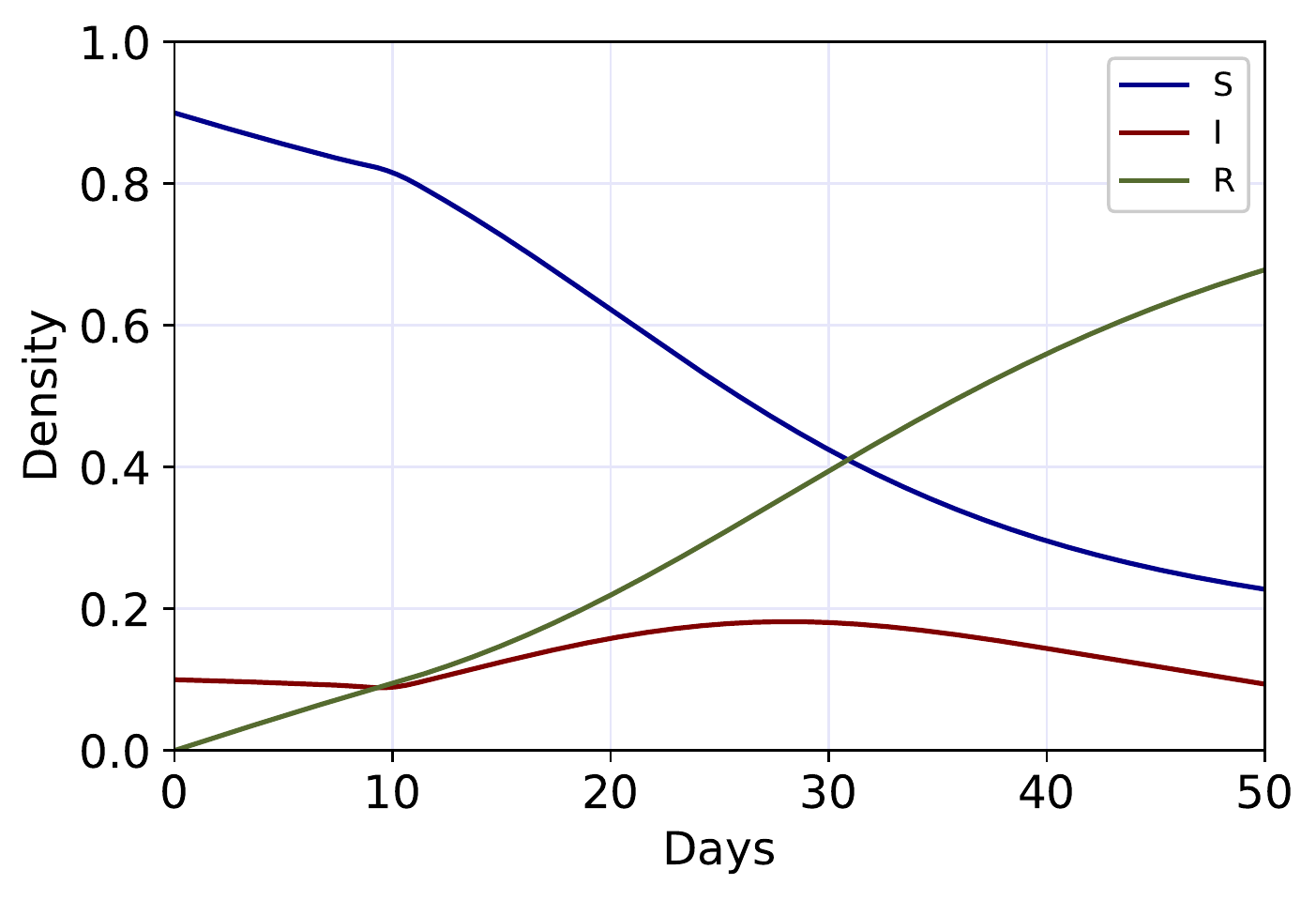}
\end{subfigure}
\hfill
\begin{subfigure}{.32\textwidth}
    \includegraphics[width=5.1cm,height=3.3cm]{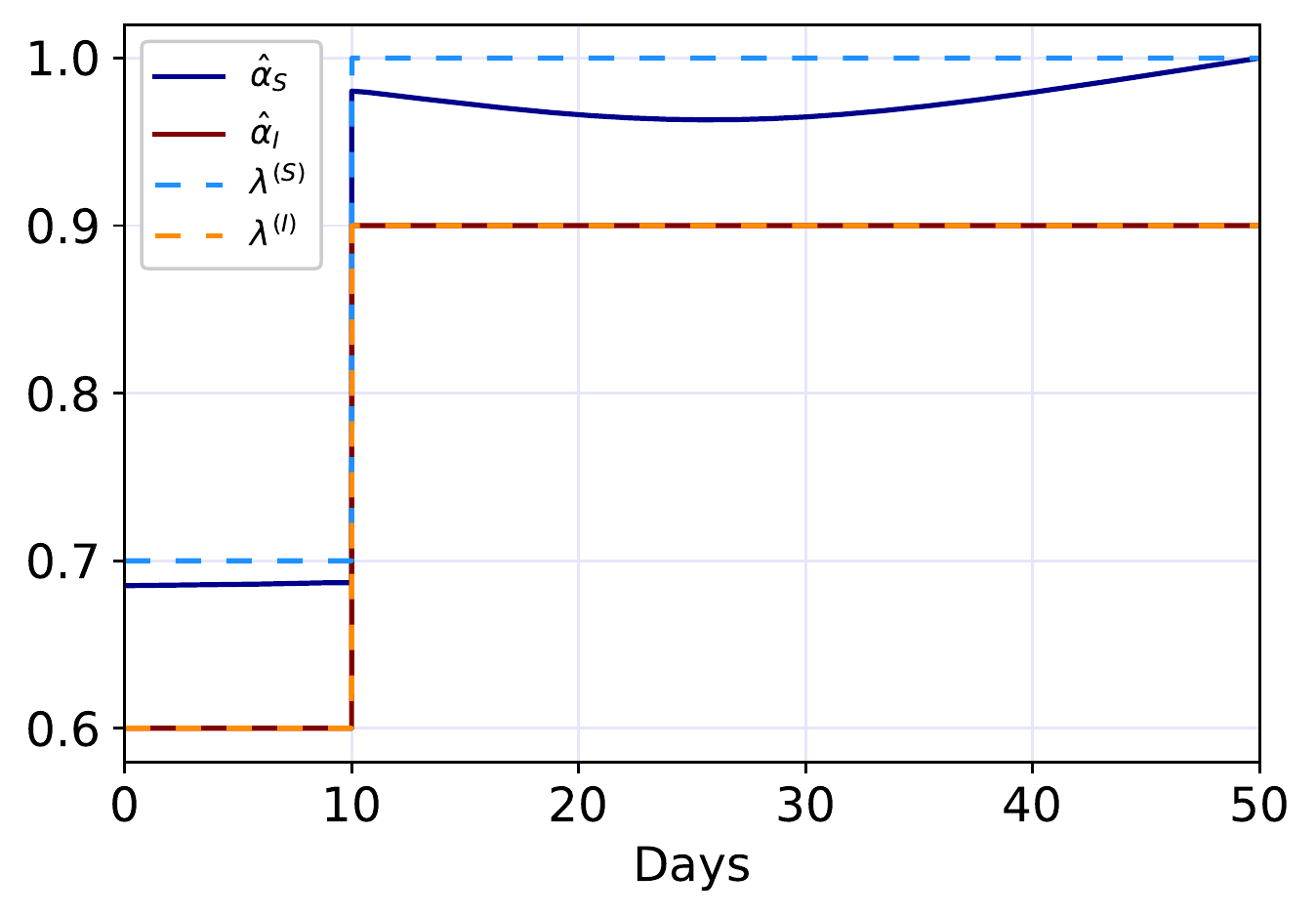}
\end{subfigure}
\hfill
\begin{subfigure}{.32\textwidth}
    \includegraphics[width=5.1cm,height=3.3cm]{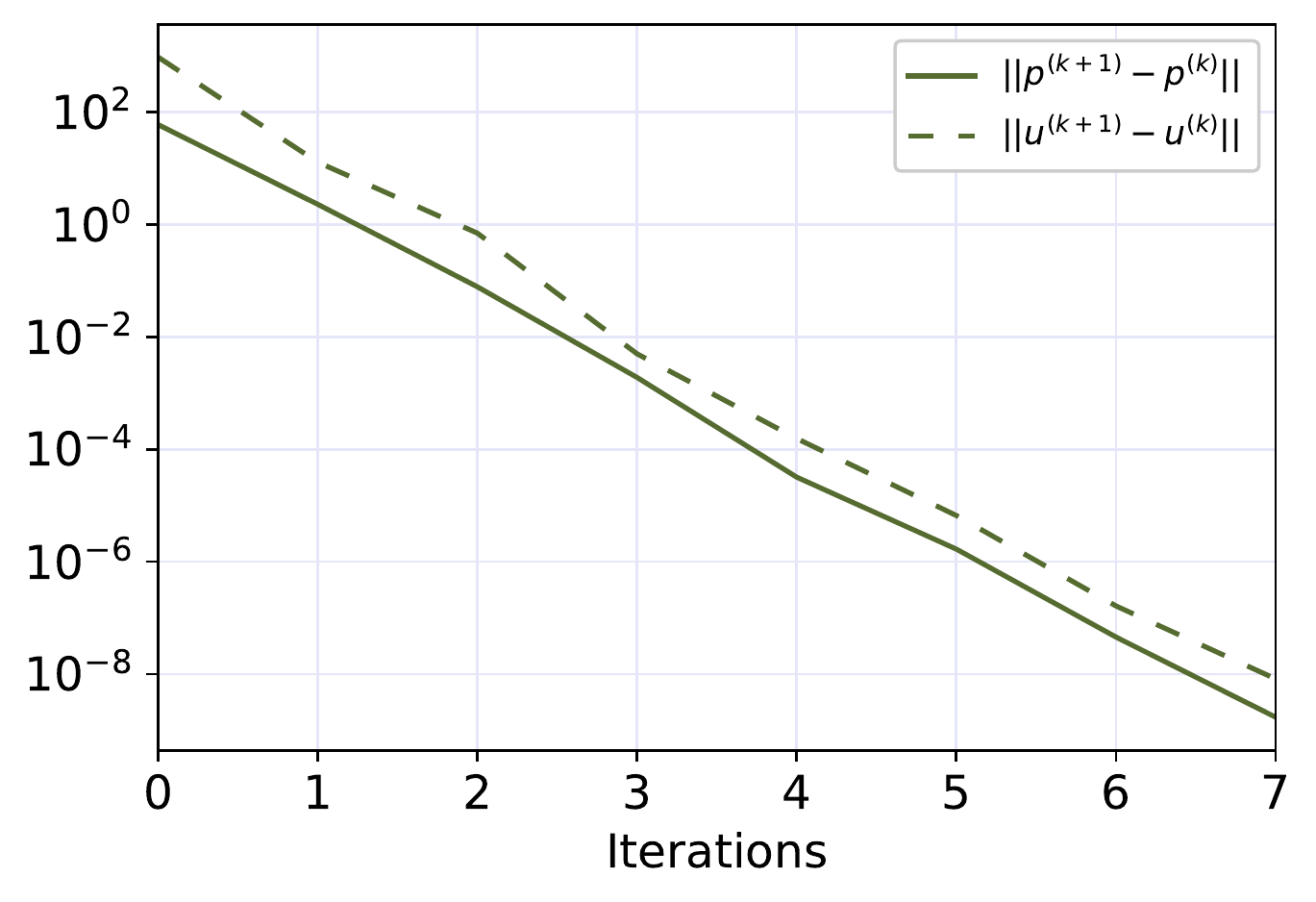}
\end{subfigure}
\caption{Early lockdown with the ODE solver. Evolution of the population state distribution (left), evolution of the controls (middle), convergence of the solver (right).}
\label{fig:EL-ode}
\end{figure}

\begin{figure}[H]
\centering
\begin{subfigure}{.32\textwidth}
    \includegraphics[width=5.5cm,height=3.5cm]{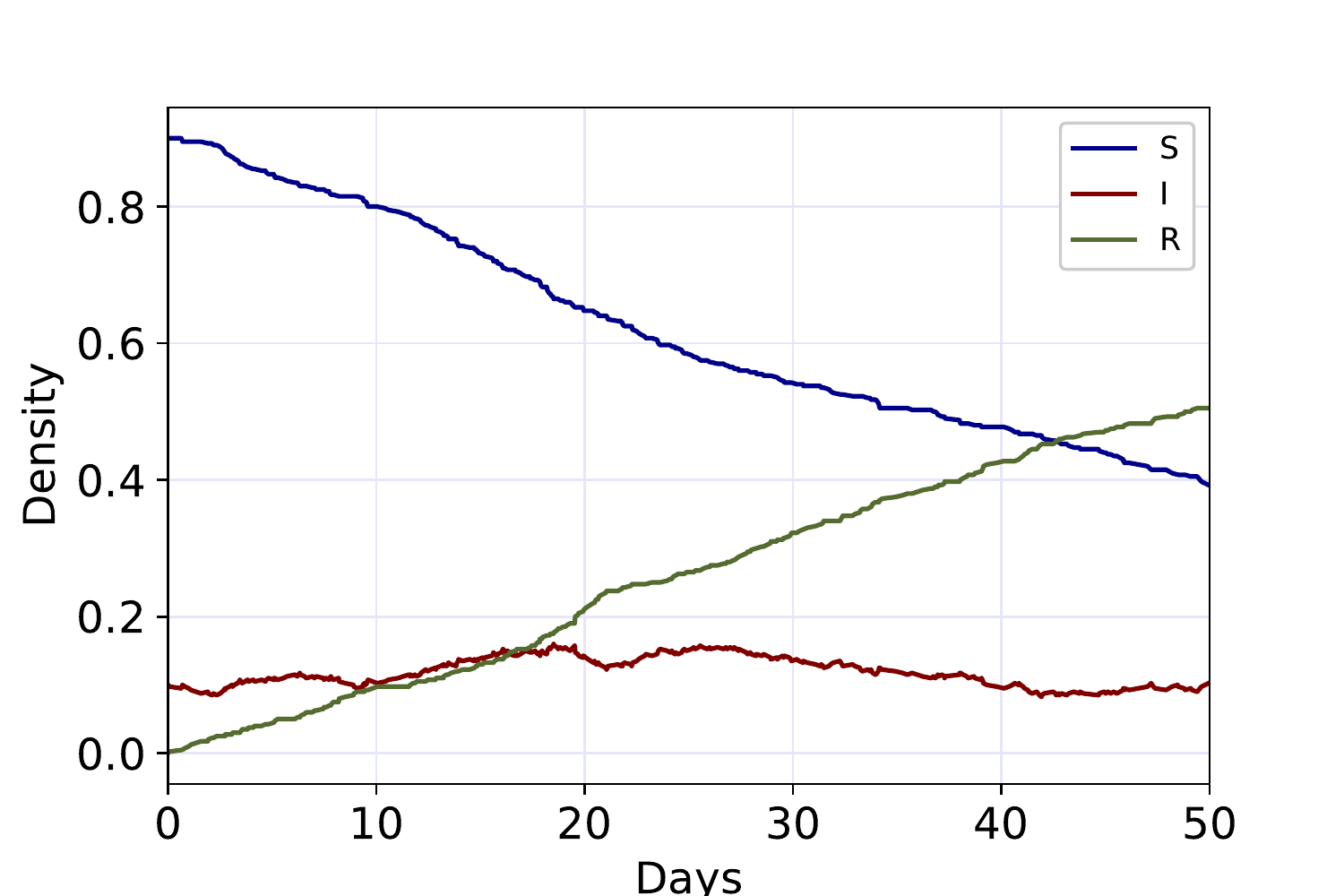}
\end{subfigure}
\hfill
\begin{subfigure}{.32\textwidth}
    \includegraphics[width=5.5cm,height=3.5cm]{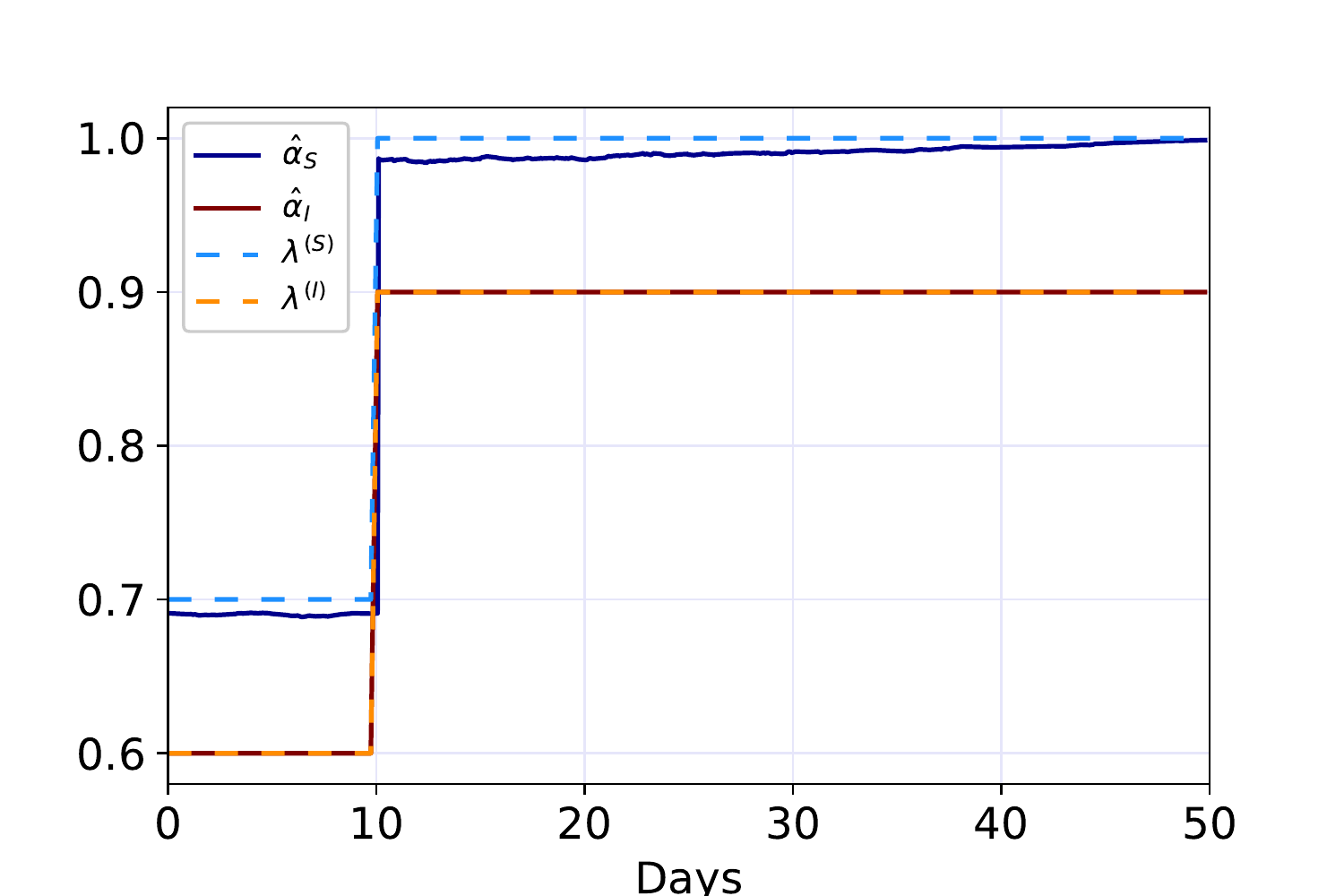}
\end{subfigure}
\hfill
\begin{subfigure}{.32\textwidth}
    \includegraphics[width=5.5cm,height=3.5cm]{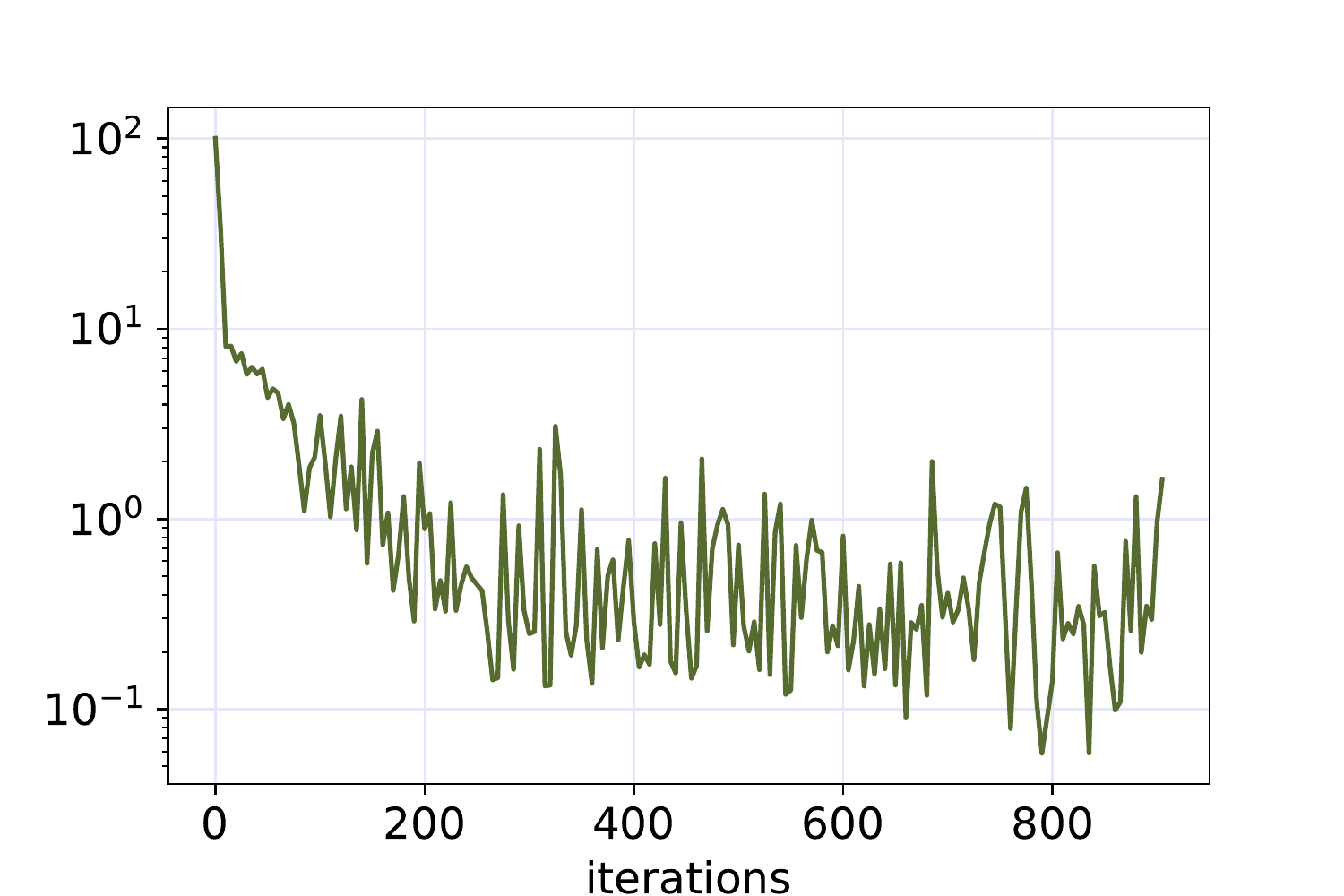}
\end{subfigure}
\caption{Early lockdown with Algorithm~\ref{algo:SGD-SMFG}. Evolution of the population state distribution (left), evolution of the controls (middle), convergence of the loss value (right).}
\label{fig:EL-NN}
\end{figure}

To illustrate the impact of the agents' optimization and the impact of the regulator's choice of policy, we consider four test cases, presented in Table~\ref{table:test1}. The parameters related to the dynamics are chosen according to the following assumptions on COVID-19 pandemic: Since the average recovery duration is around 10 days, the recovery rate is taken as $\gamma = 0.1$ 1/days; furthermore, to our up to date knowledge, the reinfection possibility is lower in the first 3 months and for later times it is uncertain.\footnotemark[\value{footnote}] 
\footnotetext{\url{https://www.cdc.gov/coronavirus/2019-ncov/hcp/duration-isolation.html}} Therefore, in some experiments in this paper, the reinfection rate $\eta$ is taken 0 or 0.01. Finally, since it is hard to estimate infection rate $\beta$ directly, we used the estimates on Basic Reproduction number $(R_0)$ of COVID-19. The CDC uses $R_0=2.5$ in the ``Current Best Estimate" scenario in its simulations.\footnote{\url{https://www.cdc.gov/coronavirus/2019-ncov/hcp/planning-scenarios.html}} Therefore, we use $\beta = R_0 \times \gamma = 0.25$ in our simulations. For the parameters related to the cost function of agents, we choose $c_I=1$ and $c_\lambda=10$ to balance the powers of the different cost terms. Since $c_I$ stands alone while $c_\lambda$ is multiplied with a term likely to be much more smaller than 1, we decide to use a higher $c_{\lambda}$ value. Further, we also rule out extreme cases. For example, if $c_I$ is taken to be too dominant then susceptible people will decide to minimize their contact factor, which leads to unrealistic results. Also, $c_I$ is not taken to be very small $(\approx 0)$, since that means sickness does not come with an added negative effect. We know from our experience during COVID-19 pandemic that being sick has both social and economic burden.

\begin{figure}[H]
\centering
\begin{subfigure}{.42\textwidth}
    \includegraphics[width=1\linewidth]{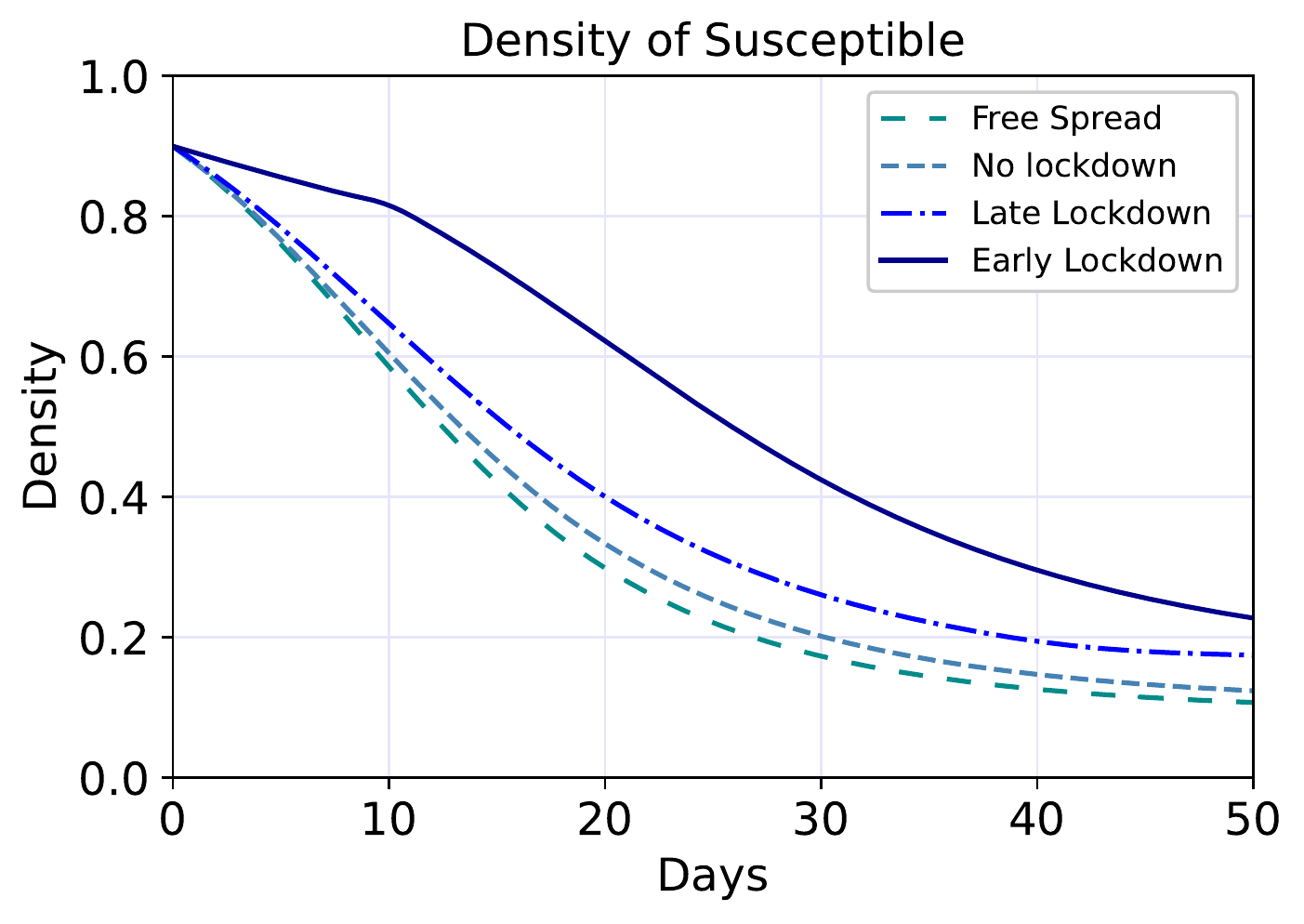}
\end{subfigure}
\begin{subfigure}{.42\textwidth}
    \includegraphics[width=1\linewidth]{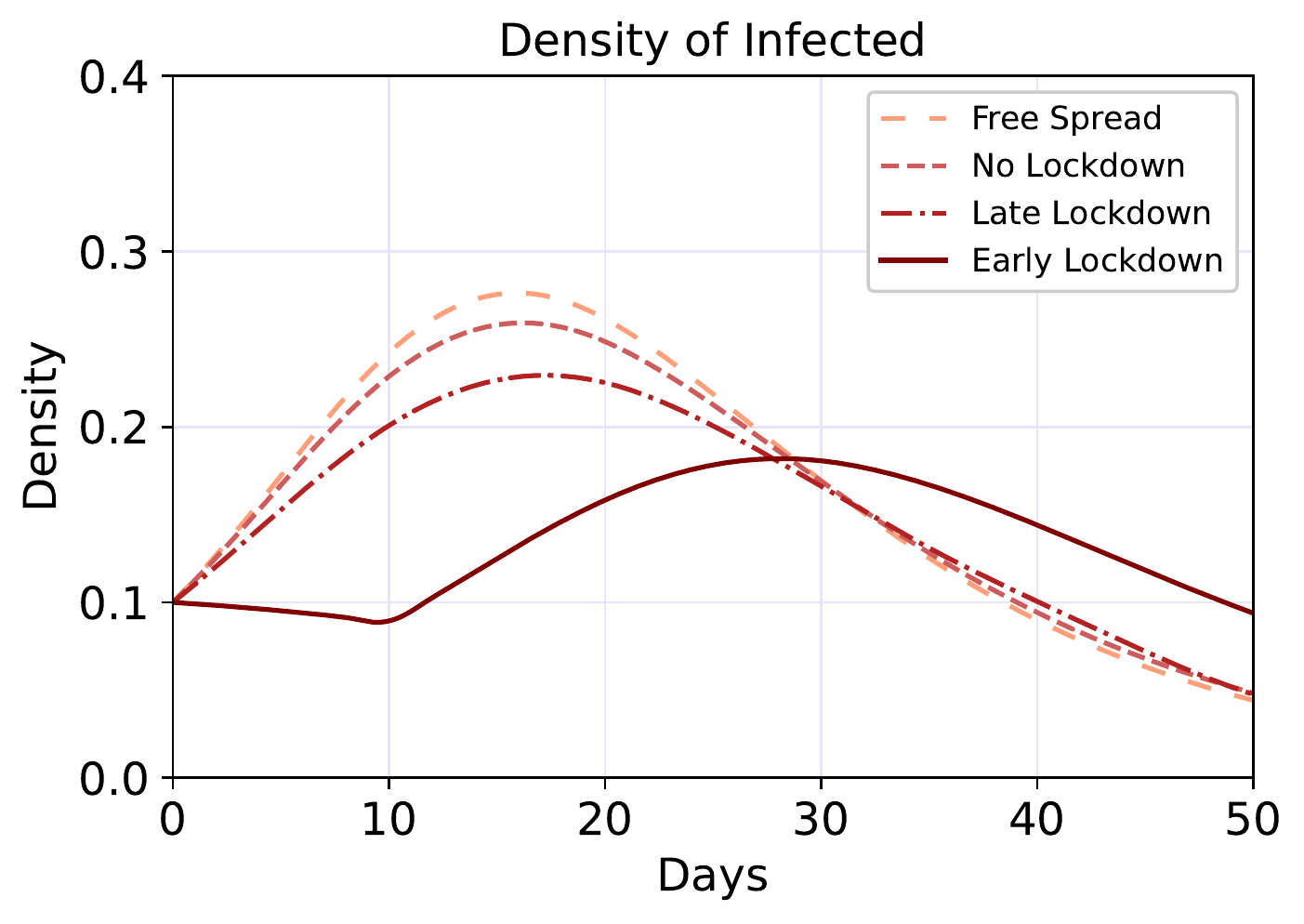}
\end{subfigure}

\begin{subfigure}{.42\textwidth}
    \includegraphics[width=1\linewidth]{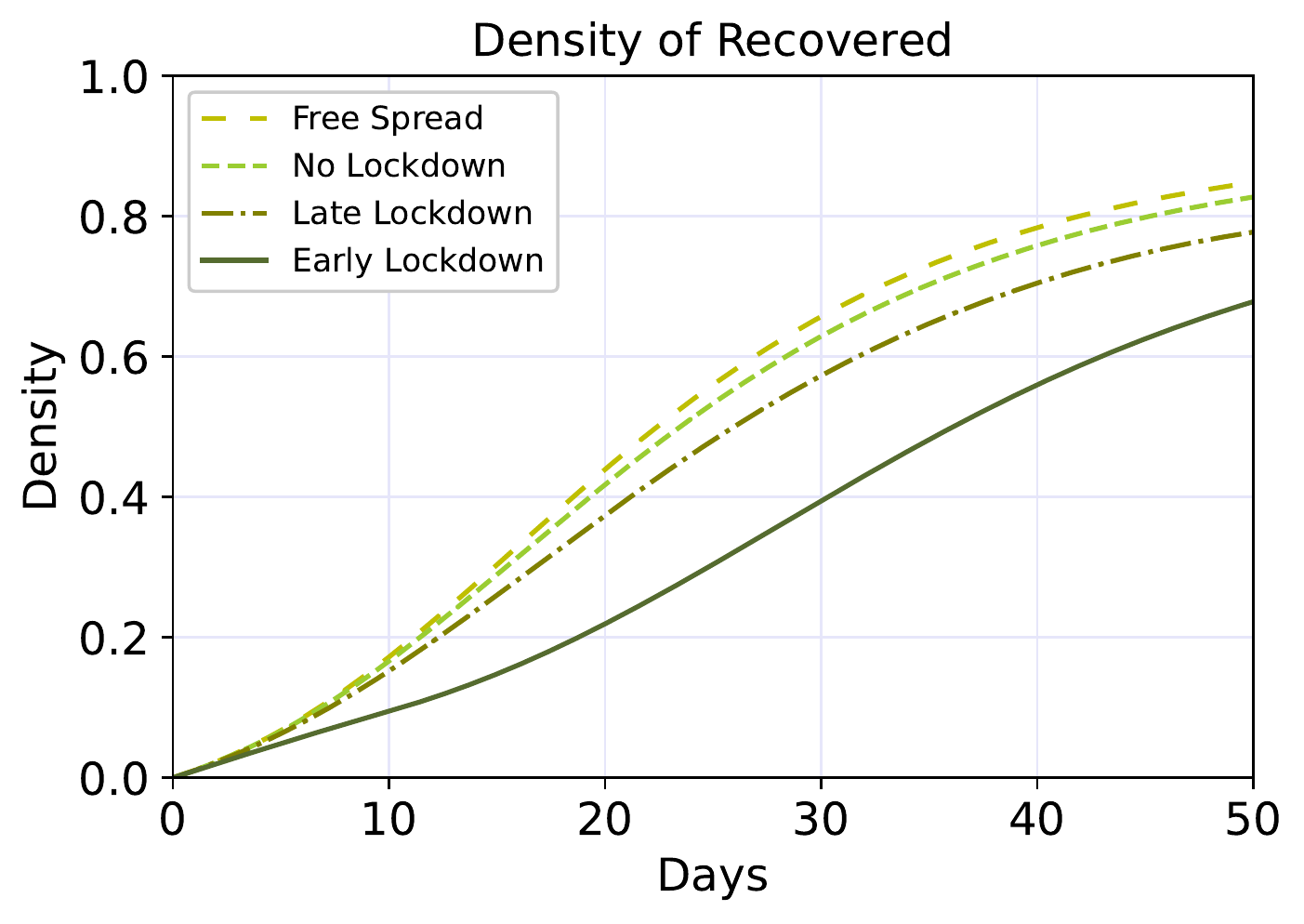}
\end{subfigure}
\begin{subfigure}{.42\textwidth}
    \includegraphics[width=1\linewidth]{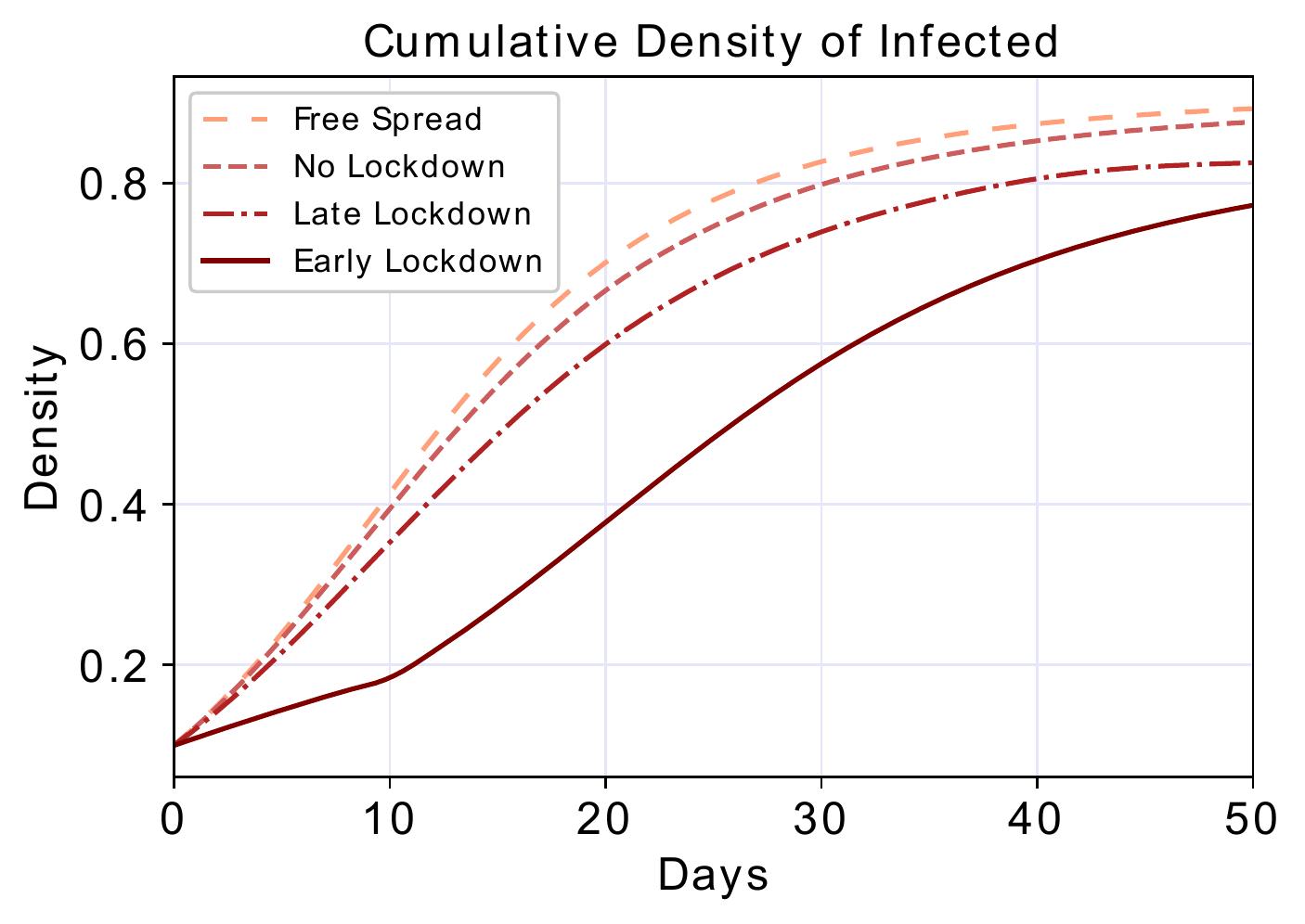}
\end{subfigure}
\caption{Evolution of the population state distribution in the four test cases, obtained using an ODE solver. 
\label{fig:SIR-comparison-evol-pop}}
\end{figure}

Fig.~\ref{fig:SIR-comparison-evol-pop} displays the evolution of the population's state distribution in each of the four test cases. In it, it is worth to note that the proportion of infected is decreasing from first test case to the last one, which can be interpreted in the following way: In the Nash equilibrium the agents take action to reduce the risk of being infected compared with letting the epidemic spread freely; imposing $\boldsymbol\lambda^{(I)} < 1$ encourages the players to be more cautious and hence decreases further infections; last, recommending a $\boldsymbol\lambda^{(I)}$ that is low in the beginning and then relaxing it (early lockdown case) helps avoiding the first peak (around $15$ days) but leads to another peak later (around $28$ days). We can further infer that the cumulative number of infected people can be decreased the most with the early lockdown policy. Therefore, we conclude that early actions are more effective at decreasing the severeness of the disease.

\subsection{A semi-explicitly solvable Stackelberg game}
\label{subsec:SIR-exp-reg}
We now add the regulator's optimization to the previous example, making it a Stackelberg game. The regulator pays a cost that is increasing with the number of infections and any deviation of the issued policy $\boldsymbol\lambda = (\lambda^{(S)}_t, \lambda^{(I)}_t, \lambda^{(R)}_t)_{t\in[0,T]}$ from some endogenously recommended levels, $\bar{\boldsymbol{\lambda}}$. There are multiple ways we can think of the latter: as levels recommended by the health authorities such as C.D.C.; an average of what other regulators are doing (\textit{e.g.}, other countries' regulations)\footnote{A scenario with multiple competing regulators is a highly interesting and relevant problem. During the past year, we have seen such opposition between US states, and EU member states. We imagine that in such a model we would sometimes see alignment of recommendations and other times specialization. However, the case is beyond the scope of this paper. }; or budgetary constraints. We will take on the viewpoint that $\boldsymbol{\bar\lambda}$ is a health authority recommendation to the regulating body that is ``the government".

Turning to the specifics, we set $C_0(p) = 0$ and we assume that the government minimizes the proportion of the infected people over time and tries to set socialization levels close to the levels recommended by the health authorities:
\begin{equation}
    c_0(t,p) = c_{\mathrm{Inf}} \, p(I)^2,\quad f_0(t,\lambda) = \sum_{i\in\{S,I,R\}}\frac{\bar{\beta}^{(i)}}{2}\left(\lambda^{(i)} - \bar{\lambda}^{(i)}\right)^2
\end{equation}
for constant $\bar{\lambda},\bar{\beta} \in \mathbb{R}^m_+$ and $c_{\mathrm{Inf}}>0$. The next proposition provides a semi-explicit construction of the optimal contract and the mean-field Nash equilibrium in this case. A proof is presented in appendix~\ref{app:SIR-derivation}.
\begin{proposition}
\label{prop:SIR-example}
Consider the Stackelberg game of this section. Let $\widetilde H : [0,T]\times \mathcal{P}(E) \times \mathbb{R}^m \times A \times \mathbb{R}_+^m \rightarrow \mathbb{R}$ be defined by:
\begin{equation}
\label{eq:Hamiltonian-tilde-W}
    \begin{aligned}
        \widetilde{H}(t,\pi,y,\widetilde\alpha,\lambda)
        &:=
        \left(y(I) - y(S)\right)\beta\lambda^{(I)}\alpha\pi(I)\pi(S)
        + 
        \left(y(R) - y(I)\right)\gamma 
        \\
        &\hspace{1cm}
        + \left(y(S)-y(R)\right)\eta 
        +
        c_0(t,\pi) + f_0(t,\lambda) 
        \\
        &\hspace{1cm}
        + \frac{c_\lambda}{2}\left(\lambda^{(S)}-\widetilde\alpha\right)^2\pi(S) + c_I\pi(I).
    \end{aligned}
\end{equation}
Let $(\hat{\alpha}(\pi_t,y_t),\hat{\lambda}^{(S)}(\pi_t,y_t), \hat{\lambda}^{(I)}(\pi_t,y_t), \hat{\lambda}^{(R)}(\pi_t,y_t))$ be the solution to 
\begin{equation}
\label{eq:example-optimal-controls}
    (\nabla_{\widetilde\alpha},\nabla_\lambda)\widetilde H(t,\pi_t, y_t, \hat{\alpha}(\pi_t,y_t), \hat{\lambda}(\pi_t,y_t)) = 0,
\end{equation}
assumed to exist uniquely and be admissible, 
and $(\boldsymbol \pi, \boldsymbol y)$ solves 
\begin{equation}
\label{eq:fbode}
\begin{aligned}
    \dot{\pi}_t 
    &= 
    \nabla_y \widetilde H(t,\pi_t,y_t,\hat\alpha(\pi_t,y_t),\hat\lambda(\pi_t,y_t),\quad \pi_0 = p_0,
    \\
    \dot{y}_t 
    &= 
    -\nabla_\pi\widetilde{H}(t, \pi_t, y_t, \hat{\alpha}(\pi_t,y_t),\hat\lambda(\pi_t,y_t)),\quad y_T = 0.
\end{aligned}
\end{equation}

Denote by $(\hat\pi,\hat y)$ the solution to \eqref{eq:fbode} and let us define the processes $\boldsymbol{\hat\alpha}\in\mathbb{A}$ and $\boldsymbol{\hat Z} \in \mathcal{H}^2_X$ by:
\begin{equation}
\begin{aligned}
    \hat{\alpha}_t &= \hat{\alpha}(\hat\pi_t,\hat y_t)\mathbbm{1}_S(X_{t-}) + \hat{\lambda}^{(I)}(\hat\pi_t,\hat y_t)\mathbbm{1}_I(X_{t-}) + \hat\lambda^{(R)}\mathbbm{1}_R(X_{t-})
    \\
    \hat{Z}_t &= \left(\frac{c_\lambda\left(\hat{\alpha}(\hat\pi_t,\hat y_t) - \hat\lambda^{(S)}(\hat\pi_t,\hat y_t)\right)}{\beta\hat\lambda^{(I)}(\hat\pi_t,\hat y_t) \hat \pi_t(I)}\mathbbm{1}_S(X_{t-})\mathbbm{1}(\lambda^{(I)}>0),0,0\right).
\end{aligned}
\end{equation}
Now let $y^0\in\mathbb{R}^m$ be such that $p_0^*y^0\leq \kappa$. We then define the random variable $\hat{\xi}$ almost surely by the Stieltjes integral
\begin{equation}
    \hat{\xi} := - X^*_0 y_0 + \int_0^T \left(f(t, X_{t-}, \hat{\alpha}_t, \hat \pi_t) + X^*_{t-}\bar{Q}(t,\hat{\alpha}_t,\hat\pi_t)\hat{Z}_t\right)dt - \int_0^T \hat{Z}^*_t dX_{t-}.
\end{equation}
Then $(\boldsymbol{\hat\lambda}, \hat{\xi})$ is an optimal contract. Moreover, under the optimal contract, every agent adopts the strategy where they pick the control $\hat{\alpha}_t$ and the flow of distribution of agents' states is $\boldsymbol{\hat\pi}$.
\end{proposition}

The next numerical experiment is a comparison of Algorithm~\ref{algo:SGD-SMFG} and the semi-explicit solution of Proposition~\ref{prop:SIR-example}. The results are presented in Fig.~\ref{fig:stackelberg-exp-NN}--\ref{fig:stackelberg-exp-ODE} and the parameters used in the simulation are found in Table~\ref{tab:second-exp}. The additional parameters are chosen as follows: For $\bar \lambda$, we assumed that health authorities recommend a stricter policy for the infected people than susceptible and recovered people; therefore $\bar{\lambda}^{(I)}$ is equal to 0.7, while $\bar{\lambda}^{(S)}$ and $\bar{\lambda}^{(R)}$ are equal to 1. Further, for the government it is more important to follow the guidelines for infected people and if there is no reinfection as in this experiment here it is not important to follow the guidelines for the recovered people; therefore, $\bar{\beta}^{(i)}$ are taken $0.2$, $1$ and $0$, respectively for $i \in \{S, I, R\}$. The value of the loss in the numerical scheme of Algorithm~\ref{algo:SGD-SMFG} converges to the optimal value from the semi-explicit solution. Furthermore, the population dynamics are very similar in both solutions. As for the controls, we see that in both cases, the agents tend to follow closely the regulator's policy. The policies $\boldsymbol\lambda^{(S)}, \boldsymbol\lambda^{(I)}$ output by Algorithm~\ref{algo:SGD-SMFG} seem to capture the average value of the policies given by the semi-explicit solution. Since the loss value is very close to the optimal one, we deduce that these policies are approximately optimal.

\begin{table}[h!]
\caption{Parameter values for the SIR Stackelberg MFG experiment.}
\centering
\ra{1.4}
\begin{tabular}{@{}ccccccccccc@{}}
\toprule
 $T$ & $p^0$ & $c_\lambda$ & $c_I$ & $c_{\mathrm{Inf}}$ & $\bar\beta$ & $\bar\lambda$ &
 $\beta$ & $\gamma$ & $\eta$ & $\kappa$
\\
\midrule
$30$ & $(0.9,0.1,0)$ & $10$ & $0.5$ & $1$ & $(0.2,1,0)$ & $(1,0.7,0)$ & $0.25$ & $0.1$ & $0$ & $0$
\\
\bottomrule
\end{tabular}
\label{tab:second-exp}
\end{table}

\begin{figure}[H]
\centering
\begin{subfigure}{.32\textwidth}
    \includegraphics[width=1\linewidth]{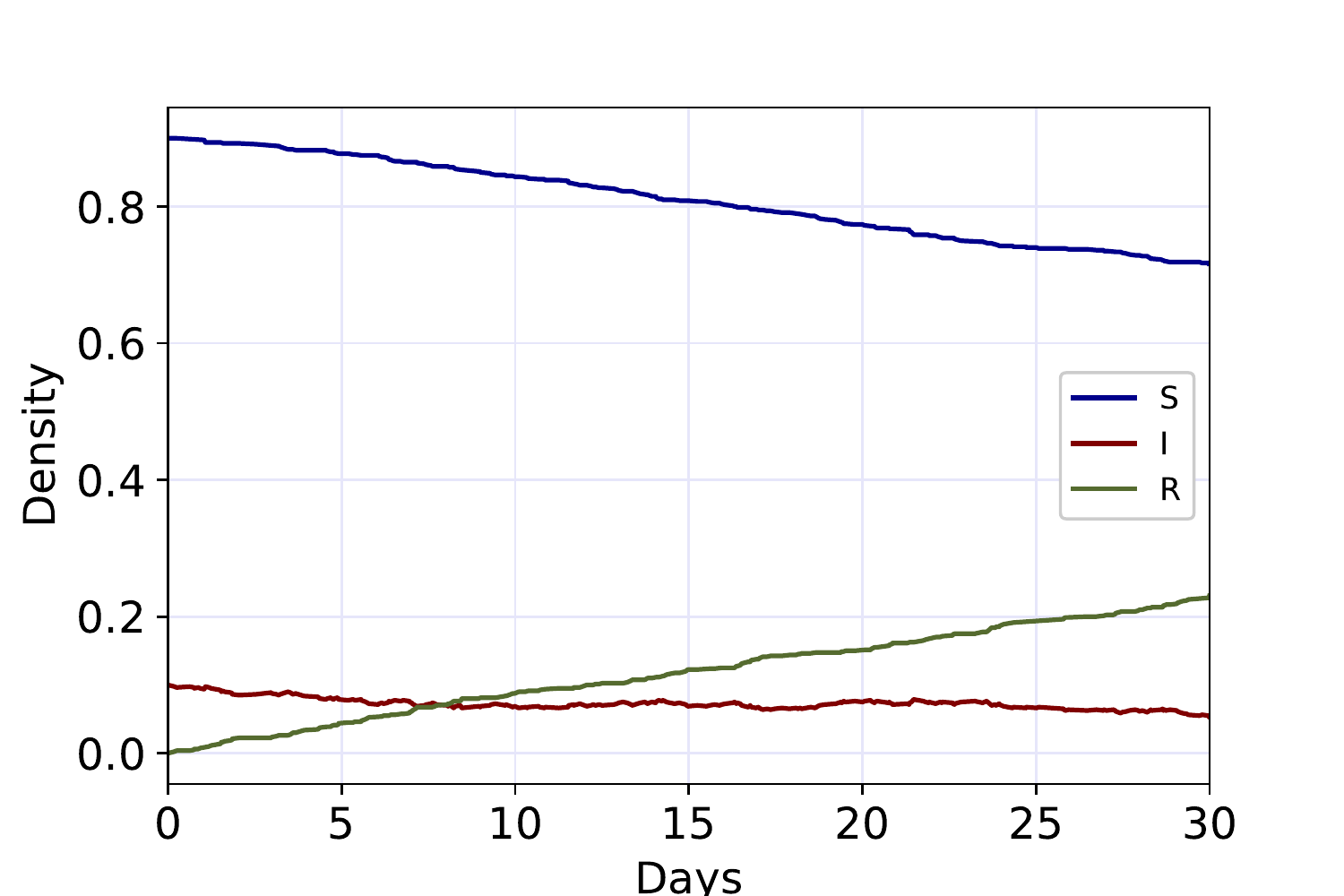}
\end{subfigure}
\hfill
\begin{subfigure}{.32\textwidth}
    \includegraphics[width=1\linewidth]{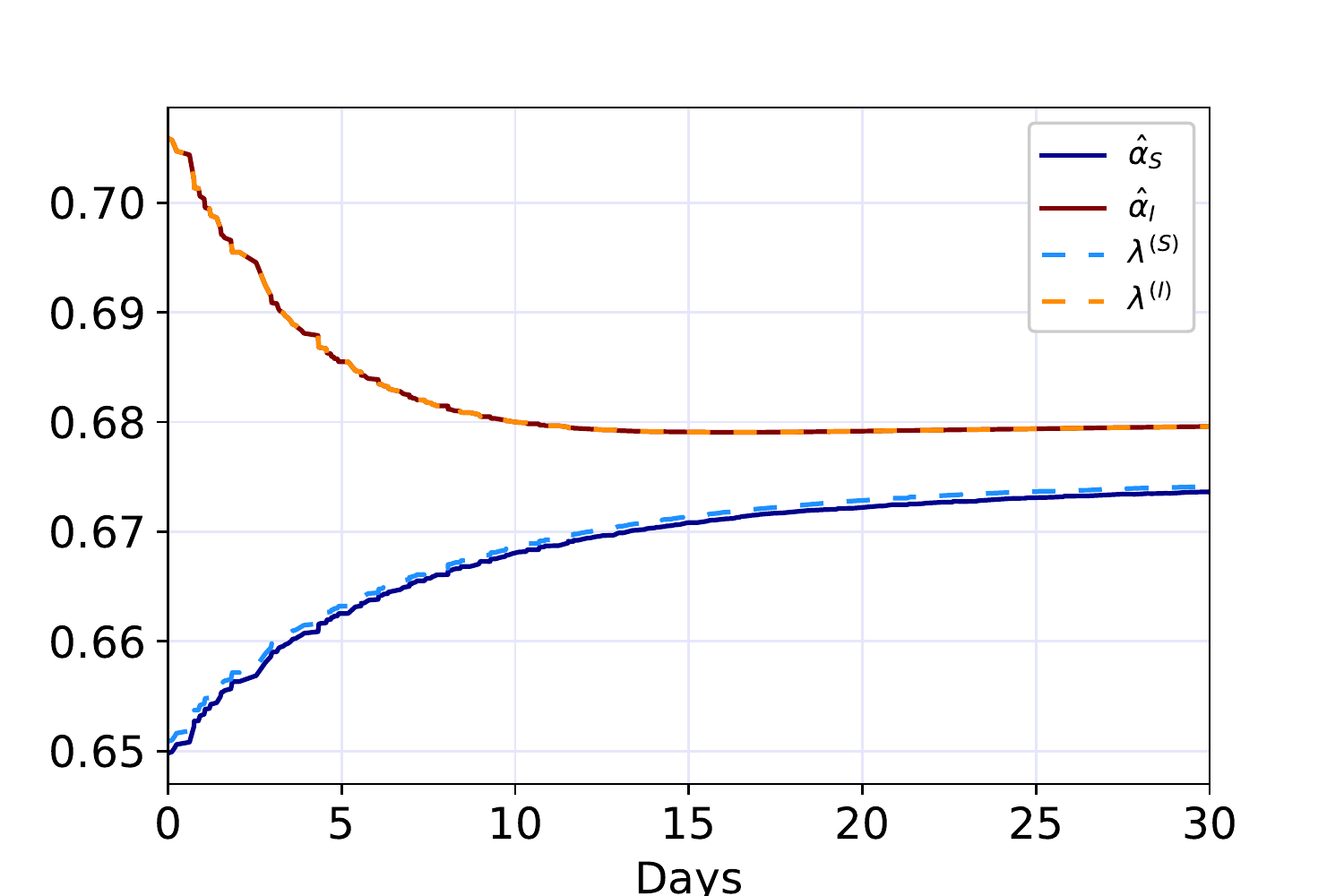}
\end{subfigure}
\hfill
\begin{subfigure}{.32\textwidth}
    \includegraphics[width=1\linewidth]{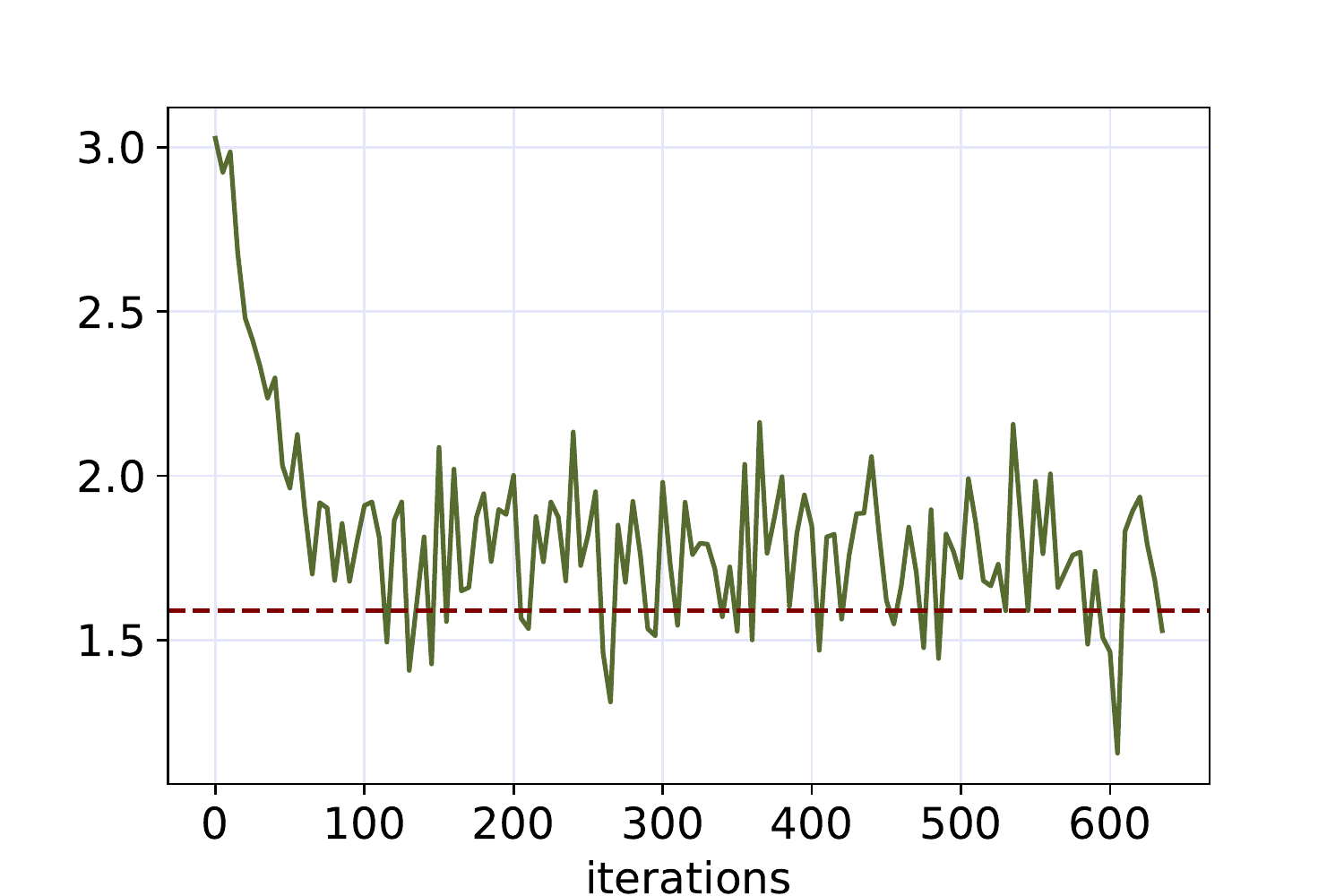}
\end{subfigure}
\caption{SIR Stackelberg mean field game with Algorithm~\ref{algo:SGD-SMFG}. Evolution of the population state distribution (left), evolution of the controls (middle), convergence of the loss value (right). Here, green line refers to the loss function found by the neural network based approach and the red line shows the optimal loss value.}
\label{fig:stackelberg-exp-NN}
\end{figure}

\begin{figure}[H]
\centering
\begin{subfigure}{.32\textwidth}
    \includegraphics[width=1\linewidth]{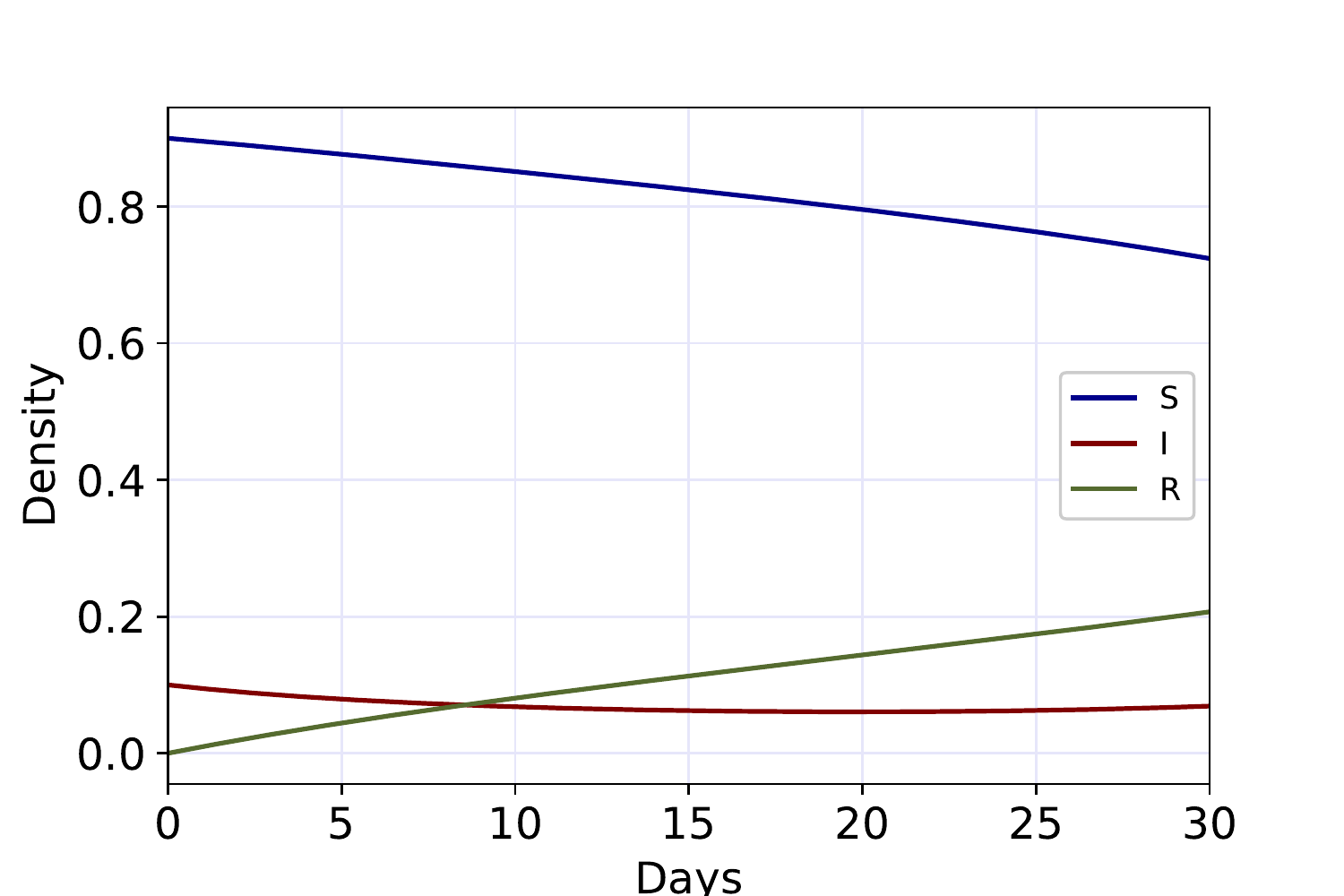}
\end{subfigure}
\hfill
\begin{subfigure}{.32\textwidth}
    \includegraphics[width=1\linewidth]{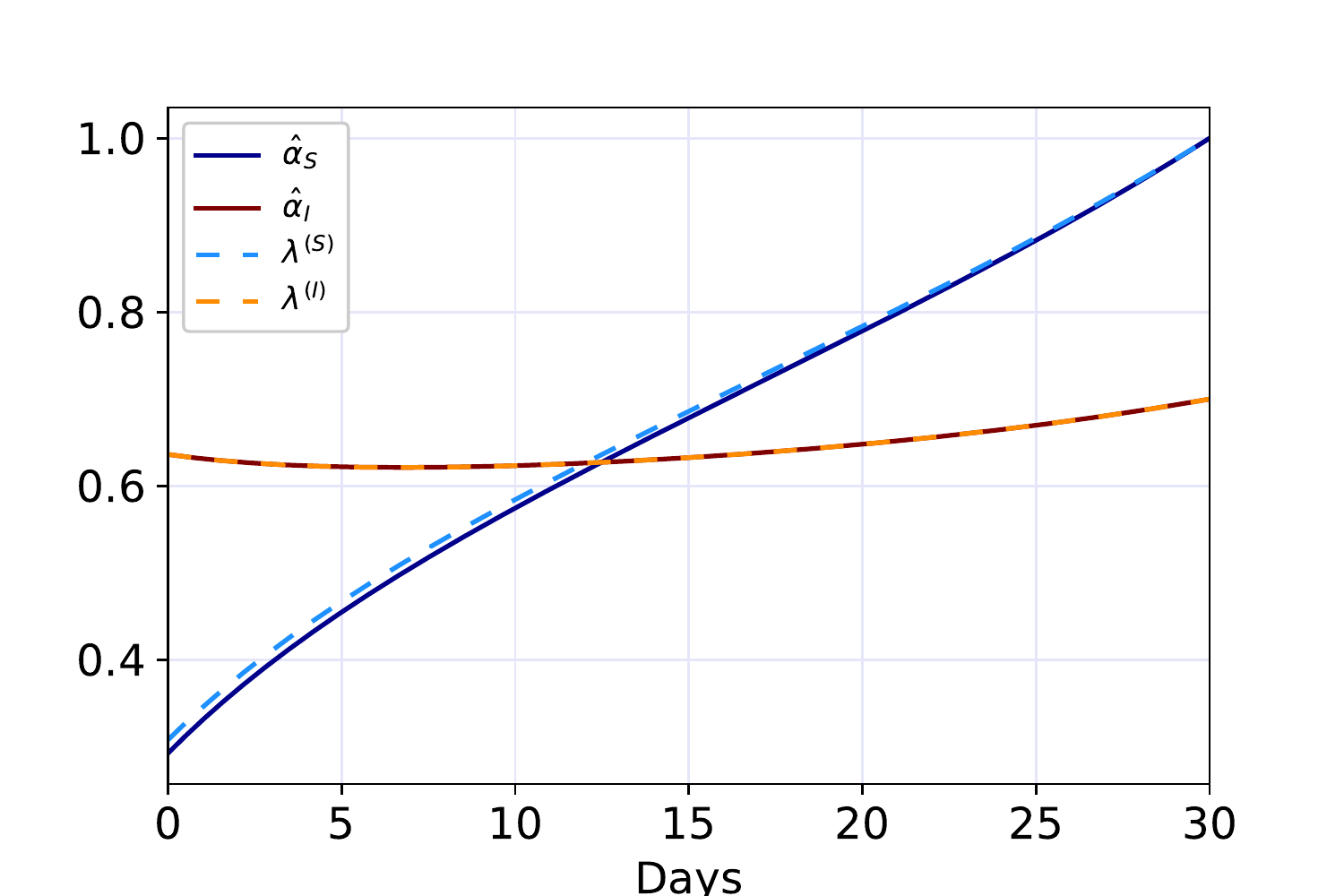}
\end{subfigure}
\hfill
\begin{subfigure}{.32\textwidth}
    \includegraphics[width=1\linewidth]{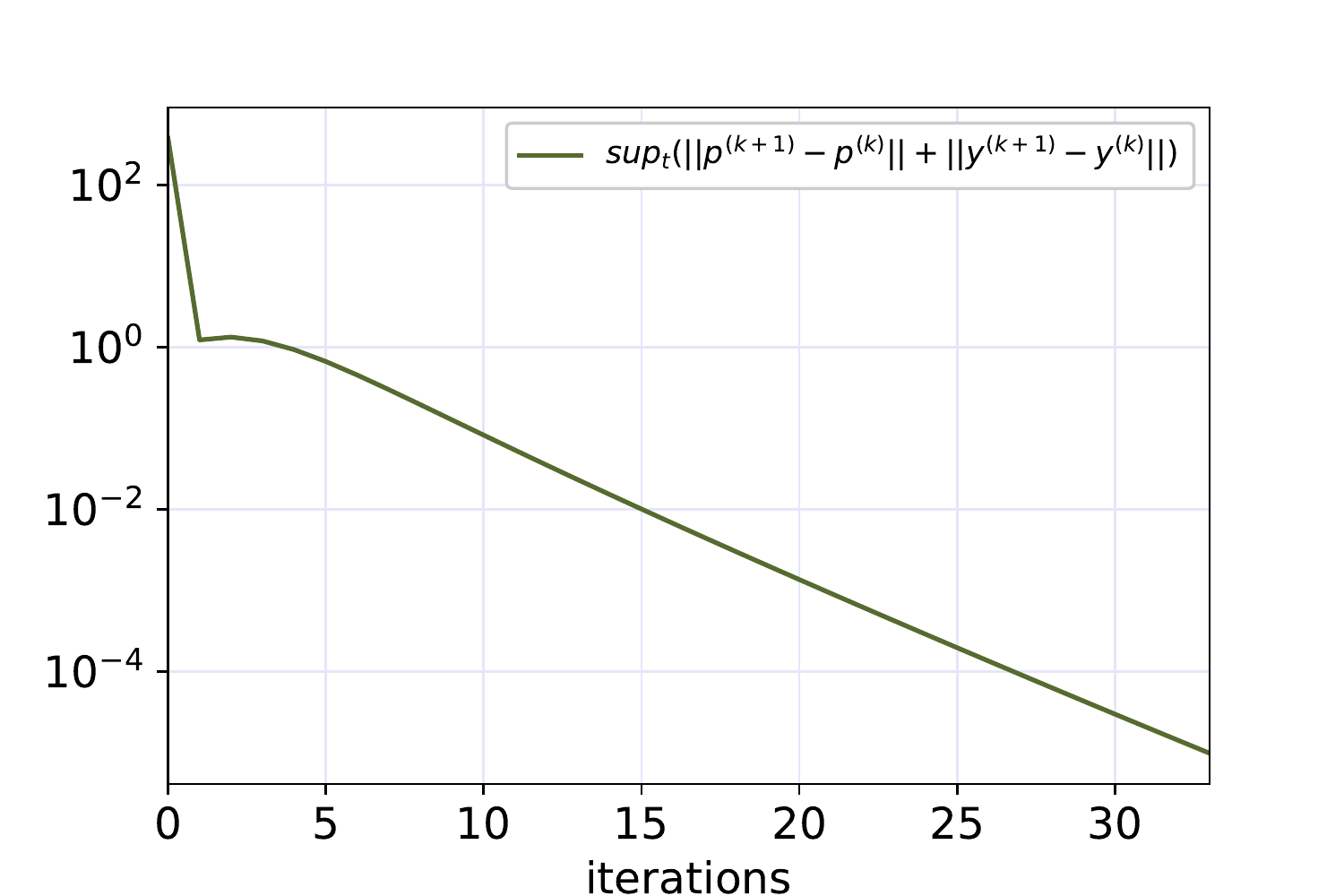}
\end{subfigure}
\caption{SIR Stackelberg mean field game with ODE solver. Evolution of the population state distribution (left), evolution of the controls (middle), convergence of the solver (right). }
\label{fig:stackelberg-exp-ODE}
\end{figure}

\subsection{A more complex model: SEIRD}
\label{sec:num-SEIRD}

To illustrate the flexibility and scalibility of the proposed numerical method, we now consider a more complex model. On top of the $S$, $I$ and $R$ states considered above, we add two new states: Exposed ($E$) and Dead ($D$). An individual is in state $E$ when it has been infected but is not yet infectious. Hence the agents evolve from $S$ to $E$ and then $I$, and the infection rate from $S$ to $E$ depends on the proportion of the infected people. From the point of view of the dynamics, the state $D$ is absorbing but it is important for the cost functions discussed below. Now $R$ is interpreted as recovered. We consider the states in the order: $S, E, I, R, D$. A representative agent evolves according to the rate matrix $Q(t,\alpha_t,\rho_t)$,
\begin{equation}
\label{eq:SEIRD-Qmatrix}
    Q(t,\alpha,\rho)
    =
    \begin{bmatrix}
    \cdots & \beta\alpha_t\int_A a\rho_t(da,I) & 0 & 0 & 0 %
    \\
    0 & \cdots & \epsilon & 0 & 0 
    \\
    0 & 0 & \cdots & \gamma & \delta 
    \\
    \eta  & 0 & 0 & \cdots & 0 
    \\
    0  & 0 & 0 & 0 & \cdots 
    \end{bmatrix},
\end{equation}
where $\beta,\gamma,\eta,\Lambda,\delta  \in \mathbb{R}_+$ are constants. See Fig.~\ref{fig:SEIRD-diagram} for a diagram of the dynamics.

The cost of the agents generalizes the previous example~\eqref{eq:cost-agents-SIR1-intro} by incorporating terms related to the new states as follows:

\begin{equation}
\label{eq:running-cost-in-SEIRD}
\begin{aligned}
    f(t,x,\alpha,\rho;&\lambda) 
    = 
    \frac{c_\lambda}{2}\left(\lambda^{(S)} - \alpha\right)^2\mathbbm{1}_{S}(x) + \frac{1}{2}\left(\lambda^{(E)} - \alpha\right)^2\mathbbm{1}_{E}(x) +\\ &\left(\frac{1}{2}\left(\lambda^{(I)} - \alpha\right)^2 + c_I\right)\mathbbm{1}_{I}(x)
     + \frac{1}{2}\left(\lambda^{(R)}-\alpha\right)^2\mathbbm{1}_{R}(x) +c_{D}\mathbbm{1}_{D}(x),
\end{aligned}
\end{equation}
where $c_\lambda,c_I, c_{D} \in\mathbb{R}_+$ are constants. We note that the final term in \eqref{eq:running-cost-in-SEIRD} represents a cost of passing due to the disease (transitioning to the absorbing state $D$) and a preference for doing so as late as possible. The terminal payment utility is $U(\xi)=\xi$. Further, we also modify the cost of the regulator:

\begin{equation}
    c_0(t,p) = c_{inf}p(I)^2,\quad f_0(t,\lambda) = \sum_{i\in\{S,E,I,R\}}\frac{\bar{\beta}^{(i)}}{2}\left(\lambda^{(i)} - \bar{\lambda}^{(i)}\right)^2, \quad C_0(p) = c_{d} p(D),
\end{equation}
for constant $\bar{\lambda},\bar{\beta} \in \mathbb{R}^m_+$ and $c_{\inf}, c_{d}>0$. Compared to previous experiments, the regulator is now paying an additional terminal cost $C_0(p)$ depending on the proportion of deceased people at the end of the time horizon.
We set the coefficients $c_d$ and $c_D$ to high values to put more importance to the cost terms related to death. Further we assumed that mortality rate $\delta$ is 1\%.

As it can be seen in the top plots in Figure \ref{fig:SEIRD}, playing the mean field Nash Equilibrium under the control of regulator flattens the curve of infections compared to free spread. We also see that when susceptible players do not feel safe enough with the recommended socialization levels they use a lower contact rate. This showcases the ability of the model to capture a population that is more risk-averse than their regulator. On a final note, this experiment shows that agents may not follow exactly what regulator proposes; therefore, the regulator should not assume during policy optimization that the population will strictly obey the announced policy.

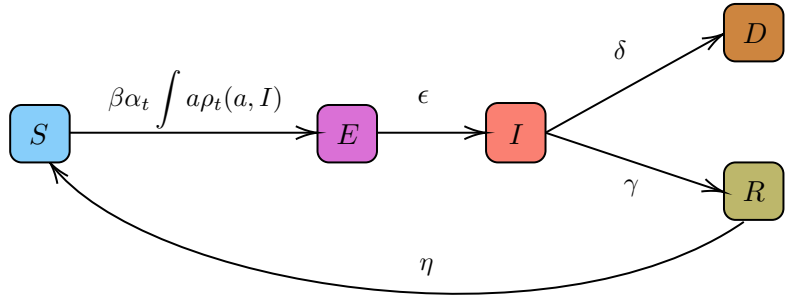
\begin{figure}
\begin{center}

\tikzset{every picture/.style={line width=0.75pt}} 

\begin{tikzpicture}[x=0.75pt,y=0.75pt,yscale=-1,xscale=1]

\draw  [fill={rgb, 255:red, 218; green, 112; blue, 214 }  ,fill opacity=1 ] (155,126.8) .. controls (155,123.6) and (157.6,121) .. (160.8,121) -- (179.2,121) .. controls (182.4,121) and (185,123.6) .. (185,126.8) -- (185,144.2) .. controls (185,147.4) and (182.4,150) .. (179.2,150) -- (160.8,150) .. controls (157.6,150) and (155,147.4) .. (155,144.2) -- cycle ;
\draw  [fill={rgb, 255:red, 250; green, 128; blue, 114 }  ,fill opacity=1 ] (240,126) .. controls (240,122.69) and (242.69,120) .. (246,120) -- (264,120) .. controls (267.31,120) and (270,122.69) .. (270,126) -- (270,144) .. controls (270,147.31) and (267.31,150) .. (264,150) -- (246,150) .. controls (242.69,150) and (240,147.31) .. (240,144) -- cycle ;
\draw  [fill={rgb, 255:red, 189; green, 183; blue, 107 }  ,fill opacity=1 ] (360,155.8) .. controls (360,152.6) and (362.6,150) .. (365.8,150) -- (384.2,150) .. controls (387.4,150) and (390,152.6) .. (390,155.8) -- (390,173.2) .. controls (390,176.4) and (387.4,179) .. (384.2,179) -- (365.8,179) .. controls (362.6,179) and (360,176.4) .. (360,173.2) -- cycle ;
\draw    (185,135) -- (238,135) ;
\draw [shift={(240,135)}, rotate = 180] [color={rgb, 255:red, 0; green, 0; blue, 0 }  ][line width=0.75]    (10.93,-3.29) .. controls (6.95,-1.4) and (3.31,-0.3) .. (0,0) .. controls (3.31,0.3) and (6.95,1.4) .. (10.93,3.29)   ;
\draw    (270,135) -- (358.25,85.97) ;
\draw [shift={(360,85)}, rotate = 510.95] [color={rgb, 255:red, 0; green, 0; blue, 0 }  ][line width=0.75]    (10.93,-3.29) .. controls (6.95,-1.4) and (3.31,-0.3) .. (0,0) .. controls (3.31,0.3) and (6.95,1.4) .. (10.93,3.29)   ;
\draw    (370,180) .. controls (273.98,246.67) and (67.58,212.35) .. (20.69,150.93) ;
\draw [shift={(20,150)}, rotate = 413.73] [color={rgb, 255:red, 0; green, 0; blue, 0 }  ][line width=0.75]    (10.93,-3.29) .. controls (6.95,-1.4) and (3.31,-0.3) .. (0,0) .. controls (3.31,0.3) and (6.95,1.4) .. (10.93,3.29)   ;
\draw  [fill={rgb, 255:red, 205; green, 133; blue, 63 }  ,fill opacity=1 ] (360,75.8) .. controls (360,72.6) and (362.6,70) .. (365.8,70) -- (384.2,70) .. controls (387.4,70) and (390,72.6) .. (390,75.8) -- (390,93.2) .. controls (390,96.4) and (387.4,99) .. (384.2,99) -- (365.8,99) .. controls (362.6,99) and (360,96.4) .. (360,93.2) -- cycle ;
\draw    (270,135) -- (358.1,164.37) ;
\draw [shift={(360,165)}, rotate = 198.43] [color={rgb, 255:red, 0; green, 0; blue, 0 }  ][line width=0.75]    (10.93,-3.29) .. controls (6.95,-1.4) and (3.31,-0.3) .. (0,0) .. controls (3.31,0.3) and (6.95,1.4) .. (10.93,3.29)   ;
\draw  [fill={rgb, 255:red, 135; green, 206; blue, 250 }  ,fill opacity=1 ] (0,126.8) .. controls (0,123.6) and (2.6,121) .. (5.8,121) -- (24.2,121) .. controls (27.4,121) and (30,123.6) .. (30,126.8) -- (30,144.2) .. controls (30,147.4) and (27.4,150) .. (24.2,150) -- (5.8,150) .. controls (2.6,150) and (0,147.4) .. (0,144.2) -- cycle ;
\draw    (30,135) -- (153,135) ;
\draw [shift={(155,135)}, rotate = 180] [color={rgb, 255:red, 0; green, 0; blue, 0 }  ][line width=0.75]    (10.93,-3.29) .. controls (6.95,-1.4) and (3.31,-0.3) .. (0,0) .. controls (3.31,0.3) and (6.95,1.4) .. (10.93,3.29)   ;

\draw (163,129) node [anchor=north west][inner sep=0.75pt]   [align=left] {$\displaystyle E$};
\draw (250,129) node [anchor=north west][inner sep=0.75pt]   [align=left] {$\displaystyle I$};
\draw (368,158) node [anchor=north west][inner sep=0.75pt]   [align=left] {$\displaystyle R$};
\draw (48,102) node [anchor=north west][inner sep=0.75pt]  [font=\small] [align=left] {$\displaystyle \beta \alpha _{t}\int a\rho _{t}( a,I)$};
\draw (308,157) node [anchor=north west][inner sep=0.75pt]  [font=\small] [align=left] {$\displaystyle \gamma $};
\draw (205,197) node [anchor=north west][inner sep=0.75pt]  [font=\small] [align=left] {$\displaystyle \eta$};
\draw (368,78) node [anchor=north west][inner sep=0.75pt]   [align=left] {$\displaystyle D$};
\draw (8,129) node [anchor=north west][inner sep=0.75pt]   [align=left] {$\displaystyle S$};
\draw (303,87) node [anchor=north west][inner sep=0.75pt]  [font=\small] [align=left] {$\displaystyle \delta $};
\draw (204,112) node [anchor=north west][inner sep=0.75pt]   [align=left] {$\displaystyle \epsilon $};

\end{tikzpicture}

\end{center}
\caption{SEIRD model corresponding to the $Q$-matrix~\eqref{eq:SEIRD-Qmatrix}.}
\label{fig:SEIRD-diagram}
\end{figure}

\begin{figure}[h]
\centering
\begin{subfigure}{.48\textwidth}
    \includegraphics[width=0.85\linewidth]{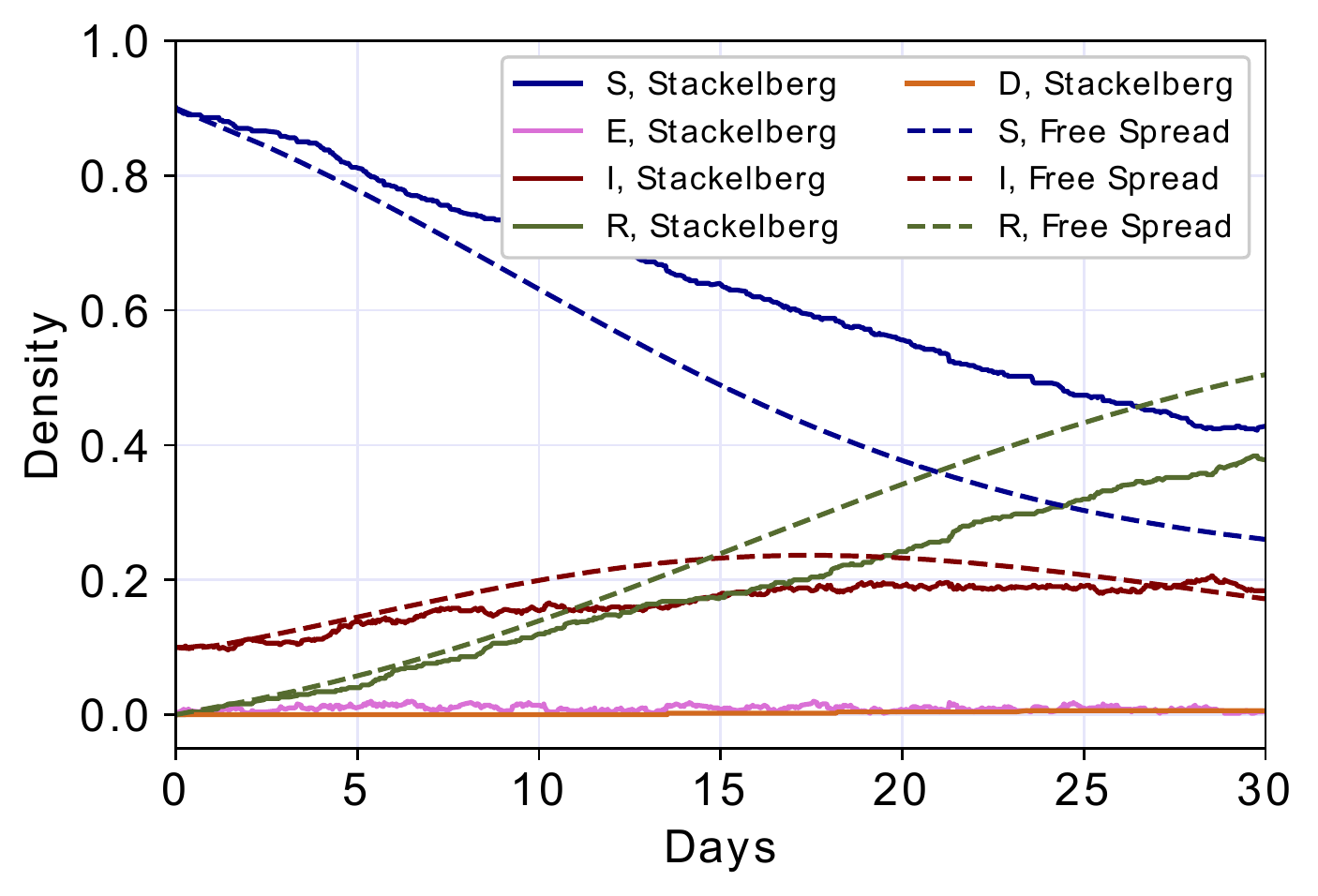}
\end{subfigure}
\begin{subfigure}{.48\textwidth}
    \includegraphics[width=0.85\linewidth]{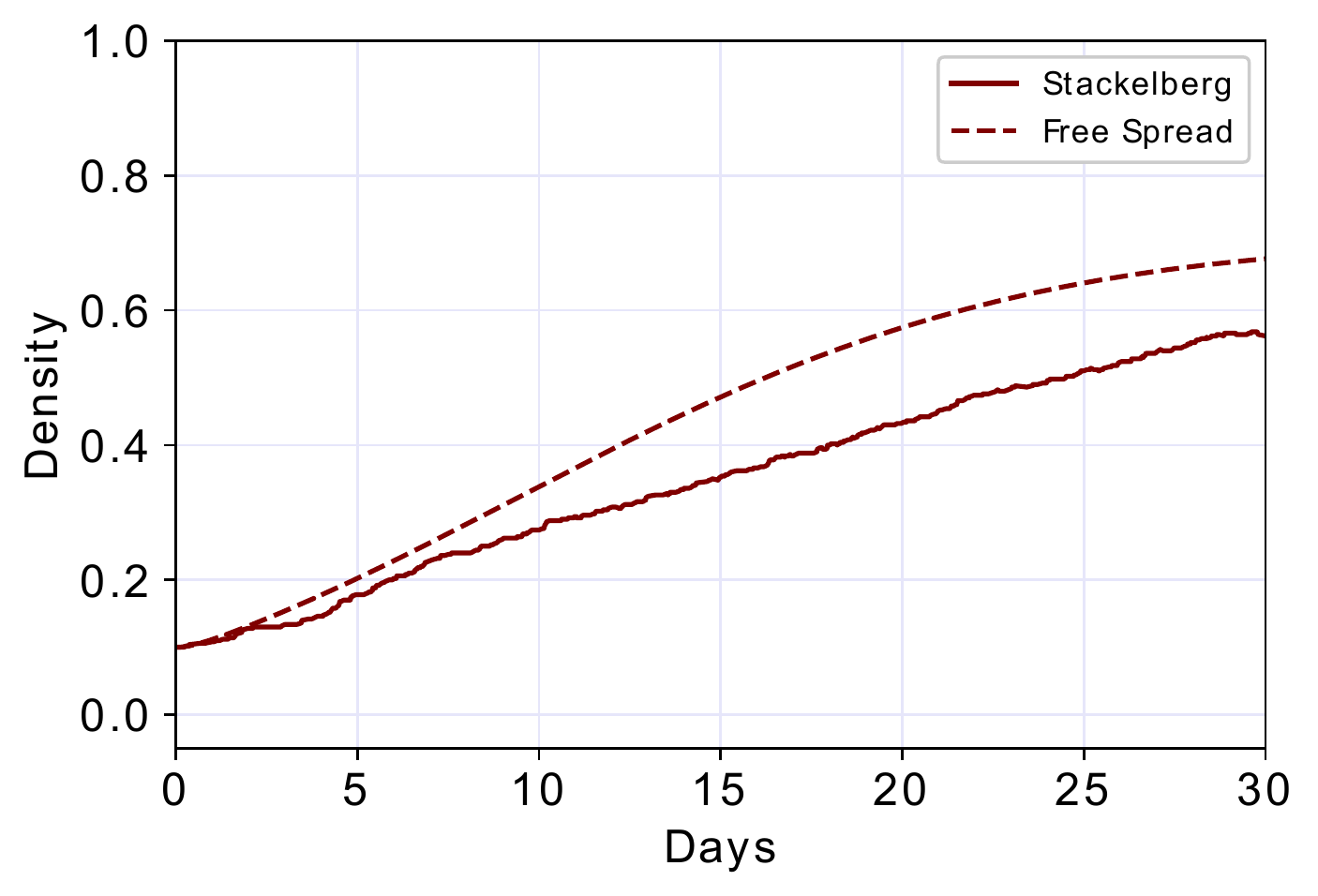}
\end{subfigure}
\begin{subfigure}{.48\textwidth}
    \includegraphics[width=0.9\linewidth]{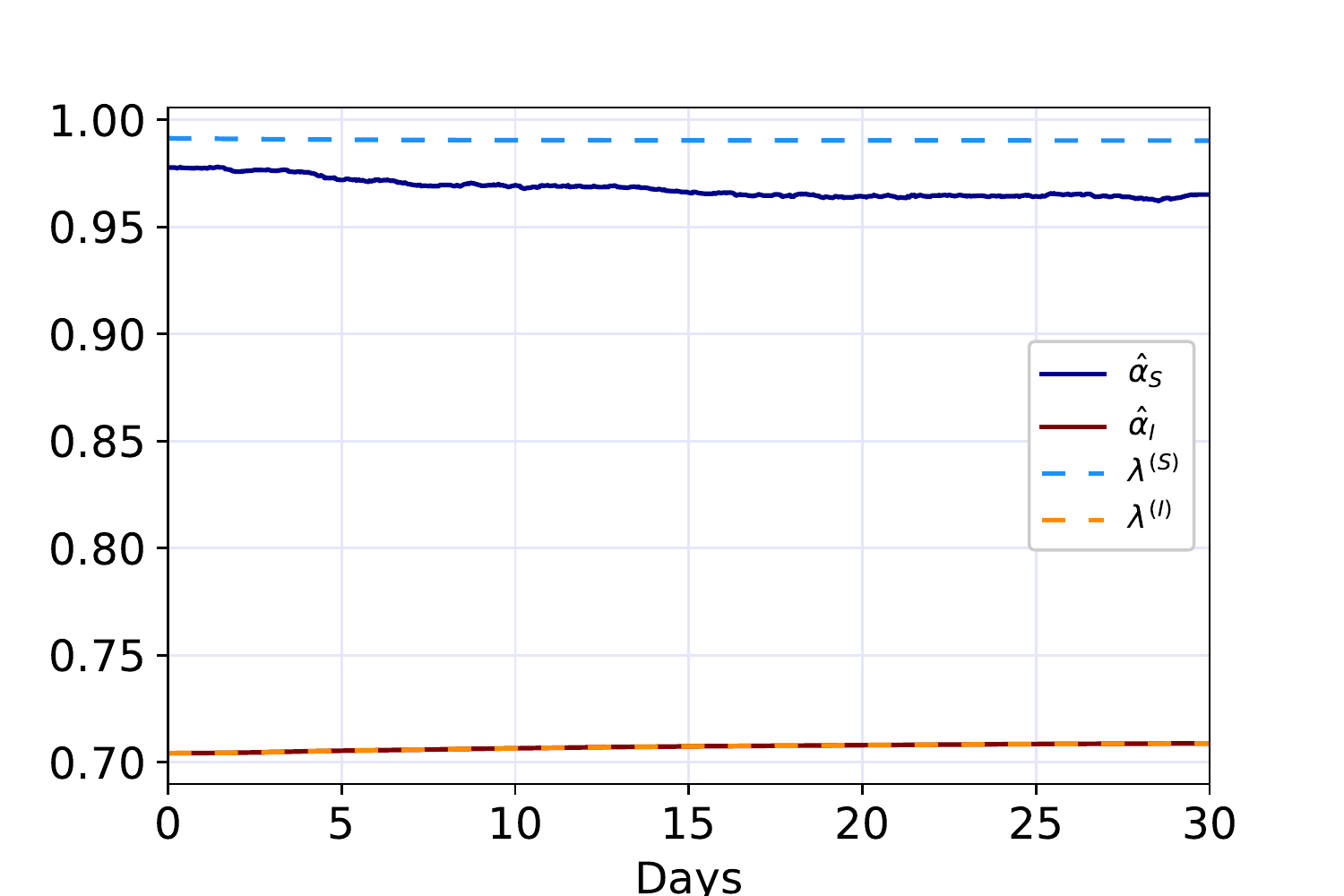}
\end{subfigure}
\begin{subfigure}{.48\textwidth}
    \includegraphics[width=0.9\linewidth]{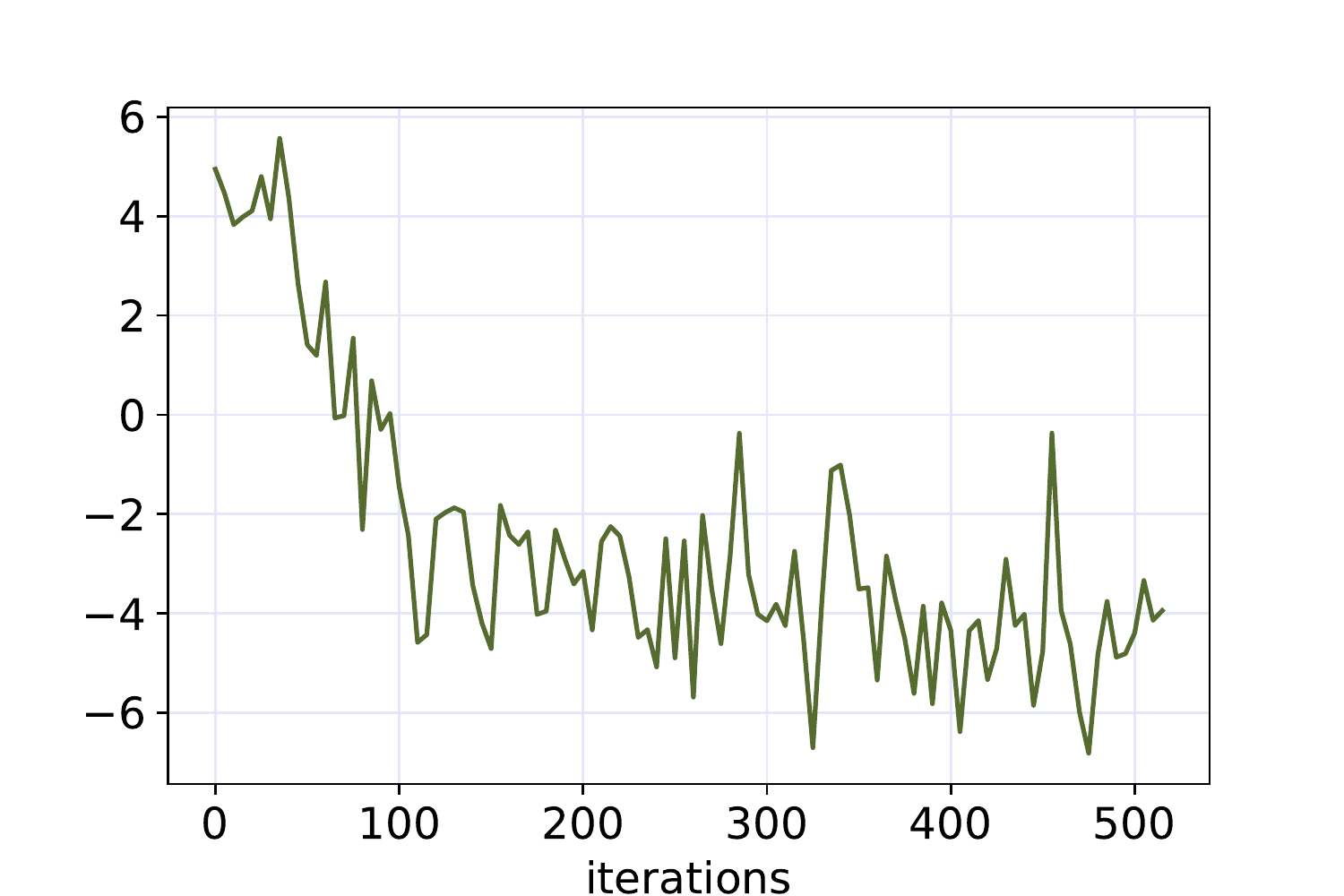}
\end{subfigure}

\caption{SEIRD Stackelberg MFG with Algorithm~\ref{algo:SGD-SMFG} in comparison with free spread SEIRD dynamics. Comparison of the Evolution of the state distribution (top left), Comparison of the Cumulative Density of Infected people under Stackelberg MFG and Free Spread (top right); evolution of the controls (bottom left), convergence of the loss (bottom right).}
\label{fig:SEIRD}
\end{figure}

\begin{table}[h!]
\caption{Parameter values for the SEIRD Stackelberg MFG experiment.}
\centering
\ra{1.4}
\begin{tabular}{@{}cccccccccccc@{}}
\toprule
 $T$ & $p^0$ & $c_\lambda$ & $c_I$ & $c_{\mathrm{Inf}}$ & $\bar\beta$ & $\bar\lambda$
\\
\midrule
$30$ & $(0.9,0,0.1,0,0)$ & $10$ & $1$ & $1$ & $(0.2,0.2,1,0)$ & $(1,1,0.7,1)$
\\
\toprule
$\beta$ & $\gamma$ & $\eta$ & $\epsilon$ & $\delta$ & $c_d$ &$c_D$
\\
\midrule
$0.25$ & $0.1$ & $0.01$ & $2$ & $0.01$ & $20$ & $20$
\\
\bottomrule
\end{tabular}
\label{tab:third-exp}
\end{table}

\section{Conclusions}
\label{sec:conclusions}

We have studied a Stackelberg mean field game with finite state space using a probabilistic approach. Compared with deterministic approaches using systems of ODEs, this approach has the advantage of describing the evolution of the system from the point of view of a typical (infinitesimal) agent. The theoretical contributions of this paper relate to formulating the Stackelberg game between a principal and a non-cooperative population. To cover the case of populations given by extended MFGs we establish the results in Section~\ref{sec:general-analysis}. We have then applied this class of models to a problem of epidemic containment with and SIR-type dynamics. In contrast with the existing literature, our model incorporates at the same time a non-cooperative population and a regulator such as a government. This problem, which can be viewed as a control problem under a constraint given by a Nash equilibrium, is complex because the regulator's decisions influence only indirectly the population's equilibrium. A naive approach would have been to solve a Nash equilibrium for each choice of the regulator's policy but this is computationally prohibitive for our model. Thus, building on the probabilistic approach, we introduced a numerical method based on neural network approximation and Monte Carlo simulations to compute the optimal policy (see Algorithm~\ref{algo:simu-forward-XY} and \ref{algo:SGD-SMFG}). We presented several numerical examples and for some of them we managed to derive semi-explicit solutions which can be used as benchmarks (see Proposition~\ref{prop:SIR-example}).   In particular, we have shown the difference between the uncontrolled scenario, the Nash equilibrium without regulator's intervention, and equilibria arising from early or late lockdowns. We have also shown that the numerical scheme can approximately learn the regulator's optimal policy, including a non-trivial model in Section~\ref{sec:num-SEIRD}.  Overall, the numerical approach is able to capture well the evolution of the epidemic in society and provides a satisfactory approximation of the optimal policy of the regulator in presence of a large number of non-cooperative agents.

Several directions are left for future work. For example, on a theoretical side, we may analyze Stackelberg mean field games with interaction through the joint state-action distribution under weaker assumptions. From an applied viewpoint, we may consider more complex finite-state models for instance to add a regulation aspect to SIR-like models appeared recently in the epidemiological literature. We believe that the derivation of the semi-explicit solutions can be generalized to more evolved compartmental models. On the numerical side, the algorithm we proposed is based on approximation by neural networks and could potentially handle even more complex models. This aspect is important for applications to epidemiological models. Some realistic features we could consider in future work are for instance age structure or geography with a network of cities. In both cases, the population is split into more sub-groups, which increases the number of possible states. Another important point is the impact of testing on the ability to take optimal decisions. Indeed, testing is directly related to the uncertainty of the regulator on the population distribution (e.g., to know what is the current proportion of infected people). For this aspect, it seems that a probabilistic approach like the one we adopted is particularly well-suited.


\appendix

\section{Proofs for Section~\ref{sec:main}}
\label{sec:proof-prop-reduction}

Propositions~\ref{prop:connection-Nash-BSDE} and \ref{prop:rewriting} are adaptations of \cite[Thm. 1]{Carmona2018Contract} and \cite[Thm. 2]{Carmona2018Contract}, respectively, to the extended case, \textit{i.e.}, to the case where the players interact through the joint distribution of the states and the actions and not only through the distribution of the states. For the sake of brevity we omit the proofs.

\begin{proof}[Proof of Proposition~\ref{prop:reduction}]
We define the representative agent's Hamiltonian in the regular mean-field game $h(t,x,z,\alpha,p)$ in line with $H$, but with the "overlined" functions $\bar{f},\bar{Q}$ replacing $f$ and $Q$. It follows from the assumptions that $\bar{a}_i(t,z,p)$ is the minimizer of  $\alpha \mapsto h(t,e_i,z,\alpha,p)$ for $(t,i,z,p)\in[0,T]\times \{1,\dots, m\}\times \mathbb{R}^m\times\mathcal{P}(E)$. Let $\bar{H}(t,e_i,z,p) := h(t,e_i,z,\bar{a}_i(t,z,p),p)$.

By Proposition~\ref{prop:reduction}, the pair $(\boldsymbol{\hat{\alpha}},\boldsymbol{\hat{\rho}})$ is a mean-field Nash equilibrium. Turning to non-extended mean-field Nash equilibrium, there exists a pair $(\boldsymbol{\alpha', p'})$ satisfying Definition~\ref{def:weak-mfg-equilibrium}, see for example \cite[Thm 4.1]{Carmona2018Extended}. Consequently, by  \cite[Thm. 1]{Carmona2018Contract}, there exists a solution $(\boldsymbol{Y', Z', \alpha', p', \mathbb{Q}'})$ to the system (under $\mathbb{P}$)
\begin{equation}
\label{eq:MKV-basde2}
\left\{
\begin{aligned}
    Y'_t 
    &= 
    U(\xi)
    + \int_t^T \bar{H}(s,X_{s-}, Z'_s, p'_s)ds - \int_t^T (Z'_s)^*d\mathcal{M}_s
    \\
    \mathcal{E}'_t 
    &= 
    1 + \int_0^t \mathcal{E}'_{s-}X^*_{s-}\left(\bar{Q}(s,\alpha'_s, p'_s) - Q^0\right)\psi^+_sd\mathcal{M}_s,
    \\
    \alpha'_t
    &= 
    \bar{a}(t,X_{t-},Z'_t, p'_t),\ \
    p'_t = \mathbb{Q}'\circ\left(X_t \right)^{-1},\ \
    \frac{d\mathbb{Q}'}{d\mathbb{P}} = \mathcal{E}'_T,
    \end{aligned}
    \right.
\end{equation}
such that $\boldsymbol{\alpha'} = \boldsymbol{\bar{\alpha}}$ $d\mathbb{P}\otimes dt$-a.e. and $p'_t = \bar{p}_t$ $dt$-a.e.. 

The claim $\hat{p}_t = \bar{p}_t$ $dt$-a.e. $t\in[0,T]$ is equivalent to
    \begin{equation}
    \label{eq:thm-proof-1}
        \mathbb{E}^{\hat{\mathbb{Q}}}\left[X_t\right]-\mathbb{E}^{\mathbb{Q}'}\left[X_t\right] = \mathbb{E}^{\mathbb{P}}[(\hat{\mathcal{E}}_t - \mathcal{E}'_t)X_t] = 0,\quad dt\text{-a.e. } t\in[0,T].
    \end{equation}
Let $\hat\sigma_t$ and $\sigma_t'$ be the volatility (viewed as row vectors in $\mathbb{R}^m$) of $\hat{\mathcal{E}}_t$ and $\mathcal{E}'_t$, respectively,
\begin{equation}
    \hat \sigma_t = 
    \hat{\mathcal{E}}_{t-}X^*_{t-} \left(\bar{Q}(t,\hat{\alpha}_t,\hat{p}_t) - Q^0\right) \psi^+_t,
    \quad
    \sigma_t' 
    = 
    \mathcal{E}'_{t-}X^*_{t-} \left(\bar{Q}(t,\alpha'_t,p'_t) - Q^0\right) \psi^+_t, 
\end{equation}
and let $\Delta \hat{\mathcal{E}}_t := \hat{\mathcal{E}}_t - \hat{\mathcal{E}}_{t-}$. In the same way we define $\Delta X_t$ and $\Delta \mathcal{E}'_t$. Since
\begin{equation}
    \Delta\hat{\mathcal{E}}_t = \hat{\sigma}_t\Delta X_t,\quad \Delta\mathcal{E}'_t = \sigma'_t\Delta X_t,
\end{equation}
we have by Ito's formula that under $\mathbb{P}$,
\begin{equation}
\begin{aligned}
        &d\|(\hat{\mathcal{E}}_t-\mathcal{E}'_t) X_t\|^2 \\
        &=
        2\left(\hat{\mathcal{E}}_{t-} - \mathcal{E}'_{t-}\right)^2X_{t-}^*Q^0X_{t-}dt 
        + 2\left(\hat{\mathcal{E}}_{t-} - \mathcal{E}'_{t-}\right)X_{t-}^*\psi_t\left(\hat{\sigma}_t - \sigma'_t\right)^*dt
        \\
        &\quad +
        \text{Tr}\left[ \left(\Sigma_t
        + (\hat{\mathcal{E}}_t - \mathcal{E}'_t)I_{m}\right)\psi_t \left(\Sigma_t +  (\hat{\mathcal{E}}_t - \mathcal{E}'_t)I_{m}\right)^* \right]dt
        +
        d \widetilde{\mathcal{M}}_t,
\end{aligned}
\end{equation}
where $\boldsymbol{\widetilde{\mathcal{M}}}$ is a $\mathbb{P}$-martingale, $I_m$ is the identity matrix in $\mathbb{R}^{m\times m}$ and
\begin{equation}
    \mathbb{R}^{m\times m} \ni \Sigma_t := 
    [X_t(i)\left(\hat\sigma_t(j)-\sigma'_t(j)\right)]_{ij}
\end{equation}
Firstly, we have that $\text{Tr}[\psi_t] = \sum_{i=1}^m [Q^0]_{X_{t-},i} = 0$, secondly, 
\begin{equation}
    \text{Tr}[\Sigma_t \psi_t] = \text{Tr}[\psi_t\Sigma_t^*] = 
    (m-1)\left(\hat{\sigma}_t - \sigma'_t\right) X_{t-},
\end{equation} 
and finally, $\text{Tr}[\Sigma_t\psi_t\Sigma^*_t] = (m-1)\left((\hat{\sigma}_t-\sigma'_t)X_{t-}\right)^2$. We gather that
\begin{equation}
\begin{aligned}
        &\text{Tr}\left[ \left(\Sigma_t 
        + (\hat{\mathcal{E}}_t - \mathcal{E}'_t)I_{m\times m}\right)\psi_t \left(\Sigma_t +  (\hat{\mathcal{E}}_t - \mathcal{E}'_t)I_{m\times m}\right)^* \right]
        \\
        &=
        2(m-1)(\hat{\sigma}_t-\sigma'_t)( \hat{\mathcal{E}}_{t-}-\mathcal{E}'_{t-}) X_{t-} + (m-1)\left((\hat{\sigma}_t-\sigma'_t)X_{t-}\right)^2.
\end{aligned}
\end{equation}
Using our calculations above, Young's inequality, Gronwall's lemma, and the Lipschitz continuity of $\bar{Q}$ and $\bar{a}$ (see Hypotheses~\ref{hyp:q} and~\ref{hyp:alpha-lip-in-z}), we get that for $dt$-a.e. $t\in[0,T]$,
    \begin{equation}
    \begin{aligned}
           \mathbb{E}^{\mathbb{P}}\left[\|(\hat{\mathcal{E}}_t-\mathcal{E}'_t)X_t\|^2\right]  
            &\leq
            C\mathbb{E}^{\mathbb{P}}\left[\int_0^t |\bar{a}(s,X_{s-},\hat{Z}_s,\hat{p}_s) - \bar{a}(s,X_{s-},Z'_s,p'_s)|^2ds \right]
            \\
            & \leq
            C\mathbb{E}^{\mathbb{P}}\left[\int_0^t
            \left(\|\hat{Z}_s- Z'_s\|_{X_{s-}}^2 + \|\hat{p}_s-p'_s\|^2\right)ds \right]
    \end{aligned}
    \end{equation}
    for some positive constant $C$.
    The estimate holds only $dt$-a.s. since 
    \begin{equation}
        (\hat{\mathcal{E}}_{t-}-\mathcal{E}'_{t-})X_{t-} = (\hat{\mathcal{E}}_t-\mathcal{E}'_t)X_t,
    \end{equation}
    an equality we need in order to apply Gronwall's lemma, holds only $dt\text{-a.e. }t\in[0,T]$.
    The constant $C$ depends on $\mathbb{E}^{\mathbb{P}}\left[\int_0^t \mathcal{E}^2_{t-}dt\right], \mathcal{E} \in \{ \hat{\mathcal{E}}, \mathcal{E}' \}$, which is bounded since $Q, \bar{Q}$, $Q^0$, and $X_{t-}$ are bounded. After one more use of Gronwall's lemma, we arrive to
    \begin{equation}
        \mathbb{E}^{\mathbb{P}}\left[\|(\hat{\mathcal{E}}_t-\mathcal{E}'_t)X_t\|^2\right]  \leq C\mathbb{E}^{\mathbb{P}}\left[\int_0^t \|\hat{Z}_s - Z'_s\|^2_{X_{s-}} ds \right],\qquad dt\text{-a.e. } t\in[0,T].
    \end{equation}
    Consider the difference process $(\boldsymbol{\delta Y, \delta Z}) := (\boldsymbol{\hat{Y}-Y', \hat{Z}-Z'})$. The BSDE estimate of \cite[Thm. 4.2.3]{zhang2017backward} yields the second inequality below, the first one follows by the equivalence of norms on $\mathbb{R}^m$:
    \begin{equation}
    \mathbb{E}^{\mathbb{P}}\left[\int_0^t \|\hat{Z}_s - \bar{Z}_s\|^2_{X_{s-}} ds \right]
    \leq 
    C\mathbb{E}^{\mathbb{P}}\left[
    \int_0^T\|\delta Z_t\|^2dt \right] 
    \leq 
    C
    \int_0^T\|\hat{p}_t-p'_t\|^2dt.
    \end{equation}
After one final application of Gronwall's lemma we see that $\hat{p}_t = p'_t$ for $dt$-a.e. $t\in [0,T]$ and therefore $\hat{p}_t = \bar{p}_t$ for $dt$-a.e. $t\in[0,T]$.

Finally, we compare the controls and get
\begin{equation}
    |\hat{\alpha}_t - \bar{\alpha}_t|
    \leq
    C\left(\|\hat{Z}_t-Z'_t\|_{X_{t-}} + \|\hat{p}_t - p'_t\|\right),\quad  d\mathbb{P}\otimes dt\text{-a.s.}
\end{equation}
where we used Hypothesis~\ref{hyp:reduction} to link $\hat{a}$ and $\bar{a}$, and to exploit the Lipschitz continuity of $\bar{a}$. 
Using the previous calculations, we conclude that $\hat{\alpha}_t = \bar{\alpha}_t,\ d\mathbb{P}\otimes dt\text{-a.s.}$.
\end{proof}

\section{Proof for Section~\ref{subsec:experiments}}
\label{app:SIR-derivation}

\begin{proof}[Proof of Proposition~\ref{prop:SIR-example}]
In the setting of the example Hypotheses~\ref{hyp:q}--\ref{hyp:minimization-of-H-extended}, \ref{hyp:cost-in-tilde-V}, and \ref{hyp:reduction} hold true. The minimizers of the reduced Hamiltonians are
\begin{equation}
\label{eq:hat-alpha-SIR}
\begin{aligned}
    \hat{a}_S(t,z,\rho) &=  \lambda^{(S)}_t + \frac{\beta}{c_\lambda}\left(\int_A a \rho_t(da,I)\right)(z(S)-z(I)),
    \\
    \hat{a}_S(t,z,\rho) &= \lambda^{(I)}_t,
    \qquad
    \hat{a}_S(t,z,\rho) = \lambda^{(R)}_t.
\end{aligned}
\end{equation}
Imposing the consistency condition on \eqref{eq:hat-alpha-SIR} we see that the game satisfies hypotheses~\ref{hyp:reduction}(ii) and \ref{hyp:reduction}(iii). Thus, since we assume hypothesis~\ref{hyp:reduction}(i), Proposition~\ref{hyp:reduction} says that the mean field Nash equilibrium is almost surely equal to the solution of the regular mean-field game with transition rate matrix
\begin{equation}
    \bar{Q}(t,\alpha,p) = 
    \begin{bmatrix}
    \cdots & \beta\alpha\lambda^{(I)}_tp(I) & 0
    \\
    0 & \cdots & \gamma
    \\
    \eta  & 0 & \cdots
    \end{bmatrix},
\end{equation}
and the minimizers of the reduced Hamiltonians for this regular mean field game are
\begin{equation}
\label{eq:bar-alpha-SIR}
\begin{aligned}
    \bar{a}_S(t,z,\rho) &=  \lambda^{(S)}_t + \frac{\beta}{c_\lambda}\lambda^{(I)}_tp(I)(z(S)-z(I)),
    \\
    \bar{a}_S(t,z,\rho) &= \lambda^{(I)}_t,
    \qquad
    \bar{a}_S(t,z,\rho) = \lambda^{(R)}_t.
\end{aligned}
\end{equation}
The principal's problem, after the same rewriting that yielded \eqref{eq:Y-ZlambdaY0}, reads
\begin{equation}
    \begin{aligned}
        V(\kappa) = 
        &\inf_{\mathbb{E}[Y_0]\leq
        \kappa}\inf_{\substack{\boldsymbol Z \in \mathcal{H}^2_X \\ \boldsymbol\lambda \in \Lambda}}\mathbb{E}^{\mathbb{Q}^{\boldsymbol Z, \boldsymbol \lambda, Y_0}}\Bigg[
        \int_0^T \left( c_0(t,p_t^{\boldsymbol Z,\boldsymbol \lambda, Y_0}) + f_0(t,\lambda_t)\right)dt - Y_T^{\boldsymbol Z,\boldsymbol \lambda, Y_0}
        \Bigg]
        \\
        &= - \kappa + \inf_{\substack{\boldsymbol Z\in\mathcal{H}^2_X\\\boldsymbol \lambda \in \Lambda }}
        \mathbb{E}^{\mathbb{Q}^{\boldsymbol Z,\boldsymbol \lambda, Y_0}}\Bigg[ 
        \int_0^T \Bigg( c_0(t,p_t^{\boldsymbol Z,\boldsymbol \lambda, Y_0}) + f_0(t,\lambda_t)
        \\
        &\hspace{1cm}
        +\frac{c_\lambda}{2}\left(\lambda^{(S)}_t - \bar{a}_S(t,Z_t,p_t^{\boldsymbol Z,\boldsymbol \lambda, Y_0})\right)^2\mathbbm{1}_S(X_{t-}) + c_I\mathbbm{1}_I(X_{t-})\Bigg)dt
        \Bigg],
    \end{aligned}
\end{equation}
 where $-\kappa$ is achieved by $\mathbb{E}[Y_0]$.

The principal's problem $V(\kappa)$ can be recast as an optimization problem over $\mathbb{A}\times\Lambda$. Given $\boldsymbol\alpha\in\mathbb{A}$ let $\boldsymbol{\bar{Z}}$ be given by, for now formally,
\begin{equation}
    \bar{Z}_t(\boldsymbol\alpha) := \left(\frac{c_\lambda(\alpha_t - \lambda^{(S)}_t)}{\beta\lambda^{(I)}_tp_t^{\boldsymbol {\bar{Z}(\boldsymbol\alpha)},\boldsymbol \lambda, Y_0}(I)}\mathbbm{1}_{S}(X_{t-})\mathbbm{1}(\lambda^{(I)}_t>0),0,0\right),\quad t\in[0,T].
\end{equation}
\begin{lemma}
\label{prop:Z-to-alpha}
Given $\boldsymbol \alpha\in\mathbb{A}$ and $\boldsymbol\lambda\in \Lambda$, let 
\begin{equation}
    \Psi(\boldsymbol z) =
    \left(\frac{c_\lambda(\alpha_t - \lambda^{(S)}_t)}{\beta\lambda^{(I)}_t\Phi_t^{\boldsymbol z}(I)}\mathbbm{1}_{S}(X_{t-})\mathbbm{1}(\lambda^{(I)}_t>0), 0, 0\right)_{t\in [0,T]}
\end{equation}
for $\boldsymbol z \in \mathcal{H}^2_X$, where $\Phi_t^{\boldsymbol z}$ is given by
\begin{equation}
    \begin{aligned}
        \Phi_t^{\boldsymbol z} = p_0 + \int_0^t \mathbb{E}^{\mathbb{Q}^{\boldsymbol {z},\boldsymbol \lambda, Y_0}}\left[ \bar{Q}^*(s,\bar{\alpha}(s, X_{s-}, z_s, \Phi_s^{\boldsymbol z}), \Phi_s^{\boldsymbol z})X_{s-} \right]ds.
    \end{aligned}
\end{equation}
Then $\Psi$ is well defined as a mapping from $\mathcal{H}^2_X$ to itself and there exists a unique fixed point to $\Psi$. We denote the fixed point of $\Psi$ by $\boldsymbol{\bar Z}(\boldsymbol\alpha)$ and $\Phi_t^{\boldsymbol{\bar Z}(\boldsymbol\alpha)}$ by $p_t^{\boldsymbol{\bar Z}(\boldsymbol\alpha), \boldsymbol \lambda, Y_0}$. Furthermore, for all $t\in[0,T]$,
\begin{equation}
\label{eq:lemma-fixed-point-in-bar-a}
\bar{a}(t, X_{t-}, \bar{Z}_t(\boldsymbol\alpha), p^{\boldsymbol{\bar{Z}}(\boldsymbol \alpha),\boldsymbol\lambda, Y_0}_t)
= 
\alpha_t\mathbbm{1}_S(X_{t-})
+ \lambda^{(I)}_t\mathbbm{1}_{I}(X_{t-})
+ \lambda^{(R)}_t\mathbbm{1}_{R}(X_{t-}).
\end{equation}
\end{lemma}
\begin{proof} 
The first observation is there exists a uniform lower bound $c$ such that $\Phi^z_t(i) \geq c > 0$ for $i\in\{S,I,R\}$ and $t\in (0,T]$. This grants the first assertion of the lemma, that $\Psi(\boldsymbol z) \in \mathcal{H}^2_X$ for all $z\in \mathcal{H}^2_X$.  
Secondly, we have that
\begin{equation}
\label{eq:inverse-lipschitz}
    \left|\Phi^z_t(I)^{-1} - \Phi^{z'}_t(I)^{-1}\right|^2
    \leq C \mathbb{E}\left[\int_0^t \|z_s-z_s'\|^2 dt\right]
\end{equation}
for some positive constant $C$.
To get \eqref{eq:inverse-lipschitz} we used Gronwall's lemma, Cauchy-Schwartz inequality, and the boundedness of the coefficients, $\boldsymbol \lambda$, and the likelihood process $d\mathbb{Q}^{\boldsymbol z,\boldsymbol\lambda, Y_0}/d\mathbb{P}$. From \eqref{eq:inverse-lipschitz} follows that
\begin{equation}
    \mathbb{E}\left[\left\|\Psi_t(\boldsymbol z) - \Psi_t(\boldsymbol {z'}) \right\|^2 \right]
    \leq C \mathbb{E}\left[\int_0^t \|z_s-z_s'\|^2dt\right], \quad t\in[0,T].
\end{equation}
and hence $\Psi^N$ is a contraction mapping for sufficiently large $N$ (by equivalence of norms on $\mathbb{R}^m$).
The existence of a unique fixed point of $\Psi$ follows by the Banach fixed-point theorem for iterated mappings. Finally, \eqref{eq:lemma-fixed-point-in-bar-a} follows by plugging in $\boldsymbol{\bar{Z}}(\boldsymbol \alpha)$ and $\boldsymbol p^{\boldsymbol{\bar{Z}}\boldsymbol\alpha, \boldsymbol\lambda, Y_0}$ into \eqref{eq:bar-alpha-SIR}.
\end{proof}
In light of Lemma~\ref{prop:Z-to-alpha} and since $Q$ and  $\bar{\alpha}$ do not depend on the representative agent's expected total cost $Y_0$, the principal's problem can be transformed to
\begin{equation}
    \begin{aligned}
        W
        &= 
        \inf_{\boldsymbol\alpha\in\mathbb{A}, \boldsymbol{\lambda}\in\Lambda}
        I(\boldsymbol\alpha,\boldsymbol\lambda),
        \\
        I(\boldsymbol\alpha,\boldsymbol\lambda)
        &:=
        \mathbb{E}^{\mathbb{Q}^{\boldsymbol\alpha, \boldsymbol\lambda}}\Bigg[
        \int_0^T 
        \Big(c_0(t,p^{\boldsymbol\alpha,\boldsymbol\lambda}_t) + f_0(t,\lambda_t) 
        + \frac{c_\lambda}{2}\left(\lambda^{(S)}_t - \alpha_t\right)^2\mathbbm{1}_S(X_{t-}) + c_I\mathbbm{1}_I(X_{t-})\Big)dt
        \Bigg].
    \end{aligned}
\end{equation}
The measure $\mathbb{Q}^{\boldsymbol\alpha,\boldsymbol\lambda}$ is such that the coordinate process $\boldsymbol X$ has transition rate matrix $\bar{Q}(t,\alpha_t, p^{\boldsymbol\alpha, \boldsymbol\lambda})$ under it, $p^{\boldsymbol\alpha, \boldsymbol\lambda}_t$ is the law of $X_t$ under $\bar{Q}(t,\alpha_t, p^{\boldsymbol\alpha, \boldsymbol\lambda})$.

From the decomposition of the coordinate process,
\begin{equation}
    X_t = X_0 + \int_0^t \bar{Q}^*(s, \alpha_s, p^{\boldsymbol\alpha, \boldsymbol\lambda}_s)X_{s-}ds + \mathcal{M}^{\boldsymbol\alpha, \boldsymbol\lambda}_t,
\end{equation}
we deduce the dynamic of $\boldsymbol p^{\boldsymbol\alpha, \boldsymbol\lambda}$:

\begin{equation}
    p^{\boldsymbol\alpha, \boldsymbol\lambda}_t = p^0 + \int_0^t \mathbb{E}^{\mathbb{Q}^{\boldsymbol\alpha, \boldsymbol\lambda}}\left[\bar Q^*(s, \alpha_s, p^{\boldsymbol\alpha, \boldsymbol\lambda}_s)X_{s-} \right]ds,
\end{equation}
where we see that the right-hand side depends on the control $\alpha_s$ only through the conditional expectation $\widetilde\alpha_s := \mathbb{E}^{\mathbb{Q}^{\boldsymbol\alpha,\boldsymbol\lambda}}[\alpha_s\ |\ X_{s-} = S]$, $s\in[0,T]$. Let
\begin{equation}
\begin{aligned}
    \widetilde I(\boldsymbol{\widetilde\alpha},\boldsymbol\lambda)
    :=
    \int_0^T 
    \Big(c_0(t,p^{\boldsymbol\alpha,\boldsymbol\lambda}_t) + f_0(t,\lambda_t) 
    + \frac{c_\lambda}{2}\left(\lambda^{(S)}_t - \widetilde\alpha_t\right)^2\mathbbm{1}_S(X_{t-}) + c_I\mathbbm{1}_I(X_{t-})\Big)dt
    \end{aligned}
\end{equation}
and consider the deterministic control problem
\begin{equation}
    \widetilde W := 
    \inf_{\boldsymbol{\widetilde\alpha}\in \widetilde{\mathbb{A}}, \boldsymbol\lambda \in \Lambda}
    \widetilde{I} (\boldsymbol{\widetilde\alpha}, \boldsymbol\lambda)
\end{equation}
where $\widetilde{\mathbb{A}}$ is the collection of all measurable mappings from $[0,T]$ to $A$.
\begin{lemma}
We have $W = \widetilde W$. If $(\boldsymbol{\widetilde \alpha},\boldsymbol\lambda)$ is a solution to the optimization problem $\widetilde W$, then the predictable process $\boldsymbol\alpha$ defined by $\alpha_t = \sum_{i\in\{S,I,R\}} \mathbbm{1}_{i}(X_{t-})\widetilde{\alpha}_t^{(i)}$ together with $\boldsymbol\lambda$ is an optimal control for $W$.
Furthermore, an optimal control exists for $\widetilde W$.
\end{lemma}
The proof readily follows by Proposition 1 and Lemma 3 of \cite{Carmona2018Contract}.

We apply the necessary part of the Pontryagin maximum principle to characterizes the solution to the optimal control problem $\widetilde W$ and the corresponding flow of probability measures. The Hamiltonian for the problem $\widetilde W$ is $\widetilde H$, defined in \eqref{eq:Hamiltonian-tilde-W}.
It is straight forward to obtain the first order optimality conditions $(\nabla_{\widetilde\alpha}, \nabla_\lambda)\widetilde H = 0$. After some tedious calculation we reach \eqref{eq:example-optimal-controls}.
This concludes the proof of Proposition~\ref{prop:SIR-example}.
\end{proof}

\bibliographystyle{siam}

\end{document}